\documentclass[10pt, a4paper]{amsart}
\usepackage[margin = 1.3in]{geometry}
\usepackage[utf8]{inputenc}
\usepackage{amsmath, amssymb, amsfonts, amscd, amsthm}
\usepackage{enumerate}
\usepackage{mathtools}
\usepackage{xcolor}
\usepackage{graphics}
\usepackage{array}
\usepackage{adjustbox}
\usepackage{textcomp}

\usepackage{times}
\usepackage{setspace}
\usepackage{microtype}
\usepackage{marvosym}
\usepackage{stmaryrd} 
\usepackage{upgreek} 
\usepackage{centernot} 
\usepackage{comment}
\usepackage[shortlabels]{enumitem}

\usepackage{tikz-cd} 
\usepackage[all]{xy} 
\usepackage{latexsym, mdwlist}
\usepackage{subfig}
\usepackage{graphicx}
\usepackage{wrapfig}
\usepackage{url}
\usepackage{pgf, tikz}
\usetikzlibrary{graphs,decorations.pathmorphing,decorations.markings}
\usetikzlibrary{arrows,chains,shapes.geometric,%
    decorations.pathreplacing,decorations.pathmorphing,shapes,%
    matrix,shapes.symbols,patterns}
\makeatletter
\def\l@subsection{\@tocline{2}{0pt}{2.5pc}{5pc}{}}
\makeatother

\DeclareMathOperator{\Hom}{Hom}

\DeclareMathOperator{\id}{1}
\DeclareMathOperator{\Spec}{Spec}

\DeclareMathOperator{\tot}{tot}
\DeclareMathOperator{\Coh}{Coh}

\DeclareMathOperator{\Higgs}{Higgs}
\DeclareMathOperator{\Bun}{Bun}

\DeclareMathOperator{\Maps}{Maps}
\DeclareMathOperator{\anMaps}{anMaps}

\DeclareMathOperator{\ad}{ad}

\DeclareMathOperator{\Sym}{Sym}

\DeclareMathOperator{\Rep}{Rep}
\DeclareMathOperator{\QCoh}{QCoh}

\DeclareMathOperator{\dSt}{\mathbf{dSt}}
\DeclareMathOperator{\dAnSt}{\mathbf{dAnSt}}

\DeclareMathOperator{\Hodge}{Hodge}

\DeclareMathOperator{\Tw}{Tw}

\DeclareMathOperator{\Betti}{Betti}

\DeclareMathOperator{\Image}{Image}

\DeclareMathOperator{\Perf}{Perf}
\newcommand{\heart}{\ensuremath\heartsuit}
\DeclareMathOperator{\Dol}{Dol}
\DeclareMathOperator{\dR}{dR}
\DeclareMathOperator{\Hod}{Hod}
\DeclareMathOperator{\Sim}{Sim}

\DeclareMathOperator{\BBB}{BBB}
\DeclareMathOperator{\sst}{sst}

\DeclareMathOperator{\DMod}{D-Mod}

\DeclareMathOperator{\TwMod}{Tw-Mod}

\DeclareMathOperator{\Nilp}{Nilp}

\DeclareMathOperator{\BAA}{BAA}

\DeclareMathOperator{\RH}{RH}
\newcommand{\op}{\operatorname}
\DeclareMathOperator{\Shv}{Shv}

\DeclareMathOperator{\IndCoh}{IndCoh}

\DeclareMathOperator{\an}{an}
\DeclareMathOperator{\restr}{restr}

\DeclareMathOperator{\Eis}{Eis}

\DeclareMathOperator{\Rees}{Rees}

\DeclareMathOperator{\Sing}{Sing}
\DeclareMathOperator{\Supp}{Supp}

\DeclareMathOperator{\Twistor}{Twistor}

\DeclareMathOperator{\irred}{irred}
\DeclareMathOperator{\CT}{CT}

\DeclareMathOperator{\cusp}{cusp}

\DeclareMathOperator{\LocSys}{LocSys}

\DeclareMathOperator{\Horiz}{Horiz}

\DeclareMathOperator{\nah}{nah}

\DeclareMathOperator{\Ind}{Ind}
\DeclareMathOperator{\GAGA}{GAGA}

\DeclareMathOperator{\alg}{alg}
\DeclareMathOperator{\SingSupp}{SingSupp}
\DeclareMathOperator{\Wilson}{Wils}
\DeclareMathOperator{\Wils}{Wils}
\DeclareMathOperator{\Map}{Map}
\DeclareMathOperator{\anMap}{anMap}
\DeclareMathOperator{\Res}{Res}

\DeclareMathOperator{\coarse}{coarse}
\DeclareMathOperator{\red}{red}
\DeclareMathOperator{\colim}{colim}
\DeclareMathOperator{\Red}{Red}
\DeclareMathOperator{\HN}{HN}
\DeclareMathOperator{\Cone}{Cone}

\newcommand{\Aa}{\mathcal{A}}
\newcommand{\Bb}{\mathcal{B}}
\newcommand{\Cc}{\mathcal{C}}
\newcommand{\Dd}{\mathcal{D}}
\newcommand{\Ee}{\mathcal{E}}
\newcommand{\Ff}{\mathcal{F}}

\newcommand{\Mm}{\mathcal{M}}
\newcommand{\Nn}{\mathcal{N}}
\newcommand{\Oo}{\mathcal{O}}
\newcommand{\Pp}{\mathcal{P}}

\newcommand{\Tt}{\mathcal{T}}
\newcommand{\Uu}{\mathcal{U}}

\newcommand{\Ww}{\mathcal{W}}
\newcommand{\Xx}{\mathcal{X}}
\newcommand{\Yy}{\mathcal{Y}}
\newcommand{\Zz}{\mathcal{Z}}

\newcommand{\B}{\mathrm{B}}

\newcommand{\ol}[1]{\overline{#1}}

\renewcommand{\AA}{\mathbb{A}}

\newcommand{\CC}{\mathbb{C}}
\newcommand{\DD}{\mathbb{D}}

\newcommand{\GG}{\mathbb{G}}

\newcommand{\LL}{\mathbb{L}}

\newcommand{\PP}{\mathbb{P}}

\newcommand{\TT}{\mathbb{T}}

\newcommand{\WW}{\mathbb{W}}

\newcommand{\gG}{\mathfrak{g}}
\newcommand{\pP}{\mathfrak{p}}
\newcommand{\mM}{\mathfrak{m}}
\newcommand{\uU}{\mathfrak{u}}

\renewcommand{\to}{\longrightarrow}

\newcommand{\RomanNumeralCaps}[1]
    {\MakeUppercase{\romannumeral #1}}

\input{toc.tex}

\numberwithin{equation}{section}
\newtheorem*{theorem*}{Theorem}
\newtheorem*{definition*}{Definition}

\newtheorem{theorem}{Theorem}[section]
\newtheorem{lemma}[theorem]{Lemma}
\newtheorem{proposition}[theorem]{Proposition}
\newtheorem{corollary}[theorem]{Corollary}

\newtheorem{remark}[theorem]{Remark}
\theoremstyle{definition}
\newtheorem{definition}[theorem]{Definition}
\newtheorem{def/prop}[theorem]{Definition/Proposition}

\theoremstyle{remark}

\usepackage[backend=biber, sorting=nyt, style = alphabetic, maxbibnames=15]{biblatex}
\bibliography{biblio}

\usepackage{hyperref}\hypersetup{colorlinks}
\usepackage{color} 
\definecolor{darkred}{rgb}{1,0,0} 
\definecolor{darkgreen}{rgb}{0,1,0}
\definecolor{darkblue}{rgb}{0, 0, 1}
\definecolor{darkpurple}{RGB}{170, 51, 106}

\hypersetup{colorlinks,
linkcolor=darkpurple,
filecolor=darkgreen,
urlcolor=darkred,
citecolor=darkpurple}

\title[Geometric Eisenstein series in non-abelian Hodge theory and hyperholomorphic branes from supersymmetry]{Geometric Eisenstein series in non-abelian Hodge theory and hyperholomorphic branes from supersymmetry}

\begin{document}

\begin{abstract}
Using \textit{geometric Eisenstein series}, foundational work of Arinkin--Gaitsgory constructs cuspidal-Eisenstein decompositions for ind-coherent nilpotent sheaves on the de Rham moduli of local systems. This article extends these constructions to coherent (not ind-coherent) nilpotent sheaves on the Dolbeault, Hodge and twistor moduli from non-abelian Hodge theory. We thus account for Higgs bundles,  Hodge filtrations and hyperk\"ahler rotations of local systems. In particular, our constructions are shown to decompose a hyperholomorphic sheaf theory of so-called \textit{BBB-branes} into cuspidal and Eisenstein components. Our work is motivated, on the one hand, by the `classical limit' or `Dolbeault geometric Langlands conjecture' of Donagi--Pantev, and on the other, by attempts to interpret Kapustin--Witten's physical duality between BBB-branes and BAA-branes in 4D supersymmetric Yang--Mills theories as a mathematical statement within the geometric Langlands program. 
\end{abstract} 

\author[R. Hanson]{Robert Hanson}
\address{R. Hanson, Department of Mathematics, Imperial College, London SW7 2AZ, United Kingdom}
\email{robert.hanson@imperial.ac.uk}

\thanks{
\, \newline
Author is supported by the Horizon Europe MSCA grant \textit{Hyperkähler mirror symmetry and Langlands duality} (ID 101204490)}

\, 

\vspace{-10mm}

\maketitle

\vspace{-1mm}

{
\hypersetup{linkcolor=black}
\setcounter{tocdepth}{2}
\renewcommand{\baselinestretch}{1.30}\normalsize
{
\Small
\tableofcontents
}
\renewcommand{\baselinestretch}{1.00}\normalsize
}

\newpage

\section{Introduction}

\setlength{\parskip}{7pt}
\setlength{\parindent}{0pt}

\subsection{Outline}
\label{outline}
A fundamental construction in the geometric Langlands program is that of parabolic induction via \textit{geometric Eisenstein series}. The construction has incarnations in the de Rham, quantum and local theories \cite{arinkin&gaitsgory, braverman&gaitsgory, drinfeld&gaitsgory, hamann, hamann&hansen&scholze, laumon_automorphic, lysenko_eis}, where in each case, a core idea is to use geometric Eisenstein series to decompose a given sheaf theory into cuspidal and Eisenstein components, in a manner that emulates the spectral analysis of an automorphic form. The objective of this article is to extend the study of geometric Eisenstein series to sheaves on the Dolbeault, Hodge and twistor type moduli stacks from non-abelian Hodge theory, and in particular, to a hyperholomorphic sheaf theory of so-called \textit{BBB-branes}, defined in this article as a geometric Langlands analogue of the namesake hyperholomorphic boundary conditions from 4D supersymmetric Yang--Mills theories. 

\subsubsection{de Rham Cuspidal-Eisenstein decompositions}
\label{se: dR cusp-Eis intro}
Let us first recall how geometric Eisenstein series act on the spectral side of the de Rham geometric Langlands correspondence. Fix a reductive group $G$ and a parabolic subgroup $P \subset G$ with associated Levi quotient $P \to M$. The corresponding moduli stacks of $M$, $P$ and $G$-structured de Rham local systems, over a fixed smooth projective curve $X$, inherit a pair of structure group induction morphisms  
\[
\begin{tikzcd}
 & \LocSys_P \arrow[dl, "q^{\dR}"'] \arrow[dr, "p^{\dR}"] & \\
\LocSys_M & & \LocSys_G
\end{tikzcd} .
\]
A pull-push procedure then defines the spectral geometric Eisenstein series functor
\begin{equation}
\label{de Rham Eis intro}
\Eis^{\dR}_P := (p^{\dR})_{*} \circ (q^{\dR})^{!} : \IndCoh_{\Nn}(\LocSys_M) \to \IndCoh_{\Nn}(\LocSys_G) .
\end{equation}
Arinkin and Gaitsgory used $\Eis^{\dR}_P$ to establish the following fundamental structure theorem for the categories $\IndCoh_{\Nn}(\LocSys_G)$ of ind-coherent sheaves on $\LocSys_G$ with nilpotent singular support\footnote{Throughout this article we conflate the terms `sheaf with nilpotent singular support' and `nilpotent sheaf'.}.  
\begin{theorem}
\label{de Rham decompositions}
\cite[Thm 13.3.6]{arinkin&gaitsgory}. $\IndCoh_{\Nn}(\LocSys_{G})$ is generated\footnote{A dg category $\Dd$ is generated by a collection of subcategories $\{\Dd_i\}_{i \in I}$ if the smallest full dg subcategory containing every $\Dd_i$ is equivalent to $\Dd$ itself.} 
by the essential images of 
\[
\Eis^{\dR}_{P} : \QCoh(\LocSys_{M}) \to \IndCoh_{\Nn}(\LocSys_{G}),
\]
for all parabolic subgroups $P \subset G$. 
\end{theorem} 
By taking $\IndCoh_{\Nn}(\LocSys_{G})_{\Eis}$ to be the full subcategory of $\IndCoh_{\Nn}(\LocSys_{G})$ generated by proper parabolics, the right orthogonal can be identified as $\QCoh(\LocSys^{\irred}_{G})$ on the moduli $\LocSys^{\irred}_{G}$ of irreducible $G$-local systems - i.e. those without a reduction to a proper parabolic of $G$. It follows that $\QCoh(\LocSys^{\irred}_{G})$ is a localisation of $\IndCoh_{\Nn}(\LocSys_{G})$ with respect to $\IndCoh_{\Nn}(\LocSys_{G})_{\Eis}$, thus one has an exact sequence of dg categories 
\begin{equation}
\begin{tikzcd}
\label{cusp-Eis dR}
\IndCoh_{\Nn}(\LocSys_{G})_{\Eis} 
\arrow[r,shift left=2pt]
& \IndCoh_{\Nn}(\LocSys_{G}) 
\arrow[l,shift left=2pt]
\arrow[r,shift left=2pt] 
& \QCoh(\LocSys^{\irred}_{G})
\arrow[l,shift left=2pt] 
\end{tikzcd} .
\end{equation} 
This defines the (de Rham, spectral) \textit{cuspidal-Eisenstein decomposition} of $\IndCoh_{\Nn}(\LocSys_{G})$. In parallel with decompositions of the automorphic categories $\DMod(\Bun_{G^{\vee}})$, the Langlands functor 
\[
\LL^{\dR}_G : \DMod(\Bun_{G^{\vee}}) \xrightarrow{\cong} \IndCoh_{\Nn}(\LocSys_{G}) , 
\]
has been proven to decompose into a cuspidal component $\LL^{\dR, \cusp}_G$ and an Eisenstein component $\LL^{\dR, \Eis}_G$, where $\LL^{\dR, \Eis}_G$ is generated by the $\LL^{\dR}_{M}$'s. Within the pentalogy of papers \textit{GLC \RomanNumeralCaps{1} - GLC \RomanNumeralCaps{5}} \cite{GLC1, GLC2, GLC3, GLC4, GLC5} that constitutes the proof of the de Rham geometric Langlands conjecture, see in particular \textit{GLC \RomanNumeralCaps{3}} \cite{GLC3} for the main results pertaining to parabolic induction. 

\subsubsection{Non-abelian Hodge theory} The de Rham geometry of local systems is part of a family of geometries referred to collectively as \textit{non-abelian Hodge theory}. In a series of papers of Simpson \cite{simpson_secondary, simpson_algebraic, simpson_geometricity} these geometries are studied as Dolbeault, de Rham, Betti and Hodge moduli stacks that emulate their namesake cohomology theories. In the twistor theory of local systems introduced by Deligne and Simpson \cite{simpson_twistor}, these moduli are understood as deformation of complex structure, within a hyperk\"ahler family indexed by the twistor $\PP^1$. 
\[
\begin{tabular}{c|c|c|c}
 & $\PP^1$ coordinate & Moduli & Parameter space of 
\\ \hline
Dolbeault & $\lambda = 0$ & $\Higgs_G$ & \text{$G$-Higgs bundles} on $X$ \\
de Rham & $\lambda  = 1$ & $\LocSys_G$ & \text{$G$-local systems} on $X$ \\
Betti & $\lambda = 1$ & $\Rep_G$ & \text{Representations $\pi_1(X) \to G$} \\
Hodge & $\lambda \in \AA^1$ & $\Hodge_G$ & \text{$(G,\lambda)$-connections} on $X$ \\
Twistor & $\lambda \in \PP^1$ & $\Twistor_G$ & \text{$(G,\lambda)$-connections} on $X$ and $\ol{X}$   
\end{tabular}
\]
\subsubsection{Classical limit} The moduli stack $\Hodge_G \to \AA^1$ interpolates between the general fiber $\LocSys_G \cong \Hodge_G \times \{ \lambda\}$ and the central fiber $\Higgs_G = \Hodge_G \times_{\AA^1} \{ 0\}$ by linearising the Leibniz rule. This degeneration family was used by Donagi and Pantev \cite{donagi&pantev} to derive the \textit{classical limit} or \textit{Dolbeault geometric Langlands conjecture}, which may approximately\footnote{It is not currently known which sheaf theory is required for a precise formulation of the Dolbeault geometric Langlands conjecture.} be formulated as an equivalence of categories 
\[
\LL_G^{\Dol} : \IndCoh_{\Nn}(\Higgs_{G^{\vee}}) \xrightarrow{ \cong }  \IndCoh_{\Nn}(\Higgs_{G}) , 
\]
compatible with Dolbeault Hecke functors, Dolbeault Eisenstein functors and the pair of algebraic integrable systems $\Higgs_G \to \Aa_G \cong \Aa_{G^{\vee}} \longleftarrow \Higgs_{G^{\vee}}$, initially discovered by Hitchin \cite{hitchin_self} and extended by Donagi--Gaitsgory \cite{donagi&gaitsgory} and Donagi--Pantev \cite[Thm A]{donagi&pantev}. 

Following the philosophy of the classical limit, we consider a Hodge and Dolbeault variant of Eisenstein functor, denoted $\Eis_P^{\Hod}$ and $\Eis_P^{\Dol}$, such that the original de Rham functors can be recovered as $\Eis_P^{\dR} = \Eis_P^{\Hod} \times_{\AA^1} \{ 1 \}$, which forgets the data of a filtration, and $\Eis_P^{\Dol} = \Eis_P^{\Hod} \times_{\AA^1} \{ 0 \}$, which computes the associated graded. Both $\Eis_P^{\Dol}$ and $\Eis_P^{\Hod}$ can be constructed by an identical parabolic induction procedure as in the de Rham theoretic definition of \eqref{de Rham Eis intro}. By studying Dolbeault and Hodge theoretic extensions to the methods of Arinkin--Gaitsgory \cite{arinkin&gaitsgory}, we construct the corresponding cuspidal-Eisenstein decompositions for compact sheaf theories\footnote{see \ref{finiteness intro} for an explanation as to why the use of compact sheaf theories (i.e. $\Coh_{\Nn}()$ instead of $\IndCoh_{\Nn}()$) is necessary.} in non-abelian Hodge theory. 
\begin{theorem}  
\label{Eis decomposition II intro}
\,
\begin{enumerate}
    \item (Theorem \ref{th eis Higgs}). $\Coh_{\Nn}(\Higgs_G)$ is generated by the essential image of Eisenstein functors
    \[
    \Eis^{\Dol}_{P} : \Perf(\Higgs_M) \to \Coh_{\Nn}(\Higgs_G) , 
    \]
    for all parabolics $P \subset G$ and all Harder--Narasimhan types $\nu$.

    \item (Theorem \ref{th eis hodge}). $\Coh_{\Nn}(\Hodge_G)$ is generated by the essential image of Eisenstein functors
    \[
    \Eis^{\Hod}_{P} : \Perf(\Hodge_M) \to \Coh_{\Nn}(\Hodge_G) , 
    \]
    for all parabolics $P \subset G$ and all Harder--Narasimhan types $\nu$.
\end{enumerate}
\end{theorem} 
Our motivation for Theorem \ref{Eis decomposition II intro} is the belief that the Dolbeault geometric Langlands correspondence $\LL_G^{\Dol}$ should commute with $\Eis_P^{\Dol}$. Here we discuss some background and consequential ideas surrounding this belief. By the same construction as in the de Rham theory, as described in \ref{se: dR cusp-Eis intro}, the statement of Theorem \ref{Eis decomposition II intro}(1) can be used to decompose $\Coh_{\Nn}(\Higgs_G)$ into an exact sequence of dg categories 
\[
\begin{tikzcd}
\Coh_{\Nn}(\Higgs_{G})_{\Eis} 
\arrow[r,shift left=2pt]
& \Coh_{\Nn}(\Higgs_{G})
\arrow[l,shift left=2pt]
\arrow[r,shift left=2pt] 
& \Perf(\Higgs_{G}^{\irred})
\arrow[l,shift left=2pt] 
\end{tikzcd} .
\]
Moreover Theorem \ref{Eis decomposition II intro}(2) shows that the Dolbeault decompositions of Theorems \ref{Eis decomposition II intro}(1) and the de Rham decompositions of Theorem \ref{de Rham decompositions} are compatible with an interpolating Hodge family of cuspidal and Eisenstein components. Indeed, the restriction to the $\lambda = 0$ and $\lambda = 1$ fibers of the structural map $\Hodge_G \to \AA^1$ yields the following commutative diagram\footnote{We emphasise that the subscripts $\Nn$ considered in this diagram refer to three different global nilpotent cones, defined within the cotangent bundles of $\LocSys_G$, $\Hodge_G$ and $\Higgs_G$ respectively. See Section \ref{singularities} for details.} with exact rows:
\begin{equation}
\label{big cusp-Eis diagram}
\begin{tikzcd}
\Coh_{\Nn}(\LocSys_{G})_{\Eis} 
\arrow[r,shift left=2pt]
& \Coh_{\Nn}(\LocSys_{G}) 
\arrow[l,shift left=2pt]
\arrow[r,shift left=2pt] 
& \Perf(\LocSys^{\irred}_{G})
\arrow[l,shift left=2pt] 
\\
\Coh_{\Nn}(\Hodge_{G})_{\Eis} 
\arrow[u, "\lambda = 1"]
\arrow[r,shift left=2pt]
\arrow[d, "\lambda = 0"']
& \Coh_{\Nn}(\Hodge_{G})
\arrow[u, "\lambda = 1"]
\arrow[l,shift left=2pt]
\arrow[r,shift left=2pt] 
\arrow[d, "\lambda = 0"']
& \Perf(\Hodge^{\irred}_{G})
\arrow[u, "\lambda = 1"]
\arrow[l,shift left=2pt] 
\arrow[d, "\lambda = 0"']
\\
\Coh_{\Nn}(\Higgs_{G})_{\Eis} 
\arrow[r,shift left=2pt]
& \Coh_{\Nn}(\Higgs_{G})
\arrow[l,shift left=2pt]
\arrow[r,shift left=2pt] 
& \Perf(\Higgs_{G}^{\irred})
\arrow[l,shift left=2pt] 
\end{tikzcd} ,
\end{equation}
as one would expect from the Donagi--Pantev idea that de Rham and Dolbeault geometric Langlands correspondences are compatible with the Hodge degeneration along the limit $\lambda \to 0$. 

For $G=GL_n$, our Dolbeault cuspidal-Eisenstein decompositions are well-suited to the work of Arinkin \cite{arinkin_transform}, who constructed, by Fourier--Mukai transform, an equivalence of categories 
\[
\LL^{\Dol, \cusp}_{GL_n} : \QCoh(\Higgs_{GL_n}^{\irred}) \xrightarrow{\cong} \QCoh(\Higgs_{GL_n}^{\irred}) , 
\]
that provides a \textit{cuspidal Dolbeault geometric Langlands correspondence}. It is reasonable to hope there exists a similar Fourier--Mukai-type equivalence $\LL^{\Dol, \cusp}_{G} : \QCoh(\Higgs_{G^{\vee}}^{\irred}) \xrightarrow{\cong} \QCoh(\Higgs_{G}^{\irred})$ for any reductive group $G$. This is easily confirmed for $G=SL_n$ using the same transform as Arinkin \cite{FHR, groechenig_complex} and is likely not too hard for other groups that are closely related to $GL_n$. Now, after Theorem \ref{Eis decomposition II intro}, a restriction of the cuspidal conjecture to compact objects can be explicitly identified as the orthogonal complement to an \textit{Eisenstein Dolbeault geometric Langlands conjecture}, 
displayed on the left hand side of the following diagram with exact rows and conjectural vertical equivalences. 
\[
\begin{tikzcd}
\Coh_{\Nn}(\Higgs_G)_{\Eis} 
\arrow[d, "(\LL^{\Dol, \Eis}_G)^{c}"] 
\arrow[r,shift left=2pt]
& \Coh_{\Nn}(\Higgs_G) 
\arrow[d, "(\LL^{\Dol}_G)^{c}"]
\arrow[l,shift left=2pt]
\arrow[r,shift left=2pt]
& \Perf(\Higgs_G^{\irred}) 
\arrow[l,shift left=2pt]
\arrow[d, "(\LL^{\Dol, \cusp}_G)^{c}"] 
\\
\Coh_{\Nn}(\Higgs_G)_{\Eis} 
\arrow[r,shift left=2pt]
& \Coh_{\Nn}(\Higgs_G)
\arrow[l,shift left=2pt]
\arrow[r,shift left=2pt]
& \Perf(\Higgs_G^{\irred})
\arrow[l,shift left=2pt]
\end{tikzcd} . 
\]
Thus, by emulating the de Rham theory, our Theorem \ref{Eis decomposition II intro} lets us explicitly specify how $\LL^{\Dol}_G$ is expected to decompose on compact objects. We suggest this interpretation of our results lends some evidence to the appearance of nilpotent singular support in the Dolbeault geometric Langlands conjecture, and moreover, that cuspidal-Eisenstein decompositions should play a central role in the Dolbeault theory. However, we say nothing on the enlargement beyond compact objects, i.e. to ind-coherent sheaves, thus avoiding well-known finiteness issues (see \ref{finiteness intro}).

\subsubsection{Twistors and BBB-branes} Alongside Dolbeault and Hodge theory, we study an associated twistor theoretic geometry, considered as a projectification of the Hodge deformation that extends the deformation coordinate $\AA^1$ to the twistor $\PP^1$. This is achieved by the moduli stack $\Twistor_G \to \PP^1$ from \cite{simpson_twistor, FH1}, presented as a pushout of derived analytic stacks
    \[
    \Twistor_G = \Hodge_G^{\an} \bigsqcup_{\Rep^{\an}_G \times \GG_m} \ol{\Hodge}_G^{\an} ,
    \]
with gluing data defined by the analytic Riemann--Hilbert correspondence $\LocSys_G^{\an} \cong \Rep_G^{\an}$ between the underlying analytic stacks. The notation $\ol{\Hodge}_G^{\an}$ denotes the analytic Hodge stack defined over the complex conjugate base curve $\ol{X}$. At the level of coarse moduli spaces, Simpson \cite[\textsection 4]{simpson_twistor} attributes this construction to letters of Deligne -- hence we name $\Twistor_G$ the \textit{Deligne twistor stack.} The pushout construction receives a natural twistor variant of geometric Eisenstein series functor 
\[
\Eis^{\Tw}_P = \Eis^{\Hod}_P \times_{\Betti} \ol{\Eis}^{\Hod}_P : \IndCoh(\Twistor_M) \to \IndCoh(\Twistor_G) ,
\]
defined as a pullback of complex conjugate pairs of Hodge Eisenstein functors (Proposition \ref{spec tw eis}). 

In the twistor setting, our focus lies with a hyperholomorphic sheaf theory $\IndCoh^{\BBB}_{\Nn}(\Twistor_G)$ of so-called \textit{BBB-branes} with nilpotent singular support. Building on the definition from \cite{FH1}, we propose $\IndCoh^{\BBB}_{\Nn}(\Twistor_G)$ as a categorification of the BBB-branes studied by Kapustin and Witten in their gauge theory interpretation of the geometric Langlands program \cite{kapustin&witten}. Our definition of $\IndCoh^{\BBB}(\Twistor_G)$ and therefore of $\IndCoh^{\BBB}_{\Nn}(\Twistor_G)$ is approximately as follows: we introduce a category $\Wils^{\Tw}_G \subset \QCoh(\Twistor_G)$ of \textit{twistor Wilson eigensheaves}, defined by skyscraper sheaves on embedded projective lines, and define $\IndCoh^{\BBB}(\Twistor_G)$ by prescribing conditions on how sheaves intersect with objects of $\WW^{\Tw}_G$. 

Our main results on BBB-branes are as follows. We check the twistor Eisenstein functors $\Eis^{\Tw}_P$ preserve the BBB-brane categories (Proposition \ref{Eis BBB to BBB}) and prove the resultant BBB-type Eisenstein functors induce the following hyperholomorphic variant of cuspidal-Eisenstein decomposition. 

\begin{theorem} (Theorem \ref{decompose BBB}). 
\label{Eis decomposition intro}
The category $\Coh^{\BBB}_{\Nn}(\Twistor_G)$ is generated by the essential image of 
\[
\Eis^{\Tw}_{P} : \Perf^{\BBB}(\Twistor_M) \to \Coh^{\BBB}_{\Nn}(\Twistor_G) ,
\]
for all parabolics $P \subset G$. 
\end{theorem} 

Theorem \ref{Eis decomposition intro} once more gives rise to an exact sequence of dg categories 
\begin{equation}
\label{eq: cusp-Eis BBB intro}
\begin{tikzcd}
\Coh^{\BBB}_{\Nn}(\Twistor_{G})_{\Eis} 
\arrow[r,shift left=2pt]
& \Coh^{\BBB}_{\Nn}(\Twistor_{G}) 
\arrow[l,shift left=2pt]
\arrow[r,shift left=2pt] 
& \Perf^{\BBB}(\Twistor^{\irred}_{G})
\arrow[l,shift left=2pt] 
\end{tikzcd} ,
\end{equation} 
providing notions of \textit{Eisenstein BBB-branes} and \textit{cuspidal BBB-branes} within $\Coh^{\BBB}_{\Nn}(\Twistor_G)$. 

\subsubsection{Twistor geometric Langlands} 

The long-term idea behind Theorem \ref{Eis decomposition intro} is to develop a mathematical theory for the duality between BBB-branes and BAA-branes, proposed by Kapustin and Witten in their physical study of S-duality in 4D supersymmetric Yang-Mills theories \cite[\textsection 12]{kapustin&witten}. In one possible mathematical formulation of the duality, one can imagine a so-called \textit{twistor geometric Langlands conjecture} as an equivalence of dg categories 
\begin{equation}
\label{eq: LL^Tw}
\LL^{\Tw}_G : \TwMod^{\BAA}(\Bun_{G^{\vee}}) \xrightarrow{\cong} \IndCoh_{\Nn}^{\BBB}(\Twistor_G) , 
\end{equation}
between our spectral category $\IndCoh^{\BBB}_{\Nn}(\Twistor_G)$ of ind-coherent nilpotent BBB-branes and an automorphic category $\TwMod^{\BAA}(\Bun_{G^{\vee}})$ of quantised BAA-branes, the latter consisting of twistor D-modules on $\Bun_{G^{\vee}}$ that satisfy prescribed intersection conditions with $\Wils^{\Tw}_G$ via a twistor-valued convolution action. The idea of twistor geometric Langlands is to take Hodge filtrations of the de Rham theory into account, in a similar manner to the classical limit, but restricted to a rigid `hyperk\"ahler' class of objects -- i.e. our notion of BAA-branes and BBB-branes -- which are assumed to be compatible with the Hodge deformation, and thus easier to control\footnote{We hope this philosophy makes twistor geometric Langlands more approachable than Dolbeault geometric Langlands.}\footnote{A somewhat analogous idea of Pădurariu and Toda \cite{BPS_dolbeault} proposes a restriction of Dolbeault geometric Langlands to quasi-BPS categories over the semistable locus $\Higgs_G^{\sst}$, understood as a more precise Hodge deformation of $\LocSys_G$.}. 

Theorem \ref{Eis decomposition intro} is motivated by the expectation that such a twistor geometric Langlands correspondence should commute with our twistor Eisenstein functors $\Eis^{\Tw}_P$, and so the Langlands duality for BBB-branes predicted by Kapustin--Witten splits into cuspidal and Eisenstein components. One could therefore expect that $\Eis^{\Tw}_P$ plays as central a role in the twistor theory as $\Eis^{\dR}_P$ in the de Rham theory, providing potential methods with which to approach a conjecture such as \eqref{eq: LL^Tw}. 

\subsubsection{Digression on finiteness}
\label{finiteness intro} 
It is generally desirable to work with dg categories that are \textit{compactly generated}. The fact that $\IndCoh_{\Nn}(\LocSys_G)$ and $\DMod(\Bun_{G^{\vee}})$ are compactly generated, as proven in  \cite{arinkin&gaitsgory, drinfeld&gaitsgory}, is an essential property of the de Rham theory, used during the construction of the de Rham Langlands functor in \cite[Cory 1.6.5]{GLC1}. 

In the Dolbeault and Hodge theories, the moduli $\Higgs_G$ and $\Hodge_G$ are not quasi-compact and the categories $\IndCoh(\Higgs_G)$, $\IndCoh(\Hodge_G)$, $\IndCoh_{\Nn}(\Higgs_G)$ and $\IndCoh_{\Nn}(\Hodge_G)$ all fail to be compactly generated. In Dolbeault geometric Langlands, one expects to either modify $\IndCoh_{\Nn}(\Higgs_G)$, or develop new tools for non-compactly generated dg categories. It is therefore unlikely that an ind-extension of Theorem \ref{Eis decomposition II intro} exists, at least on the nose and via available methods. However, we hope to show in future work that the category $\IndCoh^{\BBB}_{\Nn}(\Twistor_G)$, or a small modification of it, is compactly generated, and thus prove an ind-extension of Theorem \ref{de Rham decompositions}. Philosophically this hope is a reflection of the scarcity and rigidity of hyperholomorphic sheaves in hyperk\"ahler geometry.

Non-compactness makes it a priori unclear how Eisenstein functors preserve compact objects, as the $G$-induction maps $\Higgs_P \to \Higgs_G$ and $\Hodge_P \to \Hodge_G$ are not proper. Following work of Laumon \cite{laumon_automorphic}, we deal with this issue in our work by considering a modified form of Eisenstein functor defined over Harder--Narasimhan stratifications. See Sections \ref{spectral Eis Hodge} and \ref{twistor spectral p and q} for the precise constructions and Corollary \ref{co: Eis compact} for the statement that $\Eis_P^{\Dol}$ and $\Eis_P^{\Hod}$ preserve compact objects. 

\subsubsection{Dirac--Higgs complex}
\label{Dirac--Higgs intro}
Categories of compact BBB-branes nonetheless contain interesting objects. To illustrate this within the context of our work, we compute the cuspidal-Eisenstein decompositions of a mathematically and physically influential BBB-brane - the \textit{Dirac--Higgs complex.} Initially coming from gauge theory constructions of Hitchin \cite{hitchin_dirac}, the Dirac--Higgs complex in our setting can be described as an object $\DD_{\rho} \in \Perf^{\BBB}(\Twistor_G)$ associated to a representation $\rho \in \Rep(G)$, as per \cite[Thm. D]{FH1}\footnote{The author thanks Emilio Franco for first suggesting that the Dirac--Higgs complex is an important class of objects in $\Perf^{\BBB}(\Twistor_G)$, an observation that led to the paper \cite{FH1}.}. On these objects, we show that cuspidal-Eisenstein decompositions are controlled by the induction functors $\Ind_{M}^{G} : \Rep(M) \to \Rep(G)$, providing a direct connection to representation theory. 
\begin{proposition}
\label{Dirac-Higgs decompositions intro}
(Proposition \ref{Dirac-Higgs decompositions}). 
Consider a collection $\{ \rho_{M} \in \Rep(M) \}_{P}$ parametrised by parabolics $P \subset G$. Let $\rho \in \Rep(G)$ be the representation generated by $\Ind^{G}_M : \Rep(M) \to \Rep(G)$ acting on the collection. Then $\DD_{\rho} \in \Perf^{\BBB}(\Twistor_G)$ is generated by the cuspidal and Eisenstein components 
\[
\DD_{\rho}^{\cusp} = \DD_{\rho} |_{\Twistor_G^{\irred}} , 
\quad 
\DD_{\rho}^{\Eis} = 
\big\langle \Eis^{\Tw}_{P}( \DD_{\rho_M} ) \big\rangle_{P} , 
\]
where the latter notation is for the object generated $\Eis^{\Tw}_P$ acting on the given family of complexes over all parabolic subgroups $P \subset G$. 
\end{proposition}
In the adjoint representation, the Dirac--Higgs construction yields the tangent complex on $\Twistor_G$, and the consequential fact that the tangent complex lies in $\Perf^{\BBB}(\Twistor_G)$ is analogous to the classical fact that the tangent bundle of a hyperk\"ahler manifold is naturally hyperholomorphic. 

\subsubsection{Riemann--Hilbert and restricted variation} The construction of twistor moduli as a pushout 
    \[
    \Twistor_G = \Hodge_G^{\an} \bigsqcup_{\Rep^{\an}_G \times \GG_m} \ol{\Hodge}_G^{\an} ,
    \]
turns the analytic Riemann--Hilbert correspondence $\Rep_G^{\an} \cong \LocSys_G^{\an}$ into a key structural component of the twistor geometry. The correspondence is algebraic in formal neighbourhoods of points, and this statement can understood as an equivalence $\Rep_G^{\restr} \cong \LocSys_G^{\restr}$ between the algebraic moduli stacks of local systems with \textit{restricted variation}, introduced by Arinkin--Gaitsgory--Kazhdan--Raskin--Rozenblyum--Varshavsky \cite{AGKRRV} to analyse eigensheaves, compare the Betti and de Rham geometric Langlands theories, and to formulate (and partially prove \cite{gaitsgory&raskin_positive}) geometric Langlands correspondences in positive characteristic. In this article we use the notion of restricted variation to construct an algebraic twistor stack 
    \[
    \Twistor^{\restr, \alg}_G = \Hodge_G^{\restr, \alg} \bigsqcup_{\Rep^{\restr, \alg}_G \times \GG_m} \ol{\Hodge}_G^{\restr, \alg} ,
    \]
for which the underlying analytic stack $\Twistor^{\restr}_G = (\Twistor^{\restr, \alg}_G)^{\an}$ contains all geometric points of $\Twistor_G$. We use this construction to prove the following structural result on $\Twistor_G$. 

\begin{proposition}
\label{Tw lines algebraic intro}
(Proposition \ref{Tw lines algebraic}). Every section of the structural map $\Twistor_G \to \PP^1$ factors as 
\[
\PP^1 \to \Twistor_G^{\restr} \to \Twistor_G , 
\]
where $\PP^1 \to \Twistor_G^{\restr}$ is the analytification of an algebraic map $\PP^1 \to \Twistor_G^{\restr, \alg}$.  
\end{proposition}
We record below some additional analysis of the Riemann--Hilbert correspondence used in our work. 
\begin{proposition}
\label{RH Propn intro}
Algebraic Riemann--Hilbert $\Rep_G^{\restr} \to \LocSys_G^{\restr}$ and analytic Riemann--Hilbert $\RH : \Rep_G^{\an} \xrightarrow{\cong} \LocSys_G^{\an}$ together induce the following results.  
\begin{enumerate}
    \item (Proposition \ref{RH and Nilp an} and \cite{AGKRRV}). A pullback relation $\RH^{*}\Nn_{\dR}^{\an} \cong \Nn_{\B}^{\an}$ on the analytic de Rhan and Betti nilpotent cones and an equivalence of nilpotent singular support categories   
    \begin{equation}
    \label{an RH intro}
    \IndCoh_{\Nn^{\an}_{\dR}}( \LocSys_G^{\an} ) \cong \IndCoh_{\Nn^{\an}_{\B}}( \Rep_G^{\an} ) , 
    \end{equation}
    which extends the analytification of the algebraic equivalence  
    \begin{equation}
    \label{alg RH intro}
    \IndCoh_{\Nn_{\B}}( \Rep_G^{\restr} ) \cong \IndCoh_{\Nn_{\dR}}( \LocSys_G^{\restr} ), 
    \end{equation}
    between spectral sides of restricted Betti/de Rham geometric Langlands \cite{AGKRRV, GLC1}. 

    \item (Section \ref{nilpotent sheaves}). A pullback presentation 
    \[
    \begin{tikzcd}
    \IndCoh_{\Nn_{\Tw}}(\Twistor_G) \arrow[d] \arrow[r] & \IndCoh_{\ol{\Nn}_{\Hod}}(\ol{\Hodge}^{\an}_G) \arrow[d] \\
    \IndCoh_{\Nn^{\an}_{\Hod}}(\Hodge^{\an}_G) \arrow[r] & \IndCoh_{\Nn_B^{\an}}(\Rep^{\an}_G \times \GG_m)
    \end{tikzcd} , 
    \]
    featuring the categories of sheaves with singular support contained within the 'twistor nilpotent cone' $\Nn_{\Tw} \subset \Twistor_G$ defined in \ref{twistor nilp}. 
    \end{enumerate}
\end{proposition}

\subsubsection{Analytic sheaves} To pass freely between the algebraic and analytic topologies in instances such as Proposition \ref{RH Propn intro}(1) we record below some general comparison results concerning singular support and analytification. Precisely, given a stack $\Yy$ and a conic of singularities $\Lambda \subset \Sing(\Yy)$, we compare the singular support categories $\IndCoh_{\Lambda}(\Yy)$ defined by Arinkin and Gaitsgory \cite{arinkin&gaitsgory} to the analytification procedure $\Yy \mapsto \Yy^{\an}$ defined by Holstein and Porta \cite{holstein&porta}. 

\begin{theorem} Let $\Yy$ be a geometric quasi-smooth stack with analytification $\Yy^{\an}$. 
\label{an and GAGA intro}
\, 
\begin{itemize}
    \item (Lemma \ref{an and Sing}). There exists a canonical pullback square 
\[
\begin{tikzcd}
\Sing(\Yy)^{\an} \arrow[d] \arrow[r] & \Sing(\Yy^{\an}) \arrow[d] \\
(T^{*}\Yy)^{\an} \arrow[r] & T^{*}(\Yy^{\an}) 
\end{tikzcd} ,
\]
in which both horizontal arrows are equivalences. 

    \item (Corollary \ref{an and GAGA}). Given a conic $\Lambda \subset \Sing(\Yy)$ then analytification induces a commutative square  
\[
\begin{tikzcd}
\IndCoh_{\Lambda}(\Yy) \arrow[d] \arrow[r, "(\cdot)^{\an}"] & \IndCoh_{\Lambda^{\an}}(\Yy^{\an}) \arrow[d] \\
\IndCoh(\Yy) \arrow[r, "(\cdot)^{\an}"] & \IndCoh(\Yy^{\an})
\end{tikzcd} . 
\]
\end{itemize}
\end{theorem}

\subsubsection{Wilson operators} Our final result concerns the comparison between two fundamental sources of symmetry on the spectral categories $\IndCoh_{\Nn}^{\BBB}(\Twistor_G)$, namely that of the Eisenstein functors $\Eis^{\Tw}_P$ alongside the \textit{twistor Wilson functors}
\[
\WW^{\Tw}_{\rho, x, \ol{x}} : \IndCoh(\Twistor_G) \to \IndCoh(\Twistor_G) , 
\]
defined by a universal tensor action that depends on the fixed choice of data $\rho \in \Rep(G)$, $x \in X^{\an}$ and $\ol{x} \in \ol{X}^{\an}$. See Proposition \ref{twistor wilson exists} for a precise definition of this functor. 

\begin{theorem}
\label{Wilson and Eis intro}
(Theorem \ref{Wilson and Eis}.)
Fix $\rho \in Rep(G)$ and take the Levi reduction $\Red_M^{G}(\rho) \in \Rep(M)$. Then there exists a commutative square 
\[
\begin{tikzcd}
\IndCoh(\Twistor_M) \arrow[d, "\WW_{\Red_M^{G}(\rho), x}"] \arrow[r, "\Eis_{P}^{\Tw}"] & \IndCoh(\Twistor_G) \arrow[d, "\WW_{\rho, x}"] \\
\IndCoh(\Twistor_M) \arrow[r, "\Eis_{P}^{\Tw}"] & \IndCoh(\Twistor_G)
\end{tikzcd}, 
\]
which restricts to the commutative square 
\[
\begin{tikzcd}
\IndCoh^{\BBB}_{\Nn}(\Twistor_M) \arrow[d, "\WW_{\Red_M^{G}(\rho), x}"] \arrow[r, "\Eis_{P}^{\Tw}"] & \IndCoh^{\BBB}_{\Nn}(\Twistor_G) \arrow[d, "\WW_{\rho, x}"] \\
\IndCoh^{\BBB}_{\Nn}(\Twistor_M) \arrow[r, "\Eis_{P}^{\Tw}"] & \IndCoh^{\BBB}_{\Nn}(\Twistor_G)
\end{tikzcd} .
\]
\end{theorem}

\subsection{Structure} Section \ref{deformations} contains preliminary material on the singular support condition for algebraic and analytic sheaves. Section \ref{singular support} briefly recalls the definition of singular support, due to Arinkin and Gaitsgory \cite{arinkin&gaitsgory}; Section \ref{sing and an} is dedicated to various compatibilities between singular support and analytification, all of which are formal consequences of a comparison between analytic and algebraic deformations.  

Section \ref{Twistors NAH} is preliminary material on twistors of local systems and their moduli. Sections \ref{NAH} provide the construction of various moduli in non-abelian Hodge theory, including the moduli of twistors; Section \ref{singularities} is a study of the various singular support conditions on these moduli; Section \ref{Instability strata} is on the Harder--Narasimhan stratification in the Dolbeault, Hodge and twistor settings. 

Section \ref{hyperholomorphic branes} introduces our primary sheaf theory of interest -- the category of ind-coherent BBB-branes with nilpotent singular support. Section \ref{classical twistors} contains motivational recollections from twistor theory; Section \ref{se: twistor lines} analyses the geometry of twistor lines on the twistor stack; Section \ref{se: eigensheaves} studies skyscraper sheaves supported on twistor lines, which play the role of Wilson eigensheaves in our theory; Section \ref{se: BBB-branes} defines the categories of BBB-branes and studies some of their basic properties. 

Section \ref{Twistors and geometric Eisenstein series} is on the construction of the Eisenstein functors. Section \ref{se: Eis functors} recalls the definition of de Rham and Betti Eisenstein functors, as well as their basic properties under Riemann--Hilbert and analytification; Section \ref{spectral Eis Hodge} defines the Dolbeault and Hodge Eisenstein functors, including their Harder--Narasimhan-type description; Section \ref{twistor spectral p and q} is on twistor Eisenstein functors, which are shown to preserve nilpotent singular support; Section \ref{se: hyperholomorphic induction} is dedicated to proving that twistor Eisenstein functors preserve the categories of BBB-branes. 

Section \ref{se: Hyperholomorphic applications} contains our results on cuspidal-Eisenstein decompositions. Section \ref{se structure} contains the statement and proofs of our cuspidal-Eisenstein decompositions in non-abelian Hodge theory; Section \ref{Dirac--Higgs decompositions} computes the cuspidal-Eisenstein decomposition of a special class of BBB-branes called the \textit{Dirac--Higgs complexes}. 

Section \ref{se: Wilson} contains our results on the compatibility between Eisenstein and Wilson functors in the twistor setting. Section \ref{se: Wilson operators} is on the construction of twistor Wilson functors; Section \ref{se: intertwining} computes the intertwining compatibility relation between twistor Eisenstein and Wilson functors.  

\subsection{Conventions} We work in the category $\dSt$ of derived algebraic stacks defined over the \'{e}tale site of derived algebraic spaces, alongside the category $\dAnSt$ of derived $\CC$-analytic stacks defined over the \'{e}tale site of derived analytic spaces. By default, when we say \textit{algebraic stack} and \textit{analytic stack}, we refer to objects of $\dSt$ and $\dAnSt$ respectively. 

We use the following standard notation from derived algebraic geometry. Given $\Xx, \Yy \in \dSt$, we denote by $\Map(\Xx, \Yy)$ the $\infty$-groupoid of maps $\Xx \to \Yy$ in the $\infty$-category $\dSt$. $\Hom(\Xx, \Yy)$ is reserved for the set $\pi_{0}(\Map(\Xx, \Yy))$ of homomorphisms in the classical underlying category. The mapping stack $\Maps(\Xx, \Yy) \in \dSt$ is defined by the functor of points $S \mapsto \Map(\Xx \times S, \Yy)$. The analogous constructions on objects $\Xx, \Yy \in \dAnSt$ yields an $\infty$-groupoid $\anMap(\Xx, \Yy)$ and an analytic stack $\anMaps(\Xx, \Yy) \in \dAnSt$. Given a stack $\Yy$, algebraic or analytic, we denote by $\Perf(\Yy), \Coh(\Yy), \QCoh(\Yy)$, $\IndCoh(\Yy)$ the respective dg categories of perfect, coherent, quasi-coherent and ind-coherent complexes of sheaves on $\Yy$. The notation $\Perf^{\heart}(\cdot)$ \textit{etc}. is for the standard underlying abelian categories. 

At various points we use the term `Betti gluing data'. This refers to the construction of twistor structures as a pullback of the form `$\Hodge \times_{\Betti} \ol{\Hodge}$', or as a pushout of the form `$\Hodge \sqcup_{\Betti} \ol{\Hodge}$', where both are products of a complex conjugate pair of Hodge structures taken over topological Betti data. The use of Betti gluing data is due to the natural invariance of the Betti theory under variations in complex structure. 

Throughout the paper our base field is always $\CC$. 

\subsection{Acknowledgements} The author would like to thank Dima Arinkin, Emilio Franco, Tony Pantev and Mauro Porta for conversations and suggestions that supported the completion of this article. A special thanks to Mauro Porta for providing the author with the statement and proof of Lemma \ref{cotangent GAGA}, which plays an essential role in our analysis of analytic sheaves. 

The author is supported by the Horizon Europe Marie Skłodowska-Curie Action grant \textit{Hyperkähler mirror symmetry and Langlands duality} with grant agreement ID 101204490. 

\section{Singular support of analytic sheaves}
\label{deformations} 

\subsection{Singular support} 
\label{singular support}
We begin by recalling the theory of singular support for ind-coherent sheaves on derived stacks, introduced by Arinkin and Gaitsgory \cite{arinkin&gaitsgory} to formulate the precise statement of the de Rham geometric Langlands correspondence. 

\subsubsection{Singularities}
\label{singularities of stacks}
Let $\Yy$ be a quasi-smooth geometric stack and denote by $\LL_\Yy \in \Perf(\Yy)$ and $\TT_{\Yy} = \LL_{\Yy}^{\vee}$ the cotangent and tangent complex respectively. By the quasi-smoothness hypothesis on $\Yy$, $\LL_{\Yy}$ is supported in cohomological degrees $[-1, \infty ]$. Let $T^{*}(\Yy) = \tot(\LL_{\Yy})$ denote the cotangent stack of $\Yy$. The classical $1$-stack $\Sing(\Yy) \subset T^{*}(\Yy)$ of singularities is defined to be the classical $1$-stack truncation of 
\[
\tot(\LL_{\Yy}[1]) = \Spec_{\Yy}( \Sym^{\bullet}_{\Oo_\Yy}(\TT_{\Yy}[1] ) ). 
\]
If $\Yy$ is smooth then $\LL_{\Yy}[1]$ is supported in positive cohomological degree and so $\Sing(\Yy)$ is empty. In general $\Sing(\Yy)$ parametrises the natural obstructions to the smoothness of $\Yy$. 

The construction $\Sing(\cdot)$ is functorial in the following sense. Given a map $f : \Yy_1 \to \Yy_2$ of quasi-smooth geometric stacks there is an induced map $\TT_{\Yy_1} \to f^{*}\TT_{\Yy_2}$ in $\Perf(\Yy_1)$, to which taking spectra of the corresponding $1$-shifted symmetric algebras yields a morphism 
\[
\Sing(f) : \Sing(\Yy_2)_{\Yy_1} \to \Sing(\Yy_1),
\]
called the \textit{singular codifferential of $f$} \cite[\textsection 2.4]{arinkin&gaitsgory}. Here the notation $\Sing(\Yy_2)_{\Yy_1}$ denotes the classical $1$-stack underlying the derived product $\Sing(\Yy_2) \times_{\Yy_2}\Yy_1$, which coincides with the spectra of $\Sym^{\bullet}(f^{*}\TT_{\Yy_2}[1])$. 

\subsubsection{Singular support}
Consider quasi-smooth $\Yy$ and an object $\Ff \in \IndCoh(\Yy)$. The singular support of $\Ff$ is the substack $\SingSupp(\Ff) \subset T^{*}\Yy$ defined to be support for $\Ff$ when considered as a module over the graded algebra $\Sym^{\bullet}(H^1(\TT_{\Yy}))$. This measures the failure of $\Ff$ to be quasi-coherent, or for $\Ff \in \Coh(\Yy)$, the failure of $\Ff$ to be perfect. By construction, $\SingSupp(\Ff)$ is a closed conic within $\Sing(\Yy)$. To a fixed conic $\Lambda \subset \Sing(\Yy)$, one can assign an intermediate category
\begin{equation}
\label{intermediate sing supp}
\QCoh(\Yy) \subset \IndCoh_{\Lambda}(\Yy) \subset \IndCoh(\Yy) ,
\end{equation}
consisting of ind-coherent sheaves with singular support contained in $\Lambda$. The category $\QCoh(\Yy)$, viewed as a subcategory of $\IndCoh_{\Lambda}(\Yy)$, can be identified with the category $\IndCoh_{\{0\}}(\Yy)$ of ind-coherent sheaves with singular support contained in the zero section $\{0\} \subset \Sing(\Yy)$. The restriction to compact objects yields  
\[
\Perf(\Yy) = \Coh_{\{0\}}(\Yy) \subset \Coh_{\Lambda}(\Yy) \subset \Coh(\Yy) .
\]

\subsection{Analytification} 
\label{sing and an}

The analytification functor $(\cdot)^{\an} : \dSt \to \dAnSt$ due to Holstein and Porta \cite[Defn. 3.3]{holstein&porta} provides a vast generalisation of analytification for algebraic varieties, as studied in the classical theory of \textit{Géométrie Algébrique et Géométrie Analytique (GAGA)} due to Serre \cite{serre_GAGA}. For our purposes it is sufficient to consider $\Yy \in \dSt$ to be geometric, in which case the associated analytic stack can be described by the following functor of points. 

\begin{theorem} 
\label{geometric analytic}
\cite[Thm 3.10]{holstein&porta}. Let $\Yy$ be a derived geometric stack with analytification $\Yy^{\an}$. Then $\Yy^{\an}$ is the sheafification of the functor that sends a derived Stein space $S$ to $\Map(\Spec(\Gamma(S)), \Yy)$, where $\Gamma(S)$ are the global sections of $S$ considered as a derived commutative ring. 
\end{theorem}

Given a quasi-smooth analytic stack $\Zz$ with analytic cotangent complex $\LL_{\Zz}$, the construction of singular support can be restated verbatim on $(\Zz, \LL_{\Zz})$, providing analytic notions of the stack of singularities $\Sing(\Zz) \subset T^{*}\Zz$ and singular support for objects in $\IndCoh(\Zz)$. In this section, we consider the case where $\Zz = \Yy^{\an}$ for some algebraic stack $\Yy$, and compare the singular support construction to the analytification functor $(\cdot)^{\an} : \IndCoh(\Yy) \to \IndCoh(\Yy^{\an})$, constructed by Holstein--Porta in \cite[Cory. 5.6]{porta_GAGA} and Porta--Yu in \cite[\textsection 6.2]{porta&yu}. 

\subsubsection{Analytic deformations} First we compare the algebraic deformations of $\Yy$ to the analytic deformations of $\Yy^{\an}$ via a canonical equivalence $(\LL_{\Yy})^{\an} \cong \LL_{\Yy^{\an}}$ in $\Perf(\Yy^{\an})$. The equivalence has been constructed by Porta and Yu in the case where $\Yy$ is a Deligne--Mumford stack \cite[Thm. 5.21]{porta&yu}. The following generalisation to the case where $\Yy$ is a geometric stack is due to Porta, who we thank for sharing the proof with the author in email correspondence and allowing us to present his result.

\begin{lemma} 
\label{cotangent GAGA}
Let $\Yy$ be a derived geometric stack with analytification $\Yy^{\an}$. Then in $\Perf(\Yy^{\an})$ there exists a natural morphism 
\[
(\LL_{\Yy})^{\an} \to \LL_{\Yy^{\an}}, 
\]
and moreover this morphism is an equivalence. 
\end{lemma} 

\begin{proof} 
We first construct the morphism $(\LL_{\Yy})^{\an} \to \LL_{\Yy^{\an}}$. Since $\Yy$ is geometric, then by Theorem \ref{geometric analytic} it suffices to construct a canonical map $f^{*}(\LL_{\Yy})^{\an} \to f^{*}\LL_{\Yy^{\an}}$ for a given morphism $f : S \to \Yy^{\an}$ from a derived Stein space $S$. Given $M \in \Perf(S)$, the definition of the analytic cotangent complex (see for instance \cite[\textsection 5.2]{porta&yu_representability}) yields an identification 
\begin{equation}
\label{representing cotangents}
\Map(f^{*} \LL_{\Yy^{\an}}, M) = \Map_{S/}(S[M], \Yy^{\an}) ,
\end{equation}
where $S[M]$ is the split square-zero extension of $S$ by $M$. Now $\Gamma(S[M])$ coincides with the split square-zero extension of $\Gamma(S)$ by $\Gamma(M)$. Therefore, up to localising on $S$, any such map in \eqref{representing cotangents} gives rise, in a canonical way, to a map in
\[
\Map_{\Spec(\Gamma(S))/}( \Spec(\Gamma(S) \oplus M), \Yy) = \Map(f^{\alg,*} \LL_{\Yy}, M) , 
\]
where $f^{\alg}$ is the map $S^{\alg} = \Spec(\Gamma(S)) \to \Yy$ corresponding to $f : S \to \Yy^{\an}$, which exists after localization on $S$. Thus, we obtain the canonical map
\[
\Map(f^{*}\LL_{\Yy^{\an}}, M) \to \Map(f^{\alg,*} \LL_{\Yy}, M). 
\]
By the description of analytification of modules in \cite[\textsection 5.1]{holstein&porta} this defines the map $f^{*}(\LL_{\Yy})^{\an} \to f^{*} \LL_{\Yy^{\an}}$. To see why this is an equivalence we proceed by induction on the level of geometricity of $\Yy$. To begin the induction, note that if $\Yy$ is ($-1$)-geometric (\textit{i.e.} ($-1$)-Artin) then the statement follows from the result of Porta--Yu for Deligne--Mumford stacks \cite[Thm. 5.21]{porta&yu}. Now let $\Yy$ be $n$-geometric (\textit{i.e.} $n$-Artin) and consider the ($n-1$)-geometric diagonal $\Omega_{f} := S^{\alg} \times_{\Yy} S^{\alg}$ with analytification $\Omega_{f}^{\an} = S \times_{\Yy^{\an}} S$ and diagonal morphism $\delta_f : S \to \Omega_{f}^{\an}$. By the inductive hypothesis we have canonical equivalences 
\begin{equation}
\label{inductive hypothesis}
\delta_f^{*} (\LL_{\Omega_{f}})^{\an} \cong \delta_f^{*} \LL_{\Omega_{f}^{\an}} , 
\quad 
\delta_f^{*} (\LL_{\Omega_{f} / S^{\alg} \times S^{\alg} })^{\an} \cong \delta_f^{*} \LL_{\Omega_{f}^{\an} / S \times S } , 
\end{equation}
in $\Perf(S)$. The inductive step is performed by applying the statement of \cite[Lemma 3.9]{porta&yu_quantum}: 
\[
f^{*}\LL_{\Yy^{\an}} \cong \delta_f^{*} \LL_{\Omega_f^{\an} / S \times S}[-1] . 
\]
This allows us to conclude that \eqref{inductive hypothesis} induces an equivalence
\[
f^{*}(\LL_{\Yy})^{\an} 
\cong \delta_f^{*} (\LL_{\Omega_{f} / S^{\alg} \times S^{\alg} })^{\an}
\cong \delta_f^{*} \LL_{\Omega_{f}^{\an} / S \times S }
\cong f^{*} \LL_{\Yy^{\an}}. \qedhere
\]
\end{proof}

\subsubsection{Analytification and singularities}

\begin{lemma}
\label{an and Sing}
Given a quasi-smooth geometric stack $\Yy$ with analytification $\Yy^{\an}$, then there exists a canonical pullback square 
\[
\begin{tikzcd}
\Sing(\Yy)^{\an} \arrow[d] \arrow[r] & \Sing(\Yy^{\an}) \arrow[d] \\
(T^{*}\Yy)^{\an} \arrow[r] & T^{*}(\Yy^{\an}) 
\end{tikzcd} ,
\]
in which both horizontal arrows are equivalences. 
\end{lemma}

\begin{proof}
Follows from an equivalence of symmetric algebras $\Sym^{\bullet}(\TT_{\Yy})^{\an} \cong \Sym^{\bullet}(\TT_{\Yy^{\an}})$ induced by the equivalence $(\LL_{\Yy})^{\an} \cong \LL_{\Yy^{\an}}$ from Lemma \ref{cotangent GAGA}.
\end{proof}

\begin{lemma}
Let $\Yy_1, \Yy_2$ be quasi-smooth geometric stacks with respective analytifications $\Yy_1^{\an}, \Yy_2^{\an}$. Then the result of Lemma \ref{an and Sing} is functorial in the following sense: a morphism $f : \Yy_1 \to \Yy_2$ and its analytification $f^{\an} : \Yy^{\an}_1 \to \Yy_2$ admit compatible singular codifferentials 
\[
\Sing(f)^{\an} \simeq \Sing(f^{\an} ) , 
\]
as morphisms $\Sing(\Yy^{\an}_1) \to \Sing(\Yy^{\an}_2)$. 
\end{lemma}

\begin{proof}
By taking spectra of the respective symmetric algebras, $\Sing(f)^{\an}$ is induced by the canonical map $(\TT_{\Yy_2})^{\an} \to (f^{*}\TT_{\Yy_1})^{\an}$, and $\Sing(f^{\an})$ is induced by the canonical map $\TT_{\Yy_2^{\an}} \to (f^{\an})^{*}\TT_{\Yy_1}$. The proof therefore follows from a relative application of Lemma \ref{geometric analytic}, in which the two maps commute with the equivalences 
\[
\begin{tikzcd}
(\TT_{\Yy_2})^{\an} \arrow[r] \arrow[d, "\cong"] & (f^{*}\TT_{\Yy_1})^{\an} \arrow[d, "\cong"] \\
\TT_{\Yy_2^{\an}} \arrow[r] & (f^{\an})^{*}\TT_{\Yy_1^{\an}}
\end{tikzcd} ,
\]
where $(f^{*}\TT_{\Yy_1})^{\an} \cong (f^{\an})^{*}\TT_{\Yy_1^{\an}}$ factors through the natural equivalence $(f^{*}\TT_{\Yy_1})^{\an} \cong (f^{\an})^{*}\TT_{\Yy_1}^{\an}$.  
\end{proof}

\subsubsection{Analytification and singular support}

\begin{proposition}
\label{an and sing supp}
Let $\Yy$ be a quasi-smooth geometric stack with analytification $\Yy^{\an}$. Moreover let $\Lambda \subset \Sing(\Yy)$ be a conic with analytification $\Lambda^{\an} \subset \Sing(\Yy^{\an})$. Then analytification induces a pullback square  
\[
\begin{tikzcd}
\IndCoh_{\Lambda}(\Yy) \arrow[d] \arrow[r, "(\cdot)^{\an}"] & \IndCoh_{\Lambda^{\an}}(\Yy^{\an}) \arrow[d] \\
\IndCoh(\Yy) \arrow[r, "(\cdot)^{\an}"] & \IndCoh(\Yy^{\an})
\end{tikzcd} , 
\]
and so $\IndCoh_{\Lambda^{\an}}(\Yy^{\an})$ is the analytification of $\IndCoh_{\Lambda}(\Yy)$. 
\end{proposition}

\begin{proof}
Fix $\Ff \in \IndCoh_{\Lambda}(\Yy)$. We recall that $\SingSupp(\Ff^{\an})$ is computed as a module over the symmetric algebra $\Sym^{\bullet}(H^1(\LL_{\Yy^{\an}}))$, and $\SingSupp(\Ff)$ is computed as a module over $\Sym^{\bullet}(H^1(\LL_{\Yy}))$. The canonical equivalence $(\LL_{\Yy})^{\an} \cong \LL_{\Yy^{\an}}$ provided by Lemma \ref{cotangent GAGA} and the resultant equivalence $\Sym^{\bullet}(H^1(\LL_{\Yy}))^{\an} \cong \Sym^{\bullet}(H^1(\LL_{\Yy^{\an}}))$ identifies the two defining module structures, inducing an equivalence of stacks 
\[
\SingSupp(\Ff)^{\an} \cong \SingSupp(\Ff^{\an}).
\]
The containment $\SingSupp(\Ff) \subset \Lambda$ therefore naturally induces a containment 
\[
\SingSupp(\Ff^{\an}) \cong \SingSupp(\Ff)^{\an} \subset \Lambda^{\an},
\]
from which the statement follows. 
\end{proof}

\subsubsection{GAGA sheaves}
\label{GAGA sheaves}
Let $\Yy$ be a quasi-smooth geometric stack with analytification $\Yy^{\an}$. Over $\Yy^{\an}$ we define the category $\IndCoh^{\GAGA}(\Yy^{\an})$ of so-called \textit{GAGA-sheaves} to be the essential image of 
\[
(\cdot)^{\an} : \IndCoh(\Yy) \to \IndCoh(\Yy^{\an}). 
\]
If $\Yy$ is additionally proper over $\Spec(\CC)$ then by a GAGA theorem of Porta \cite[Thm. 7.2]{porta_GAGA}, analytification induces an equivalence of categories 
\[
\IndCoh(\Yy) \cong \IndCoh(\Yy^{\an}), 
\]
and so in this case our definition yields $\IndCoh^{\GAGA}(\Yy^{\an}) = \IndCoh(\Yy^{\an})$. For $\Yy$ not necessarily proper, $\IndCoh^{\GAGA}(\Yy^{\an})$ can be viewed as a rigid substitute for $\IndCoh(\Yy^{\an})$.  

Given a conic $\Delta \subset \Sing(\Yy^{\an})$ we also define the singular support GAGA categories 
\[
\IndCoh^{\GAGA}_{\Delta}(\Yy^{\an}) := \IndCoh_{\Delta}(\Yy^{\an}) \cap \IndCoh^{\GAGA}(\Yy^{\an}).
\]
In the case where $\Delta = \Lambda^{\an} \subset \Sing(\Yy^{\an})$ for some algebraic conic $\Lambda \subset \Sing(\Yy)$, Proposition \ref{an and sing supp} yields the following. 

\begin{corollary}
\label{an and GAGA}
If $\Yy$ is a quasi-smooth geometric stack then analytification induces a commutative square  
\[
\begin{tikzcd}
\IndCoh_{\Lambda}(\Yy) \arrow[d] \arrow[r, "(\cdot)^{\an}"] & \IndCoh^{\GAGA}_{\Lambda^{\an}}(\Yy^{\an}) \arrow[d] \\
\IndCoh(\Yy) \arrow[r, "(\cdot)^{\an}"] & \IndCoh^{\GAGA}(\Yy^{\an})
\end{tikzcd} , 
\]
and so $\IndCoh^{\GAGA}_{\Lambda^{\an}}(\Yy^{\an})$ is the analytification of $\IndCoh_{\Lambda}(\Yy)$. 
\end{corollary}

\section{Non-abelian Hodge theory}
\label{Twistors NAH}

\subsection{Moduli} 
\label{NAH}
\label{spectral}

Over a smooth projective curve\footnote{The constructions hold over far more general spaces than smooth curves -- for instance see \cite{porta&sala_notes}.} denoted by $X$ we present the geometries of de Rham, Betti, Dolbeault and Hodge type from non-abelian Hodge theory. We do so by associating to $X$ the \textit{Simpson shapes} $X_{\dR}$, $X_\B$, $X_{\Dol}$ and $X_{\Hod}$, introduced by Simpson in the series of papers including \cite{simpson_secondary, simpson_algebraic, simpson_geometricity}. In each instance, our moduli are defined to be mapping stacks of the form $\Maps(X_{\Sim}, BG)$, for $\Sim$ a formal variable taking values in $\{ \dR, \B, \Dol, \Hod\}$. 

\subsubsection{de Rham} The \textit{de Rham shape} $X_{\dR}$ is the pre-stack with $S$-points $X_{\dR}(S) = \Hom(h^{0}(S)_{red}, X)$ where $h^0(S)_{red}$ is the reduced and classically truncated subscheme of $S$. $X_{\dR}$ can also be presented as the formal completion of the diagonal embedding $X \hookrightarrow X \times X$. In the de Rham theory established by Gaitsgory and Rozenblyum \cite{gaitsgory&rozenblyum}, the dg category of $D$-modules on $X$ can be defined by $\DMod(X) = \QCoh(X_{\dR})$, and contains, within $\DMod^{\heart}(X)$, the subcategory of flat connections on vector bundles over $X$. Given an algebraic group $G$ with classifying stack $BG = [pt / G]$, the mapping stack 
\[
\LocSys_G = \Maps(X_{\dR}, BG) ,
\]
is the quasi-smooth geometric stack that parametrises flat connections on $G$-bundles over $X$. 

\subsubsection{Betti} The \textit{Betti shape} $X_\B$ is defined to be the constant stack associated to the underlying topological space of $X$. The category $\Shv(\Yy) = \QCoh(\Yy_\B)$ is the category of sheaves of $\CC$-vector spaces on $X$, which contains the subcategory of locally constant sheaves. The mapping stack 
\[
\Rep_{G} = \Maps(X_{\op{B}}, BG) ,
\]
is the quasi-smooth geometric stack that parametrises representations $\pi_1(X) \to G$. 

\subsubsection{Dolbeault} The Dolbeault shape $X_{\Dol}$ is defined to be the relative classifying stack for the formal group scheme $\widehat{TX} \to X$, where $\widehat{TX}$ is the completion of the tangent bundle of $X$ along the zero section. By the spectral correspondence, as per \cite{simpson_II}, the category $\QCoh^{\heart}(X_{\Dol})$ is the category of $\Sym^{\bullet}(T_{X})$-modules, where a $\Sym^{\bullet}(T_{X})$-action is determined by a twisted endomorphism $\phi : \Ee \to \Ee \otimes \Omega_X$. Following Hitchin and Simpson's foundational works \cite{hitchin_self, simpson_higgs}, such a $\phi$ is called a \textit{Higgs field}, and such pairs $(\Ee, \phi)$ are called \textit{Higgs sheaves}. Locally free objects are called \textit{Higgs bundles}, and moreover a \textit{$G$-Higgs bundle} is a pair $(P, \theta)$ consisting of a $G$-bundle $P \to X$ equipped with an adjoint-valued section $\theta \in H^{0}(\ad(P) \otimes \Omega_X)$. The mapping stack
\[
\Higgs_G = \Maps(X_{\Dol}, BG ) ,
\]
is the quasi-smooth geometric stack that parametrises $G$-Higgs bundles on $X$. 

\subsubsection{Hodge}
\label{NAH Hodge}

The Hodge shape $X_{\Hod}$ is defined to be the deformation to the normal bundle of the natural map $X \to X_{\dR}$ (see \cite{gaitsgory_VolII} for such deformations). It inherits a structural morphism $X_{\Hod} \to \AA^1_{\lambda}$ with general fiber $X_{\dR} \cong X_{\Hod} \times_{\AA^1} \{ \lambda \}$ for $\lambda \neq 0$ and central fiber $X_{\Dol} = X_{\Hod} \times_{\AA^1} \{0\}$. A description of the category $\QCoh^{\heart}(X_{\Hod})$ can be given in two equivalent ways: firstly as filtered objects and secondly as $\GG_m$-equivariant families over $\AA^1$. In the first instance, $\QCoh^{\heart}(X_{\Hod})$ consists of Rees modules taken over the algebra $\Rees(D_{X})$ obtained by filtering $D_X$ by the order of differential operators. In the second instance, $\QCoh^{\heart}(X_{\Hod})$ consists of families of $D_{\lambda}$-modules indexed by $\lambda \in \AA^1$, where $D_{\lambda}$ is the sheaf of $\lambda$-differential operators. For a study of $D_{\lambda}$ see for instance notes of Beilinson--Drinfeld \cite{beilinson&drinfeld_opers}. The locally free objects are $\lambda$-connections, introduced by Deligne and Simpson \cite{simpson_twistor} as $\Omega_X$-valued endomorphims that satisfy a $\lambda$-rescaled Leibniz rule. The mapping stack 
\[
\Hodge_G = \Maps_{/\AA^1}(X_{\Hod} , BG) ,
\]
is the quasi-smooth geometric stack that parametrises $G$-bundles equipped with a $\lambda$-connection. It inherits a $\GG_m$-equivariant structural morphism $\Hodge_G \to \AA^1$, for which the rescaling $\GG_m$-action over the fibers away from $\lambda = 0$ yields an equivalence 
\begin{equation}
\label{Hodge trivialisation}
\Hodge_G \times_{\AA^1} \GG_m \cong \LocSys_G \times \GG_m, 
\end{equation}
which to a $\lambda$-connection $(E, \nabla^{\lambda})$ assigns the flat bundle $(E, \lambda^{-1} \cdot \nabla^{\lambda})$ and the value of $\lambda$.  
\subsubsection{Analytic moduli}
\label{Analytic prelude}

Let $\Sim$ be a formal variable equal to one of $\dR, \B, \Dol, \Hod$. The non-abelian Hodge moduli stacks introduced in Section \ref{NAH} can be constructed verbatim in the category $\dAnSt$ of derived analytic stacks by simply replacing the mapping stacks $\Maps(X_{\Sim}, BG)$ with their analytic counterparts $\anMaps((X^{\an})_{\Sim}, (BG)^{\an})$, where $X^{\an}$ is the underlying complex analytic curve of $X$, as studied in Serre's GAGA theory \cite{serre_GAGA}. 

These moduli satisfy a \textit{universal GAGA property}, defined and proven by Holstein and Porta \cite{holstein&porta}, which follows from the existence of the equivalences 
\begin{equation}
\label{universal GAGA}
\Maps(X_{\Sim}, BG)^{\an} \cong \Maps((X_{\Sim})^{\an}, (BG)^{\an}) \cong \anMaps((X^{\an})_{\Sim}, (BG)^{\an}) . 
\end{equation}
We therefore introduce the notation  
\[
\LocSys_G^{\an} := \anMaps\big( (X^{\an})_{\dR} , BG^{\an} \big) , 
\]
\[
\Rep_G^{\an} := \anMaps\big( (X^{\an})_{\B} , BG^{\an} \big) , 
\]
\[
\Higgs_G^{\an} := \anMaps\big( (X^{\an})_{\Dol}  , BG^{\an} \big) , 
\]
\[
\Hodge_G^{\an} := \anMaps\big( (X^{\an})_{\Hod}  , BG^{\an} \big) ,
\]
and use \eqref{universal GAGA} to conclude that each analytic moduli stack in non-abelian Hodge theory is the analytification of its algebraic counterpart. 

\subsubsection{Riemann--Hilbert} 
\label{se: twistor stack} 
We recall the construction of the Deligne twistor stack from \cite{simpson_twistor, FH1}. Passage to the analytic topology allows us to apply the Riemann--Hilbert correspondence of Porta \cite[Thm. 2]{porta_RH} and Holstein--Porta \cite[Thm. 1.5]{holstein&porta}, in which the Riemann--Hilbert transformation morphism 
\[
\nu_{\RH} : (X^{\an})_{\dR} \to (X^{\an})_{\B} , 
\]
is constructed and shown to induce the equivalence of analytic stacks  
\begin{equation}
\label{RH stacks}
\RH := \nu_{\RH}^{*} : \Rep_G^{\an} \xrightarrow{\cong} \LocSys_G^{\an},
\end{equation}
formally extending the Riemann--Hilbert correspondence of Deligne \cite{deligne_RH} to families of local systems. 

\subsubsection{Twistors} 
\label{se twistor stack}
We apply the same moduli constructions over the curve $\ol{X}^{\an}$ with complex conjugate holomorphic structure. These moduli are denoted by the notation 
\[
\ol{\LocSys}^{\an}_G := \anMaps\big( (\ol{X}^{\an})_{\dR} , BG^{\an} \big) , 
\]
\[
\ol{\Rep}^{\an}_G  := \anMaps\big( (\ol{X}^{\an})_{\B} , BG^{\an} \big) , 
\]
\[
\ol{\Higgs}^{\an}_G := \anMaps\big( (\ol{X}^{\an})_{\Dol}  , BG^{\an} \big) , 
\]
\[
\ol{\Hodge}^{\an}_G := \anMaps\big( (\ol{X}^{\an})_{\Hod}  , BG^{\an} \big) .
\]
The identification $\pi_1(X^{\an}) = \pi_1(\ol{X}^{\an})$ of fundamental groups gives rise to an identification $\Rep_G = \ol{\Rep}_G$ between the moduli of representations $\pi_1(X^{\an}) = \pi_1(\ol{X}^{\an}) \to G$. We can use this so-called \textit{Betti invariance} property to adjoin two complex conjugate Riemann--Hilbert correspondences into the diagram 
\begin{equation}
\label{gluing derived}
\LocSys^{\an}_G \xleftarrow{\RH} \Rep^{\an}_G = \ol{\Rep}^{\an}_G \xrightarrow{\ol{\RH}} \ol{\LocSys}^{\an}_G. 
\end{equation}
We recall from \ref{NAH Hodge} that deformation to the normal cone provides a $\GG_m$-equivariant map $\Hodge^{\an}_G \to \AA^1$ and a natural equivalence $\Hodge_G^{\an} \times_{\AA^1} \GG_m = \LocSys_G^{\an} \times \GG_m$. A pullback of the Riemann--Hilbert correspondence therefore defines an equivalence of derived analytic stacks 
\[
\RH : \Rep_G^{\an} \times \GG_m \xrightarrow{\cong} \Hodge_G^{\an} \times_{\AA^1} \GG_m.
\]
This allows us to present the \textit{Deligne twistor stack} \cite[Defn. 3.2]{FH1} as the pushout 
\begin{equation}
\label{twistor stack}
\begin{tikzcd}
\Rep_G^{\an} \times \GG_m = \ol{\Rep}_G \times \GG_m \arrow[r, "\ol{\RH}"] \arrow[d, "\RH"'] &  \ol{\Hodge}_G^{\an} \arrow[d] \\
 \Hodge_G^{\an} \arrow[r] & \Twistor_G 
\end{tikzcd},
\end{equation}
which is a quasi-smooth analytic stack, equipped with a pushout structural morphism $\Twistor_G \to \PP^1$ with general fiber 
\[
\LocSys_G^{\an} \cong \Twistor_G \times_{\PP^1} \{ \lambda \} ,
\]
for the values $\lambda \neq 0, \infty$, and special fibers 
\[
\Higgs_G^{\an} = \Twistor_G \times_{\PP^1} \{ 0 \}, \quad 
\ol{\Higgs}_G^{\an} = \Twistor_G \times_{\PP^1} \{ \infty \} .
\]

\subsubsection{Algebraic twistors} 
\label{algebraic construction}
The analytic topology is essential for the pushout in \eqref{twistor stack} to be well-defined. It fails algebraically because $\LocSys_G$ and $\Rep_G$ are not isomorphic as algebraic stacks, and so $\Twistor_G$ cannot be expressed as the analytification of an algebraic stack. Nonetheless the construction can be 'forced' to work algebraically for moduli of local systems with \textit{restricted variation}, as introduced by Arinkin--Gaitsgory--Kazhdan--Raskin--Rozenblyum--Varshavsky \cite{AGKRRV}, who construct a pair of substacks $\LocSys_G^{\restr} \to \LocSys_G$ and $\Rep_G^{\restr} \to \Rep_G$ that admit an algebraic Riemann--Hilbert correspondence 
\begin{equation}
\label{RH restr}
\RH^{\alg} : \LocSys_G^{\restr} \xrightarrow{\cong} \Rep_G^{\restr} . 
\end{equation}
The existence and construction of $\RH^{\alg}$ is as reflection of the fact that the Riemann--Hilbert correspondence is algebraic within formal neighbourhoods of points. 

For our twistor setting let us define the moduli $\Hodge_G^{\restr} \subset \Hodge_G$ by restricting the substack 
\[
\Hodge_G \times_{\AA^1} \GG_m = \LocSys_G \times \GG_m \subset \Hodge_G
\]
to $\LocSys_G^{\restr} \times \GG_m$. A base change of \eqref{RH restr} induces the equivalence
\[
\RH^{\alg} : \Rep_G^{\restr} \times \GG_m \cong \LocSys_G^{\restr} \times \GG_m = \Hodge_G^{\restr} \times_{\AA^1} \GG_m ,  
\]
which in turn defines an immersion 
\[
\RH^{\alg} : \Rep_G^{\restr} \times \GG_m \xrightarrow{\cong} \Hodge_G^{\restr} \times_{\AA^1} \GG_m \to \Hodge_G^{\restr} .
\]
By repeating the above constructions over $\ol{X}^{\an}$ we then take the pushout of derived algebraic stacks 
\begin{equation}
\label{twistor stack restr}
\begin{tikzcd}
\Rep_G^{\restr} \times \GG_m = \ol{\Rep}^{\restr}_G \times \GG_m \arrow[r, "\ol{\RH}^{\alg}"] \arrow[d, "\RH^{\alg}"'] &  \ol{\Hodge}_G^{\restr} \arrow[d] \\
\Hodge^{\restr}_G \arrow[r] & \Twistor_G^{\restr, \alg}
\end{tikzcd},
\end{equation}
to define the \textit{algebraic Deligne twistor stack} $\Twistor_G^{\restr, \alg} \in \dSt$. 

By comparison with the analytic pushout construction of \eqref{twistor stack}, an application of the analytification functor $(\cdot)^{\an} : \dSt \to \dAnSt$ and the equivalence $(\RH^{\alg})^{\an} \simeq \RH$ naturally yields a substack 
\[
\Twistor_G^{\restr} := (\Twistor_G^{\restr, \alg})^{\an} \subset \Twistor_G .
\]
We summarise the main conclusion of our constructions as follows. 

\begin{corollary}
$\Twistor_G$ contains a substack $\Twistor_G^{\restr}$ parametrising local systems with restricted variation that can be expressed as the analytification of an algebraic stack. 
\end{corollary}

\begin{remark}
\label{Dol restr}
It would be interesting to construct a "more precise" restricted variation twistor stack by also modifying the Dolbeault loci over $0$ and $\infty$. This would require a moduli stack $\Higgs_G^{\restr}$ of Higgs bundles with restricted variation that can be understood in terms of a Hodge degeneration of $\LocSys_G^{\restr}$ (\textit{i.e.}  constructing the classical limit of restricted de Rham geometric Langlands). 
\end{remark}

\subsubsection{$\GG_m$-action}
\label{GG_m spectral} 
We now describe how to perform a 'hyperk\"ahler rotation' on $\Twistor_G$, obtained by varying the complex structure via a $\GG_m$-action. Consider the conjugate pair of $\GG_m$-equivariant structural morphisms 
\[
\Hodge_G^{\an} \to \AA^1, \quad \quad \ol{\Hodge}_G^{\an} \to \AA^1, 
\]
where the pair of affine lines are considered with complex conjugate coordinates. Given $\mu \in \GG_m$ consider the complex conjugate point $\ol{\mu} = 1/\mu \in \GG_m$. One has a conjugate pairs of automorphisms 
\[
\mu : \Hodge_G^{\an} \to \Hodge_G^{\an}, \quad \quad \ol{\mu} : \ol{\Hodge}_G^{\an} \to \ol{\Hodge}_G^{\an}, 
\]
that are equivariant with respect to dilations on $\AA^1$, thus giving rise to commutative squares 
\[
\begin{tikzcd}
\Hodge_G^{\an} \arrow[d, "\mu"] \arrow[r] & \AA^1 \arrow[d, "\mu"] \\
\Hodge_G^{\an} \arrow[r] & \AA^1 
\end{tikzcd}, 
\quad \quad 
\begin{tikzcd}
\ol{\Hodge}_G^{\an} \arrow[d, "\ol{\mu}"] \arrow[r] & \AA^1 \arrow[d, "\ol{\mu}"] \\
\ol{\Hodge}_{G}^{\an} \arrow[r] & \AA^1 
\end{tikzcd}, 
\]
which glue to an automorphism $\mu : \Twistor_{G} \to \Twistor_{G}$ that lives over the dilating $\GG_m$-action on $\PP^1$.

\subsection{Singularities}
\label{singularities}
We recall from \cite{arinkin&gaitsgory} how to define the \textit{global nilpotent cone} and study the corresponding singular support categories of so-called nilpotent sheaves. We observe that the de Rham constructions from \cite{arinkin&gaitsgory} can be extended to other moduli from non-abelian Hodge theory under our consideration -- \textit{i.e.} those of type Dolbeault, Betti, Hodge and twistor.  

\subsubsection{Nilpotent singularities} 
\label{Hodge nilpotent cone}
\label{nilpotents}

The definition of the global nilpotent cone in $\Sing(\LocSys_G)$ as a locus of nilpotent singularities \cite[\textsection 11.1.1]{arinkin&gaitsgory} can be stated verbatim in the following generality. Take any quasi-smooth mapping stack $\Mm = \Maps(Y, BG)$ such that the tangent complex can be presented as 
\[
\TT_{\Mm} = p_{2, *}\ad(\Uu_{\Mm})[1], 
\]
where $\ad(\Uu_{\Mm})$ is the universal $G$-bundle $\Uu_{\Mm}$ on $Y \times \Mm$ considered in the adjoint representation and $p_{2}$ is the projection onto the second factor. The $S$-points of the cotangent stack $T^{*}\Mm$ are given by pairs $(f,A)$ where $f : Y \times S \to BG$ is an $S$-point of $\Mm$ and $A$ is an object of  
\[
A \in H^{\bullet}(Y \times S, f^{*}\ad(\Uu_{\Mm}^{\vee})) . 
\]
Moreover the $1$-stack of singularities $\Sing(\Mm) \subset T^{*}\Mm$ is parametrised by the subclass of objects for which $A$ lies in $H^{0}(Y, f^{*}\ad(\Uu_{\Mm}^{\vee}))$ (see \cite[\textsection 10.3.2]{arinkin&gaitsgory} for the de Rham version of this statement). Given an $S$-point $(f, A)$ of $\Sing(\Mm)$, we call the section $A \in H^{0}(Y, f^{*}\ad(\Uu_{\Mm}^{\vee}))$ \textit{nilpotent} if, for any local trivialisation of the $G$-bundle $f^{*}\ad(\Uu_{\Mm}^{\vee}) \to S \times Y$, the induced local section takes values in the conical set of nilpotent elements $\Nilp(\gG^{\vee}) \subset \gG^{\vee}$. 

\begin{definition}
The \textit{global nilpotent cone} $\Nn_{\Mm} \subset \Sing(\Mm)$ associated to $\Mm$ is the $1$-stack parametrising pairs $(f, A)$ for which $A$ is nilpotent. 
\end{definition}
The definition applies for $\Mm$ equal to any of the stacks $\LocSys_G$, $\Higgs_G$, $\Rep_G$ and $\Hodge_G$ from non-abelian Hodge theory. We therefore obtain the family of global nilpotent cones 
\[
\Nn_{\dR} \subset \Sing(\LocSys_G), 
\quad \Nn_{\Dol} \subset \Sing(\Higgs_G), 
\quad \Nn_{\B} \subset \Sing(\Rep_G), 
\quad \Nn_{\Hod} \subset \Sing(\Hodge_G) ,
\]
which parametrise nilpotent sections of the respective universal families denoted by $\Uu_{\dR}$, $\Uu_{\Dol}$, $\Uu_{\B}$ and $\Uu_{\Hod}$. Also for $\Yy = \Bun_G$ we obtain the global nilpotent cone 
\[
\Nn_{\Bun_G} \subset T^{*}\Bun_G = \Higgs_G
\]
which famously coincides with the zero fiber of the Hitchin fibration\footnote{$\Nn_{\Bun} \subset T^{*}\Bun_G$ will appear only in passing in this article because its singular support categories belong to the automorphic side of geometric Langlands and our work lies strictly on the spectral side.} \cite{hitchin_self}.  In \ref{twistor nilp} we also define a conic $\Nn_{\Tw} \subset \Sing(\Twistor_G)$ by a twistoral gluing procedure, and to this end we first provide a series of comparison results between the non-abelian Hodge family of nilpotent cones.  

\subsubsection{de Rham to Dolbeault comparison} $\Nn_{\Hod}$ is the natural $\AA^1$-interpolation between $\Nn_{\Dol}$ and $\Nn_{\dR}$ which can be canonically recovered via the pullback squares 
\[
\begin{tikzcd}
\Nn_{\Dol} \arrow[d] \arrow[r] & \Nn_{\Hod} \arrow[d] & \Nn_{\dR} \arrow[l] \arrow[d] \\
\{ 0 \} \arrow[r] & \AA^1 & \{ 1 \} \arrow[l]
\end{tikzcd} , 
\]
determined by the structural map $\Sing(\Hodge_G) \to \Hodge_G \to \AA^1$. The corresponding singular support categories therefore inherit the restriction functors 
\[
\IndCoh_{\Nn_{\Hod}}(\Hodge_G) \to \IndCoh_{\Nn_{\Dol}}(\Higgs_G) , 
\]
\[
\IndCoh_{\Nn_{\Hod}}(\Hodge_G) \to \IndCoh_{\Nn_{\dR}}(\LocSys_G) , 
\]
which can be understood, respectively, as the associated graded construction and forgetting the data of a good filtration. In this sense $\IndCoh_{\Nn_{\Hod}}(\Hodge_G)$ can be interpreted as a Hodge deformation of the de Rham category $\IndCoh_{\Nn_{\dR}}(\LocSys_G)$ that degenerates to the Dolbeault category $\IndCoh_{\Nn_{\Dol}}(\Higgs_G)$. 

\subsubsection{de Rham to Betti comparison} Recall from \ref{algebraic construction} the algebraic Riemann--Hilbert correspondence 
\[
\RH^{\alg} : \LocSys_G^{\restr} \to \Rep_G^{\restr} . 
\]
We record the following key conclusions regarding $\RH^{\alg}$ from \cite{AGKRRV, GLC1}. The canonical map  
\[
\Nn^{\restr}_{\dR} := \Nn_{\dR} \times_{\LocSys_G} \times \Sing(\LocSys_G^{\restr}) \to \Nn_{\dR}, 
\]
is a closed immersion and surjective at geometric points. We apply the slight misuse of notation of identifying $\Nn^{\restr}_{\dR}$ with $\Nn_{\dR}$ and similarly $\Nn^{\restr}_{\B}$ with $\Nn_{\B}$. 

\begin{proposition} 
\label{RH and Nilp alg}
The Riemann--Hilbert correspondence $\RH^{\alg} : \Rep_G^{\restr} \xrightarrow{\cong} \LocSys_G^{\restr}$ induces a canonical isomorphism $\RH^{\alg} : \Nn_{\B} \xrightarrow{\cong} \Nn_{\dR}$ and an equivalence of categories 
\[
\IndCoh_{\Nn_{\B}}(\Rep_G^{\restr}) \cong \IndCoh_{\Nn_{\dR}}(\LocSys_G^{\restr}) . 
\]
\end{proposition}
 
\begin{proof}
Both are consequences of the fact that $\RH^{\alg}$ induces a universal equivalence 
\[
\RH^{\alg, *} ( {\Uu_{\dR}^{\vee}}|_{X_{\dR} \times \LocSys_G^{\restr}} ) \cong \Uu_{\B}^{\vee}|_{X_{\B} \times \Rep_G^{\restr}} , 
\]
and moreover this equivalence preserves nilpotent sections. See also \cite[(4.2)]{GLC1}. 
\end{proof}

\begin{remark} In \textit{loc. cit.}, Proposition \ref{RH and Nilp alg} part of the deep result that the Riemann--Hilbert correspondence can be used to prove a logical equivalence between the (restricted and non-restricted) Betti and de Rham geometric Langlands correspondences \cite[Thm. 3.5.6]{GLC1}. 
\end{remark}
Analytification of Proposition \ref{RH and Nilp alg} induces the corresponding statement in the analytic topology.  

\begin{proposition}
\label{RH and Nilp an}
\label{de Rham Betti nilp sheaves}
The Riemann--Hilbert correspondence $\RH : \Rep_G^{\an} \xrightarrow{\cong} \LocSys_G^{\an}$ induces a canonical isomorphism $\RH : \Nn^{\an}_{\B} \xrightarrow{\cong} \Nn^{\an}_{\dR}$ and an equivalence 
\[
\IndCoh_{\Nn_{\B}^{\an}}(\Rep_G^{\an}) \cong \IndCoh_{\Nn_{\dR}^{\an}}(\LocSys_G^{\an}) . 
\]
\end{proposition}

\subsubsection{Twistor nilpotent cone} 
\label{twistor nilp}
Following the twistor gluing constructions of \ref{se twistor stack} we now apply Proposition \ref{RH and Nilp an} as an instance of Betti gluing data. Let $\Nn_{\B \times \GG_m}^{\an}$ and $\Nn_{\dR \times \GG_m}^{\an}$ denote the $\GG_m$-extensions to the trivial families $\Rep_G \times \GG_m$ and $\LocSys_G \times \GG_m$. Then Proposition \ref{RH and Nilp an} induces an equivalence 
\[
\RH : \Nn_{\B \times \GG_m}^{\an} \xrightarrow{\cong} \Nn_{\dR \times \GG_m}^{\an} .  
\]
Via the isomorphism $\LocSys^{\an}_G \times \GG_m \cong \Hodge^{\an}_G \times_{\AA^1} \GG_m$ one also has an equivalence 
\[
\Nn_{\dR \times \GG_m}^{\an} \cong \Nn_{\Hod}^{\an} \times_{\AA^1} \GG_m ,  
\]
and we denote the resultant immersion also by $\RH : \Nn_{\B \times \GG_m}^{\an} \to \Nn_{\Hod}$. By considering the complex conjugate constructions we define the following pushout of global nilpotent cones.

\begin{definition}
The twistor nilpotent cone $\Nn_{\Tw} \subset \Sing(\Twistor_G)$ is defined to be the pushout of analytic 1-stacks 
\[
\begin{tikzcd}
\Nn_{\B \times \GG_m}^{\an} = \ol{\Nn}_{\B \times \GG_m}^{\an} \arrow[d, "\RH"'] \arrow[r, "\ol{\RH}"] 
& \ol{\Nn}_{\Hod}^{\an} \arrow[d] 
\\
\Nn_{\Hod}^{\an} \arrow[r] & \Nn_{\Tw}
\end{tikzcd} . 
\]
\end{definition}

Following Proposition \ref{RH and Nilp alg} the same construction works in the algebraic topology to define a conic 
\[
\Nn_{\Tw}^{\alg} \subset \Sing(\Twistor_G^{\restr, \alg})
\]
over the algebraic twistor stack defined in \ref{algebraic construction}. Then $\Nn_{\Tw}$ can be recovered as the analytification
\[
\Nn_{\Tw} = (\Nn_{\Tw}^{\alg})^{\an} . 
\]

\subsubsection{Nilpotent sheaves} 
\label{nilpotent sheaves}
By the functoriality of singular support \cite[Propn. 7.1.3]{arinkin&gaitsgory} the nilpotent singular support categories admit a pullback functor 
\[
\IndCoh_{\Nn_{\Tw}}(\Twistor_G) \to \IndCoh_{\Nn^{\an}_{\Hod}}(\Hodge^{\an}_G) ,
\]
given by restriction of the structural map $\Twistor_G \to \PP^1$ over the standard affine chart $\PP^1 - \{\infty\} \subset \PP^1$. From the pushout presentations of $\Twistor_G$ and $\Nn_{\Tw}$ one has a pullback square 
\begin{equation}
\label{eq pullback nilps}
\begin{tikzcd}
\IndCoh_{\Nn_{\Tw}}(\Twistor_G) \arrow[d] \arrow[r] & \IndCoh_{\ol{\Nn}^{\an}_{\Hod}}(\ol{\Hodge}^{\an}_G) \arrow[d] \\
\IndCoh_{\Nn^{\an}_{\Hod}}(\Hodge^{\an}_G) \arrow[r] & \IndCoh_{\Nn_{\B \times \GG_m}^{\an}}(\Rep^{\an}_G \times \GG_m)
\end{tikzcd} , 
\end{equation}
where the lower horizontal arrow factors through the Riemann--Hilbert equivalence $\IndCoh_{\Nn^{\an}_{\dR}}( \LocSys_G^{\an} ) \cong \IndCoh_{\Nn^{\an}_{\B}}( \Rep_G^{\an} )$ as stated in Proposition \ref{RH and Nilp an}.  This determines the natural decomposition of nilpotent sheaves on $\Twistor_G$ over the two hemispherical affine charts $\PP^1 - \{\infty\} \subset \PP^1$ and $\PP^1 - \{\ 0 \} \subset \PP^1$. 

\subsubsection{GAGA sheaves} 
\label{GAGA sheaves Tw_G}
\label{Nilp GAGA}
Over an algebraic stack $\Yy$ with analytification $\Yy^{\an}$, we have defined in \ref{GAGA sheaves} the category $\IndCoh^{\GAGA}(\Yy^{\an})$ of \textit{GAGA sheaves} to be the essential image of the analytification functor $\IndCoh(\Yy) \to \IndCoh(\Yy^{\an})$. This construction applies to $\Yy^{\an} = \Hodge_G^{\an}$ but not to $\Twistor_G$, for the twistor stack cannot globally be presented as the analytification of an algebraic stack. Instead, we make the modification of defining $\IndCoh^{\GAGA}(\Twistor_G)$ to be the pullback
\begin{equation}
\label{GAGA on Tw_G}
\begin{tikzcd}
\IndCoh^{\GAGA}(\Twistor_G) \arrow[d] \arrow[r] & \IndCoh^{\GAGA}(\ol{\Hodge}^{\an}_G) \arrow[d] \\
\IndCoh^{\GAGA}(\Hodge^{\an}_G) \arrow[r] & \IndCoh(\Rep^{\an}_G \times \GG_m)
\end{tikzcd} .
\end{equation}
Objects of $\IndCoh^{\GAGA}(\Twistor_G)$ are analytic sheaves that admit an algebraic structure after restriction to the hemispherical affine charts. One can view the subcategory $\IndCoh^{\GAGA}(\Twistor_G)$ as a rigid substitute for $\IndCoh(\Twistor_G)$. 

We then define the category of \textit{nilpotent GAGA sheaves} to be 
\[
\IndCoh^{\GAGA}_{\Nn_{\Tw}}(\Twistor_G) := \IndCoh^{\GAGA}(\Twistor_G) \cap \IndCoh_{\Nn_{\Tw}}(\Twistor_G) ,
\]
which from \eqref{eq pullback nilps} and \eqref{GAGA on Tw_G} can naturally be presented as a pullback 
\[
\begin{tikzcd}
\IndCoh^{\GAGA}_{\Nn_{\Tw}}(\Twistor_G) \arrow[d] \arrow[r] & \IndCoh^{\GAGA}_{\Nn^{\an}_{\Hod}}(\ol{\Hodge}^{\an}_G) \arrow[d] \\
\IndCoh^{\GAGA}_{\Nn^{\an}_{\Hod}}(\Hodge^{\an}_G) \arrow[r] & \IndCoh(\Rep^{\an}_G \times \GG_m)
\end{tikzcd} .
\]
Thus $\IndCoh^{\GAGA}_{\Nn_{\Tw}}(\Twistor_G)$ is an analytic projectivisation of the deformation family $\IndCoh_{\Nn_{\Hod}}(\Hodge_G)$ that expresses $\IndCoh_{\Nn_{\Dol}}(\Higgs_G)$ as a degeneration of $\IndCoh_{\Nn_{\dR}}(\LocSys_G)$. In other words, an extension of deformation coordinate from $\AA^1$ to $\PP^1$ that adjoins $\IndCoh_{\ol{\Nn}^{\an}_{\Dol}}(\ol{\Higgs}^{\an}_G)$ at infinity. 
 
From this point onwards we simplify our notation by dropping the non-abelian Hodge subscripts $(\cdot)_{\dR}, (\cdot)_{\Hod}$, \textit{etc.} and the analytification superscripts $(\cdot)^{\an}$ from the family of global nilpotent cones. 

\subsection{Stability} 
\label{Instability strata}

Harder--Narasimhan theory provides the following fundamental structure theorem for $G$-bundles with extra structure: an object is either semistable, or admits a canonical semistable reduction to a parabolic subgroup of $G$. Harder--Narasimhan theory in a geometric Langlands context is often used to verify finiteness conditions, an idea that goes back to Laumon \cite{laumon_automorphic} who used the Harder--Narasimhan stratification to verify finiteness conditions for geometric Eisenstein series. The stratification is also a central feature of the classical Drinfeld--Gaitsgory proof that $\DMod(\Bun_G)$ is compactly generated \cite{drinfeld&gaitsgory_compact}. Our uses follow Laumon and will be taken up in Section \ref{spectral Eis Hodge}. 

\subsubsection{Harder--Narasimhan strata}
\label{HN strata}

Given a vector bundle $E$ on a space $\Xx$, the Harder--Narasimhan filtration is a canonically constructed increasing filtration of vector bundles  
\[
E_0 \subset E_1 \subset \cdots \subset E_{n-1} \subset E_n = E ,  
\]
such that the quotients $E_i / E_{i - 1}$ are semistable sheaves and the slope function $\{\mu(E_i / E_{i - 1})\}_{1\leq i \leq n}$ is a strictly decreasing series. One then defines the Harder--Narasimhan type $\mu(E)$ to be the ordered sequence of slopes. One can generalise the Harder--Narasimhan filtration to points on the stack of $G$-bundles $\Maps(\Xx, BG)$, as in \cite{anchouche_HN}, where one considers the $G$-induction maps 
\[
\Maps^{\sst, \mu}(\Xx, BP) \to \Maps(\Xx, BG) , 
\]
from the substacks $\Maps^{\sst, \mu}(\Xx, BP) \subset \Maps(\Xx, BP)$ of semistable $P$-bundles with fixed Harder--Narasimhan type $\mu$. By existence and uniqueness of the Harder--Narasimhan filtration, such maps define a covering of $\Maps(\Xx, BG)$ by locally closed immersions. 

\subsubsection{Dolbeault and Hodge strata} We apply the Harder--Narasimhan filtration to $G$-bundles on $\Xx = X_{\Dol}$ and $\Xx = X_{\Hod}$, \textit{i.e.} to $G$-Higgs bundles and $(G, \lambda)$-connections on $X$. The Harder--Narasimhan filtration for $G$-Higgs bundles on curves was constructed by Dey and Parthasarathi \cite{dey_harder} and extended to $(G, \lambda)$-connections on curves (or more generally $\Lambda$-modules) by Gurjar and Nitsure \cite{gurjar&nitsure}. Our strata are parametrised by the topological invariant given by pairs $\nu = (\chi, \mu)$ consisting of Harder--Narasimhan type $\mu$ and an object $\chi \in \pi_1(G)$ that parametrises the irreducible components $\Higgs^{\chi}_G$ and $\Hodge^{\chi}_G$ of $\Higgs_G$ and $\Hodge_G$ respectively. We then introduce the notation 
\[
\Higgs_P^{\sst, \nu} := \Higgs_P^{\sst, \mu} \cap \Higgs_P^{\sst, \chi} , 
\quad \quad 
\Higgs_P^{\sst, \nu} := \Higgs_P^{\sst, \mu} \cap \Higgs_P^{\sst, \chi}, 
\]
where for each $\nu = (\chi, \mu)$ the corresponding stacks are quasi-compact and geometric. We obtain the pair of Dolbeault and Hodge type stratifications 
\begin{equation}
\label{eq: Hodge strata}
\bigsqcup_{P, \nu} \Higgs_P^{\sst, \nu} \to \Higgs _G, 
\quad \quad 
\bigsqcup_{P, \nu} \Hodge_P^{\sst, \nu} \to \Hodge_G .   
\end{equation}

\subsubsection{Twistor strata} A $(G, \lambda)$-connection $(E, \nabla^{\lambda})$ on $X^{\an}$ is semistable if and only if the complex conjugate $\ol{(E, \nabla^{\lambda})}$ on $\ol{X}^{\an}$ is semistable. Indeed, complex conjugation is functorial and preserves the topological data of degree, rank and slope, while also preserving subbundles and filtrations. Therefore the Harder--Narasimhan stratifications in \eqref{eq: Hodge strata} glue to define a stratification of $\Twistor_G$ determined by the strata
\[
\Twistor_{P}^{\sst, \nu} := 
\Hodge_P^{\sst, \nu} \bigsqcup_{\Rep_G \times \GG_m} \ol{\Hodge}_P^{\sst, \nu} .
\]
By taking $G$-inductions there is an induced covering  
\[
\bigsqcup_{P, \nu} \Twistor_{P}^{\sst, \nu} \to \Twistor_G , 
\]
that we refer to as the \textit{Harder--Narasimhan stratification} on $\Twistor_G$. 

\subsubsection{Twistor space} 
\label{twistor space}
\label{spectral Betti gluing}
The locus $\Maps^{\sst}(\Xx, BG)$ of semistable $G$-bundles corresponds to the strata for which the parabolic is given by $P = G$ and the Harder--Narasimhan filtration is trivial. On $\Twistor_G$ this locus can be presented as a pushout 
\[
\Twistor_G^{\sst} = \Hodge_G^{\sst} \sqcup_{\Betti} \, \ol{\Hodge}_G^{\sst} . 
\]
After a 1-stack truncation one has a coarse moduli space structure 
\begin{equation}
\label{coarse moduli}
t_{0}(\Twistor_G^{\sst}) \to \Twistor_G^{\coarse} , 
\end{equation}
where $\Twistor_G^{\coarse}$ is represented by a classical complex analytic space (See \cite[Propn. 3.3]{FH1}). At the level of coarse moduli, the twistor construction is given by a gluing of classical analytic spaces  
\[
\Twistor_G^{\coarse} = \Hodge_G^{\coarse} \sqcup_{\Rep_G^{\coarse}} \ol{\Hodge}_G^{\coarse} , 
\]
as constructed and studied in the seminal work of Simpson \cite{simpson_twistor}. This space is realised as the twistor space of the moduli space $\Higgs_G^{\coarse}$ of semistable $G$-Higgs bundles, equipped with Hitchin's hyperk\"ahler metric \cite{hitchin_self}, where verifying the twistor axioms \cite[Thm. 4.2]{simpson_twistor} involves a $C^{\infty}$-trivialisation of the Hodge filtration obtained from the non-abelian Hodge correspondence. 

\subsubsection{Irreducibles} The irreducible or \textit{cuspidal} parts of our moduli stacks $\Maps(X_{\Sim}, BG)$ parametrise objects that do not admit a parabolic reduction, i.e. one considers $S$-points $S \times X_{\Sim} \to BG$ that do not factor through $BP$ for any proper parabolic $P \subset G$. By considering each of $\Sim \in \{\Dol, \dR, \B, \Hod \}$, we define the irreducible components of each non-abelian Hodge moduli by
\[
\Higgs_G^{\irred} \subset \Higgs_G ,
\quad \LocSys_G^{\irred} \subset \LocSys_G , 
\quad \Rep_G^{\irred} \subset \Rep_G , 
\quad \Hodge_G^{\irred} \subset \Hodge_G .
\]
The irreducible corresponding on the Deligne twistor stack can then be presented as a pushout 
\[
\Twistor_G^{\irred} = \Hodge_G^{\irred} \sqcup_{\Rep_G^{\irred} \times \GG_m} \, \ol{\Hodge}_G^{\irred} \subset \Twistor_G . 
\]
Irreducible objects are automatically stable, so for instance, $\Twistor_G^{\irred}$ is naturally a substack of $\Twistor_G^{\sst}$, and thus by restriction of \eqref{coarse moduli}, one obtains a map 
\[
t_{0}(\Twistor_G^{\irred}) \to \Twistor_G^{\coarse} , 
\]
which lands in the subspace of $\Twistor_G^{\coarse}$ that parametrises isomorphism classes of irreducible $(G, \lambda)$-connections over $X^{\an}$ and $\ol{X}^{\an}$. 

\section{Monoidal characterisation of hyperholomorphic branes}
\label{hyperholomorphic branes}

In Kapustin and Witten's study of 4D supersymmetric gauge theories \cite{kapustin&witten}, it is suggested that certain hyperholomorphic boundary conditions, known as \textit{BBB-branes}, should play some role within the geometric Langlands program. A core class of BBB-branes was initially proposed to be hyperholomorphic bundles on the coarse moduli space of semistable Higgs bundles \cite[\textsection 13]{kapustin&witten}. In this section we propose a broader definition for a dg category of BBB-branes, in following Langlands theoretic terms. We introduce a category $\Wils^{\Tw}_G \subset \QCoh(\Twistor_G)$ of \textit{twistor Wilson eigensheaves}, defined by skyscraper sheaves on embedded projective lines. The category $\IndCoh^{\BBB}(\Twistor_G)$ of \textit{ind-coherent BBB-branes} is then defined by prescribing conditions on how sheaves intersect with objects of $\WW^{\Tw}_G$. 

\subsection{Classical twistors} 
\label{classical twistors}
Our definition of BBB-branes is meaningful in the classical setting of twistor theory. We describe this first as motivation for our constructions. 

\subsubsection{Twistor transform} A hyperk\"ahler manifold $\Mm = (M ,g, I, J, K)$ has an associated twistor space $\Tt = \Tw(\Mm)$, from which the hyperk\"ahler structure can be reconstructed \cite[Thm 3.3]{HKLR}. One can thus pass hyperk\"ahler geometry back and forth between $\Mm$ and $\Tt$. For instance, in classical work of Kaledin and Verbitsky \cite{kaledin&verbitsky}, one has an equivalence of categories 
\begin{equation}
\label{twistor transform}
\left\{ 
\begin{array}{l}
\text{hyperholomorphic bundles on the}
\\
\qquad \text{hyperk\"ahler manifold $\Mm$} 
\end{array}
\right\}
\xrightarrow{\cong }
\left\{ 
\begin{array}{l}
\text{holomorphic bundles on twistor space $\Tt$} \\
\text{that are trivial on horizontal twistor lines}
\end{array}
\right\} ,
\end{equation}
known as the \textit{twistor transform.} This involves the following classical objects from twistor theory. A \textit{twistor line} is a section of the twistor projection $\Tt \to \PP^1$. A \textit{horizontal twistor line} is of the form $\lambda \mapsto (\lambda, m)$ for fixed $m \in \Mm$ after composition with the structural diffeomorphism $\phi : \Tt \xrightarrow{\cong} \PP^1 \times \Mm$ that $C^{\infty}$-trivialises the twistor fibration $\Tt \to \PP^1$. The existence of such a $\phi$ is one of the defining properties of twistor space -- see for instance \cite[Thm. 3.3]{HKLR}. 

\subsubsection{Categorification} Via a natural extension of the condition on the right hand side of \eqref{twistor transform}, we define the following hyperholomorphic dg categorical sheaf theory on twistor space $\Tt$.  
\begin{definition}
The category $\QCoh^{\BBB}(\Tt)$ of \textit{hyperholomorphic complexes}, or \textit{BBB-branes}, is defined to be the full subcategory of $\QCoh(\Tt)$ consisting of objects $\Bb \in \QCoh(\Tt)$ such that, for every horizontal twistor line $\sigma : \PP^1 \to \Tt$, there exists a graded vector space $V_{\sigma}^{\bullet} \in \QCoh(pt)$ and an equivalence 
\[
\Bb \otimes \sigma_{*}\Oo_{\PP^1} \cong V^{\bullet}_{\sigma} \otimes \sigma_{*}\Oo_{\PP^1} . 
\]
\end{definition}
The dg category $\QCoh^{\BBB}(\Tt)$ is a natural enhancement of the classical $1$-category of hyperholomorphic bundles on $\Mm$, and by \eqref{twistor transform}, contains it as a subcategory within the heart of the natural t-structure. 

In purely physical terms, the idea of defining a category of BBB-branes using twistor space was first proposed by Gaiotto \cite[App. C]{gaiotto}, based on an analysis of the supersymmetric ground states. 

\subsubsection{Application to our setting} 
\label{se: app to coarse}
Consider the case where $\Tt = \Twistor_G^{\coarse}$ is the coarse moduli space of the Deligne twistor stack (see \ref{twistor space}). The structural diffeomorphism $\phi : \Tt \xrightarrow{\cong} \Mm \times \PP^1$ is determined by non-abelian Hodge theory. In particular, given a horizontal twistor line $\sigma : \PP^1 \to \Twistor_G^{\coarse}$, the Higgs bundle $\sigma(0)$ and the local system $\sigma(1)$ are related by the non-abelian Hodge correspondence - \textit{i.e.} $\sigma$ is compatible with the $C^{\infty}$-trivialisation of the Hodge filtration. 

The idea of this section is to extend the construction of $\QCoh^{\BBB}(\Twistor_G^{\coarse})$ to define a similar category of \textit{`hyperholomorphic ind-coherent sheaves'}\footnote{We use the terminology `hyperholomorphic sheaf on $\Twistor_G$' informally, and avoid it in general,  because we do not specify what hyperk\"ahler geometry means in the context of derived algebraic/analytic geometry. Our constructions do however seem to anticipate that twistor space as a $\PP^1$-family of shifted symplectic stacks would be an intrinsic part of any such theory -- as suggested by the proposals of Katzarkov--Pandit--Spaide in \cite[Defn. 6.3]{KPS}.} on the derived analytic stack $\Twistor_G$.

\subsection{Twistor lines} 
\label{se: twistor lines}

We begin by studying $\PP^1$-families on $\Twistor_G$. Recall that this stack comes equipped with a structural map $\Twistor_G \to \PP^1$. We call the sections of this map \textit{twistor lines.} 

In this section we review the notion of a \textit{semi-simple} local system from \cite[\textsection 3.6]{AGKRRV}, adapted from geometric points of $\LocSys_G$ to twistor lines on $\Twistor_G$. 

\subsubsection{Association} 
\label{association}
Consider two standard parabolics $P_1$ and $P_2$ of $G$ with Levi quotients $M_1$ and $M_2$. Given twistor lines $\sigma_{M_1} : \PP^1 \to \Twistor_{M_1}$ and $\sigma_{M_2} : \PP^1 \to \Twistor_{M_2}$, we say that the pairs $(P_1, \sigma_{M_1})$ and $(P_2, \sigma_{M_2})$ are \textit{associated} whenever there exists a $G$-orbit $O$ in the partial flag variety $G/P$ such that, for any pair of points $(P'_1, P_2') \in O$, 
    \begin{itemize}
        \item The maps
        \[
        M_1 \longleftarrow P_1' \longleftarrow P_1' \cap P_2' \to P_2' \to M_2 ,
        \]
        identify both $M_1$ and $M_2$ with the Levi quotient of $P_1' \cap P_2'$.  
        \item Under the resulting isomorphism $M_1 \cong M_2$, the points parametrised by $\sigma_{M_1}$ and $\sigma_{M_2}$ are isomorphic local systems.
    \end{itemize}

\begin{lemma}
\label{associated} 
The pairs $(P_1, \sigma_{M_1})$ and $(P_2, \sigma_{M_2})$ are associated if and only if there exists a twistor line $\sigma : \PP^1 \to \Twistor_G$ equipped with $P_1$ and $P_2$ reductions such that the induced $M_1$ and $M_2$ maps are isomorphic to $\sigma_{M_1}$ and $\sigma_{M_2}$ respectively. 
\end{lemma}

\begin{proof}
Identical to the proof of \cite[Lem. 3.6.4]{AGKRRV} after replacing local systems with twistor lines.  
\end{proof}

\subsubsection{Semi-simple objects}
\label{semi simple objects}
A Levi splitting $M \to P$ induces a section $\Twistor_M \to \Twistor_P$ of the map $\Twistor_P \to \Twistor_M$ and allows us to take $G$-inductions of twistor lines on $\Twistor_M$ and obtain twistor lines on $\Twistor_G$. A \textit{Levi reduction} of a twistor line $\sigma : \PP^1 \to \Twistor_G$ with respect to a Levi splitting $M \to P$ is then a twistor line $\sigma_M : \PP^1 \to \Twistor_M$ that recovers $\sigma$ by $G$-inducing. We then make the following definition. 

\begin{definition} A twistor line $\sigma : \PP^1 \to \Twistor_G$ is called \textit{semi-simple} if, for any reduction $\sigma_P$ to a parabolic $P \subset G$, there exists a Levi reduction $\sigma_M$ of $\sigma$ with respect to some Levi splitting $M \to P$.  
\end{definition} 
The substack $\Twistor_G^{\irred} \subset \Twistor_G$ parametrises \textit{irreducible} objects that do not admit a $P$-reduction. If $\sigma : \PP^1 \to \Twistor_G$ factors through $\Twistor_G^{\irred}$ then we call $\sigma$ irreducible. If $\sigma$ is irreducible then it is automatically semi-simple. 

The following are consequences of Lemma \ref{associated}, as per \cite[Cory. 3.6.9, Cory. 3.6.10]{AGKRRV}. 

\begin{proposition} \,
\label{associated iff equivalent}
\begin{enumerate}
    \item For two pairs $(P_1, \sigma_{M_1})$ and $(P_2, \sigma_{M_2})$, the $G$-induced twistor lines $\sigma_{G,1}$ and $\sigma_{G,2} : \PP^1 \to \Twistor_G$ are isomorphic if and only if $(P_1, \sigma_{M_1})$ and $(P_2, \sigma_{M_2})$ are associated. 

    \item Association defines an equivalence relation on pairs $(P, \sigma_M)$ such that $\sigma_M$ is irreducible. 
    
    \item The assignment $(P, \sigma_M) \mapsto \sigma_G$ of taking $G$-inductions along Levi splittings establishes a bijection between classes of association of pairs $(P, \sigma_M)$ for which $\sigma_M$ is irreducible and isomorphism classes of semi-simple twistor lines $\PP^1 \to \Twistor_G$.
\end{enumerate}
\end{proposition}

\begin{definition} Given a semi-simple twistor line $\sigma : \PP^1 \to \Twistor_G$, we call the associated irreducible twistor line $\sigma_M : \PP^1 \to \Twistor^{\irred}_M$ the \textit{Levi component of $\sigma$}. 
\end{definition}

By Proposition \ref{associated iff equivalent}, Levi components are well-defined up to the equivalence relation of association.  

\begin{remark}
\label{HN type of Levi component} 
Levi components lands in a fixed connected component of $\Twistor_M$ corresponding to some fixed $\chi \in \pi_1(G)$. Moreover the irreducible $\sigma_M$ has Harder--Narasimhan type $\mu^{triv}$ corresponding to the trivial filtration. Let $\nu^{triv} = (\chi, \mu^{triv})$. Then $\sigma_M$ can be considered as a map to the Harder--Narasimhan stratum
\[
\sigma_M : \PP^1 \to \Twistor^{\sst, \nu^{triv}}_M . 
\]
This change in notation will be useful for later Harder--Narasimhan analysis. 
\end{remark}

\subsubsection{Semi-simplification} Given two twistor lines $\sigma$ and $ss(\sigma)$ on $\Twistor_G$, we say that $ss(\sigma)$ is a \textit{semi-simplification} of $\sigma$ if
 \begin{itemize}
     \item $ss(\sigma)$ is semi-simple, 
     \item there exists a parabolic $P \subset G$ and $P$-reductions $\sigma_{P}$ and $ss(\sigma)_{P}$ such that the induced $M$-local systems with respect to the Levi quotient $P \to M$ are isomorphic.
 \end{itemize}
Every twistor line $\sigma : \PP^1 \to \Twistor_G$ admits a semi-simplification: if $\sigma$ is semi-simple, then take $ss(\sigma) = \sigma$; if $\sigma$ is not semi-simple, then in particular $\sigma$ is reducible, so there exists a minimal standard parabolic $P$ to which $\sigma$ can be reduced to $\sigma_{P}$, and take $ss(\sigma)$ to be induced from $\sigma_{P}$ via $P \to M \to P \to G$ for some Levi splitting of $M$. 

By this discussion alongside Lemma \ref{associated} and Proposition \ref{associated iff equivalent}(3) we have: 

\begin{corollary} 
Semi-simplification of $\sigma : \PP^1 \to 
\Twistor_G$ is well-defined up to isomorphism.
\end{corollary}

\subsubsection{Horizontal twistor lines} 

\begin{definition}
\label{horizontal}
An twistor line $\sigma : \PP^1 \to \Twistor_G$ is called a \textit{horizontal} if:
\begin{itemize}
    \item $\sigma$ is semi-simple, 
    \item $\sigma$ has a Levi component $\sigma_M : \PP^1 \to \Twistor^{\irred}_M$ such that descent to the coarse moduli space 
    \[
    [\sigma_M] : \PP^1 \to \Twistor_M^{\irred} \to \Twistor_M^{\coarse} , 
    \]
    defines a horizontal twistor line in the classical sense, as described in \ref{se: app to coarse}. 
\end{itemize}
We denote this class of maps by $\Horiz_G$. 
\end{definition}

See \ref{twistor space} for the construction of the coarse moduli space and of the map $\Twistor_M^{\irred} \to \Twistor_M^{\coarse}$. 

\begin{remark} 
\, 
\begin{enumerate}
\item Definition \ref{horizontal} is independent of the choice of Levi component because by Proposition \ref{associated iff equivalent}(3) all choices are isomorphic and descent to the coarse moduli space identifies isomorphism classes. 

\item By the process of \textit{semi-simplification}, any twistor line admits a semi-simplification that can be used to check the non-abelian Hodge condition of Definition \ref{horizontal}. In this sense, any twistor line can participate in the Definition. 

\item The definition of horizontality on $\Twistor_M^{\coarse}$ ensures a compatibility with the non-abelian Hodge correspondence provided by the equations
\[
\nah[\sigma_M(0)] = [\sigma_M(1)] = \ol{\nah}[\sigma_M(\infty)] . 
\]
\end{enumerate}
\end{remark}

\subsection{Twistor Wilson eigensheaves} 
\label{se: eigensheaves}
In de Rham geometric Langlands a \textit{Wilson eigensheaf} is a skyscraper sheaf supported at a geometric point of $\LocSys_G$. They are Langlands dual to the Hecke eigensheaves in $\DMod(\Bun_{G^{\vee}})$ and points of $\LocSys_G$ become eigenvalues in the spectral decomposition. In the twistor context we take the following objects to play the role of Wilson eigensheaves. 

\begin{definition} 
\label{Wils}
The category $\Wils^{\Tw}_G$ of \textit{twistor Wilson eigensheaves} is defined to be the essential image of the functor $\Horiz_G \to \IndCoh(\Twistor_G)$ that evaluates the pushforwards $\sigma \mapsto \sigma_{*}\Oo_{\PP^1}$.
\end{definition}

\subsubsection{de Rham components}
\label{se: de Rham components}
Objects of $\Wils^{\Tw}_G$ can be thought of as \textit{'hyperk\"ahler rotations'} of skyscraper sheaves on $\LocSys_G$. To recover the purely de Rham objects one can pullback along the closed immersion $\LocSys_G \to \Twistor_G$ represented as the fiber of $\Twistor_G \to \PP^1$ over $1 \in \PP^1$. To each $\sigma_{*}\Oo_{\PP^1} \in \Wils^{\Tw}_G$ we obtain the forgetful assignment 
\[
\sigma_{*}\Oo_{\PP^1} \mapsto (\sigma \times_{\PP^1} \{ 1 \})_{*} \Oo_{pt} .  
\]
The local system $\sigma_{\dR} := \sigma \times_{\PP^1} \{ 1 \} \in \LocSys_G$ is semi-simple in the sense defined in \cite{AGKRRV}. This class of local systems is sufficient for the purposes of eigensheaf analysis, for the singular support category $\DMod_{\Nn}(\Bun_{G^{\vee}})$ contains all Hecke eigensheaves and splits as direct sum $ \oplus_{\sigma} \DMod_{\Nn}(\Bun_{G^{\vee}})_{\sigma}$ indexed by isomorphism classes of semi-simple $G$-local systems \cite[(0.9)]{AGKRRV}. This is the underlying motivation for why we believe it is sufficient to define twistor Wilson eigensheaves to be supported on semi-simple twistor lines -- so the eigenvalues in our theory are a priori semi-simple. 

\subsubsection{GAGA condition} Let us recall the category of GAGA sheaves 
\[
\IndCoh^{\GAGA}(\Twistor_G) = \IndCoh^{\GAGA}(\Hodge^{\an}_G) \times_{\IndCoh(\Rep_G \times \GG_m)} \IndCoh^{\GAGA}(\ol{\Hodge}^{\an}_G) ,
\]
as defined in \eqref{GAGA on Tw_G}.  
\begin{proposition}
\label{Wils is GAGA}
Every object of $\Wils^{\Tw}_G$ is an object of $\IndCoh^{\GAGA}(\Twistor_G)$. 
\end{proposition}
The statement is a consequence of the following general analysis of twistor lines. 

The key ingredient is the moduli $\LocSys_G^{\restr}$ of local systems with restricted variation from \cite{AGKRRV}, which itself was defined to analyse points on $\LocSys_G$ and their Hecke eigensheaves. We recall from \ref{algebraic construction} that restricted variation gives rise to an algebraic twistor stack $\Twistor^{\restr, \alg}_G$ whose analytification $\Twistor_G^{\restr}$ is a substack of $\Twistor_G$. 

We use this geometry to prove the following structural property of twistor lines. 
\begin{proposition}
\label{Tw lines algebraic}
Every twistor line $\PP^1 \to \Twistor_G$ factors as 
\[
\PP^1 \to \Twistor_G^{\restr} \to \Twistor_G , 
\]
where $\PP^1 \to \Twistor_G^{\restr}$ is the analytification of an algebraic twistor line $\PP^1 \to \Twistor_G^{\restr, \alg}$.  
\end{proposition}

\begin{proof} The idea of the proof is to surround $\sigma$ in a formal neighbourhood in which the Riemann-Hilbert correspondence acts algebraically. 

Let us first note that the universal property of analytification provides an identification 
\[
\Map_{\dAnSt}(\PP^1, \Twistor_G^{\restr}) = \Map_{\dSt}(\PP^1, \Twistor_G^{\restr, \alg}),
\]
and so any map $\PP^1 \to \Twistor_G^{\restr}$ can be represented as an analytification. It therefore remains to show the existence of the factorisation. Given a twistor line $\sigma : \PP^1 \to \Twistor_G$, consider the diagram 
\[
\begin{tikzcd}
\LocSys_G^{\restr} \times \GG_m \arrow[r] & \LocSys_G \times \GG_m \arrow[d, "\cong"] \\
\PP^1 - \{ 0, \infty \} \arrow[u, dashed, "\tau^{0, \infty}"] \arrow[d] \arrow[r, "\sigma^{0, \infty}"] & \Twistor \times_{\PP^1} \GG_m \arrow[d] \\
\PP^1 \arrow[r, "\sigma"] & \Twistor_G
\end{tikzcd}, 
\]
where $\sigma^{0, \infty}$ denotes the restriction of $\sigma$ away from the Dolbeault fibers. 

Since $\LocSys^{\restr}_G \to \LocSys_G$ is bijective on $\CC$-points \cite[\textsection 4.1.3]{AGKRRV} it follows that $\sigma^{0, \infty}$ factors through a morphism $\tau^{0, \infty} : \PP^1 - \{ 0, \infty \} \to \LocSys_G^{\restr} \times \GG_m$ . The factorisation extends to the $0$ and $\infty$ fibers because $\Twistor^{\restr}_G \to \Twistor_G$ is the identity on this locus, i.e. 
\[
\Twistor^{\restr}_G \times_{\PP^1} \{ 0 \} = \Twistor_G \times_{\PP^1} \{ 0 \} = \Higgs_G, 
\]
\[
\Twistor^{\restr}_G \times_{\PP^1} \{ \infty \} = \Twistor_G \times_{\PP^1} \{ \infty \} = \ol{\Higgs}_G . 
\]
See also Remark \ref{Dol restr}. This allows us to conclude that $\sigma$ factors through $\tau : \PP^1 \to \Twistor_G^{\restr}$ and this concludes the proof. \end{proof}

\subsubsection{Orthogonality} 

We show that twistor Wilson eigensheaves satisfy an orthogonality condition.   

\begin{proposition} 
\label{disjointness}
Given two objects $\Ww_1$ and $\Ww_2$ in the monoidal category $(\Wils^{\Tw}_G, \otimes)$ we have 
\[
\Ww_1 \otimes \Ww_2 \cong 
\begin{cases} 
      \Ww_1 & \text{ if } \Ww_1 \cong \Ww_2 , \\
      0 & \text{ if } \Ww_1 \not\cong \Ww_2 .
   \end{cases}
\]
\end{proposition}
 
\begin{proof} 
The first conditional follows immediately from the fact that $\Ww_i$ are of rank one. 

For the second conditional, consider $\Ww_i = (\sigma_i)_{*}\Oo_{\PP^1}$, $i = 1, 2$, such that $\Ww_i$ are not isomorphic. Then $\Ww_i$ are supported on non-isomorphic objects $\sigma_i \in \Horiz_G$, so by Proposition \ref{associated iff equivalent}(3), the associated Levi components $(\sigma_{i})_{M_i}$ of the semi-simple twistor lines $\sigma_{i}$ are also not isomorphic. 

We then proceed by a case by case argument on the Levis $M_1$ and $M_1$. 

If $M_1 \neq M_2$, then $\sigma_1$ and $\sigma_2$ are automatically disjoint, and so $\Ww_1 \otimes \Ww_2 = 0$

If $M_1 = M_2$, then after descent, $[(\sigma_{i})_{M_i}]$ are distinct horizontal twistor lines on classical twistor space $\Twistor_G^{\coarse}$. It follows from the non-abelian Hodge condition that if $[(\sigma_{i})_{M_i}]$ are distinct then they must be disjoint. This is a classical property of horizontal twistor lines on the twistor space of hyperk\"ahler manifolds -- see \cite[Thm. 3.3]{HKLR}. But since the equivalence classes $[(\sigma_{i})_{M_i}]$ are disjoint, then so are $(\sigma_{i})_{M_i}$, and so once more $\Ww_1 \otimes \Ww_2 = 0$. 
\end{proof}

This orthogonality property of $\Wils^{\Tw}_G$ can be viewed, on the one hand, as an analogue of the classical twistor theory fact that distinct horizontals twistor lines are disjoint \cite{HKLR}. On the other hand, we think of orthogonality as justification for considering objects of $\Wils^{\Tw}_G$ to play the role of Wilson eigensheaves in \textit{`twistor geometric Langlands'}, for in any given geometric Langlands theory, a category of eigensheaves should always be orthogonal with respect to a monoidal structure. For instance, in de Rham geometric Langlands, this is immediately true for Wilson eigensheaves, and proven for Hecke eigensheaves in \cite[Cory. 14.3.8]{AGKRRV}.

\subsection{Hyperholomorphic branes} 
\label{se: BBB-branes}

We define and study categories of \textit{ind-coherent BBB-branes} on $\Twistor_G$. The definition identifies a subcategory of $\IndCoh(\Twistor_G)$ cut out by prescribed conditions upon intersection with objects from $\Wils^{\Tw}_G$. 

\subsubsection{BBB-branes} 
\label{BBB-branes}
The definition is recalled and modified\footnote{The modification happens at the level of $\Horiz_G$ and subsequently in the definition of $\Wils^{\Tw}_G$. In this article $\Horiz_G$ is defined by a horizontal condition on twistor lines obtained from a non-abelian Hodge condition on the Levi reduction of a semi-simple twistor line. In \cite{FH1}, the conditions are imposed on any reduction to a reductive subgroup. The modification used here is motivated by compatibility with parabolic induction and the role of semi-simple local systems in eigensheaf decompositions of the de Rham theory, as per \cite{AGKRRV} (see also \ref{se: de Rham components}).} from \cite{FH1}. 

\begin{definition}
\label{BBB-brane}
The full subcategory $\IndCoh^{\BBB}(\Twistor_G) \subset \IndCoh^{\GAGA}(\Twistor_G)$ of \textit{BBB-branes} consists of objects $\Bb \in \IndCoh^{\GAGA}(\Twistor_G)$ such that, for every $\Ww \in \Wilson^{\Tw}_G$, there exists a graded vector space $V^{\bullet}_{\Ww} \in \QCoh(pt)$ and an equivalence 
\[
\Bb \otimes \Ww \cong V^{\bullet}_{\Ww} \otimes_{\CC} \Ww.
\]
\end{definition}
\begin{remark}
Definition \ref{BBB-brane} is constructed within the categories $\IndCoh^{\GAGA}(\Twistor_G)$ of \textit{GAGA sheaves} defined in \ref{GAGA sheaves Tw_G} in order to consider a more rigid class of analytic sheaves. This allows us to extend certain algebraic sheaf theory properties to the analytic topology via analytification. See for instance Lemma \ref{th eis Higgs} where this idea is used explicitly to induce analytic cuspidal-Eisenstein decompositions from their algebraic counterparts. 
\end{remark}

\begin{remark}
When verifying Definition \ref{BBB-brane} we refer to a given $\Ww \in \Wils^{\Tw}_G$ as a `test object'.
\end{remark}

\subsubsection{Zerobranes} An immediate class of BBB-branes is given by twistor Wilson eigensheaves. 
\begin{proposition}
$\Wilson^{\Tw}_G$ is naturally a subcategory of $\IndCoh^{\BBB}(\Twistor_G)$. 
\end{proposition}

\begin{proof}
This is a reinterpretation of Proposition \ref{Wils is GAGA} and Proposition \ref{disjointness}. In the former it is shown that $\Wils^{\Tw}_G$ is a subcategory of $\IndCoh^{\GAGA}(\Twistor_G)$. In the latter, it is shown that, given $\Ww_1 \in \Wils^{\Tw}_G$ and a test object $\Ww_2 \in \Wils^{\Tw}_G$, there exists an isomorphism 
\[
\Ww_1 \otimes \Ww_2 \cong V^{\bullet}_{\Ww_2} \otimes_{\CC} \Ww_2 ,
\]
for the following Kronecker delta type choices: when $\Ww_1 \cong \Ww_2$ take $V^{\bullet}_{\Ww_2} = \CC$ and when $\Ww_1 \not\cong \Ww_2$ take $V^{\bullet}_{\Ww_2} = 0$. This choice of $V^{\bullet}_{\Ww_2}$ describes $\Ww_1$ as an object of $\IndCoh^{\BBB}(\Twistor_G)$. 
\end{proof}

The realisation of $\Wilson^{\Tw}_G$ as a subcategory of $\IndCoh^{\BBB}(\Twistor_G)$ resembles the \textit{zerobranes} studied by Kapustin and Witten \cite[\textsection 8.2]{kapustin&witten}, in which skyscraper sheaves supported on geometric points of the Dolbeault moduli spaces are equipped with a BBB-brane structure and studied physically as Wilson eigensheaves, or as \textit{`magnetic eigenbranes'} in the terminology of Kapustin--Witten. 

\subsubsection{$\GG_m$-action} We show that one can rotate the hyperholomorphic structures exhibited by the categories $\Wils_G^{\Tw}$ and $\IndCoh^{\BBB}(\Twistor_G)$ via the natural $\GG_m$-action on $\Twistor_G$ inherited from deformation to the normal cone, as defined in \ref{GG_m spectral}. 

\begin{proposition}
The $\GG_m$-action preserves $\Wils^{\Tw}_G$ and $\IndCoh^{\BBB}(\Twistor_G)$.  
\end{proposition}

\begin{proof}
Given $\mu \in \GG_m$ let $\mu : \Twistor_G \to \Twistor_G$ also denote the corresponding automorphism. 

For the first statement consider $\Ww = \sigma_{*}\Oo_{\PP^1} \in \Wils_G^{\Tw}$ supported on $\sigma \in \Horiz_G$. Let $\sigma_M \in \Horiz_M$ denote an irreducible Levi component corresponding to the semi-simple hypothesis on $\sigma$. Then $\mu \cdot \sigma_M$ is an irreducible Levi component of $\mu \cdot \sigma$. Since $\sigma_M$ is in $\Horiz_M$, $[\sigma_M] : \PP^1 \to \Twistor_M^{\coarse}$ is a classical horizontal twistor line, and the same is true for $\mu \cdot [\sigma_M] = [\mu \cdot \sigma_M]$. Then, $[\mu \cdot \sigma_M]$ being horizontal implies that $\mu \cdot \sigma$ is an object of $\Horiz_G$. It follows that $\mu_{*} \Ww = (\mu \cdot \sigma)_{*}\Oo_{\PP^1}$ is indeed an object of $\Wils_G^{\Tw}$. 

For the second statement consider $\Bb \in \IndCoh^{\BBB}(\Twistor_G)$ and a test object $\Ww \in \Wils_G^{\Tw}$. Consider the equivalence  
\[
(\mu_{*}\Bb) \otimes \Ww \cong \mu_{*}(\Bb \otimes (1/\mu)_{*}\Ww ). 
\]
By the first statement, the sheaf $(1/\mu)_{*}\Ww$ is an object of $\Wils_G^{\Tw}$. The BBB-brane hypothesis on $\Bb$ therefore yields an equivalence  
\[
\Bb \otimes (1/\mu)_{*}\Ww \cong V^{\bullet} \otimes_{\CC} (1/\mu)_{*}\Ww, 
\]
for some $V^{\bullet} \in \QCoh(pt)$. Comparing this with the previous equation yields 
\[
\mu_{*} \Bb \otimes \Ww 
\cong \mu_{*}( V^{\bullet} \otimes_{\CC} (1/\mu)_{*}\Ww ) 
\cong V^{\bullet} \otimes_{\CC} \Ww, 
\]
and so $V^{\bullet}_{\Ww} = V^{\bullet}$ provides the required BBB-brane property on $\mu_{*}\Bb$. 
\end{proof}

\subsubsection{BBB-branes strata-by-strata} Recall the Harder--Narasimhan strata $\Twistor_P^{\sst, \nu} \to \Twistor_G$, defined in \ref{HN strata}, indexed by pairs $(P, \nu)$ consisting of a parabolic $P \subset G$ and a Harder--Narasimhan type $\nu$. Let $\Twistor_G^{P, \nu} \subset \Twistor_G$ denote the image of each strata. For a sheaf $\Bb$ on $\Twistor_G$, let $\Bb^{P, \nu}$ denote the restriction to $\Twistor_G^{P, \nu}$. 

We show that the BBB-brane condition can be verified strata-by-strata. 
\begin{lemma}
\label{BBB and HN strata}
Fix $\Bb \in \IndCoh(\Twistor_G)$. Then $\Bb$ moreover lies in $\IndCoh^{\BBB}(\Twistor_G)$ if and only if $\Bb^{P, \nu}$ lies in $\IndCoh^{\BBB}(\Twistor^{P, \nu}_G)$ for every pair $(P, \nu)$. 
\end{lemma}

\begin{proof}
If $\Bb$ is a BBB-brane then each $\Bb^{P, \nu}$ is immediately a BBB-brane. We now consider the converse. Assume that $\Bb^{P, \nu}$ is a BBB-brane for all $(P, \nu)$. Consider a test object $\Ww = \sigma_{*}\Oo_{\PP^1} \in \Wils^{\Tw}_G$ supported on $\sigma \in \Horiz_G$. By definition $\sigma$ is semi-simple, so admits an irreducible Levi component $\sigma_M$, which by Remark \ref{HN type of Levi component} can be considered as a map 
\[
\sigma_M : \PP^1 \to \Twistor^{\sst, \nu^{triv}}_M , 
\]
to a Harder--Narasimhan strata of trivial Harder--Narasimhan type $\nu^{triv}$. By Proposition \ref{associated iff equivalent}(3), $\sigma$ is recovered from $\sigma_M$ by $G$-inducing along a Levi splitting $M \to P \to G$. Thus $\sigma$ takes values in the image of $\Twistor_P^{\sst, \nu^{triv}} \to \Twistor_G$, which is precisely $\Twistor_G^{P, \nu^{triv}}$. By restriction to the support of $\Ww$ one therefore has 
\[
\Bb \otimes \Ww \cong \Bb^{P, \nu} \otimes \Ww . 
\]
By the BBB-brane hypothesis on $\Bb^{P, \nu}$, one has $V^{\bullet}_{\Ww} \in \QCoh(pt)$ and an equivalence 
\[
\Bb^{P, \nu} \otimes \Ww \cong V^{\bullet}_{\Ww} \otimes_{\CC} \Ww . 
\]
Thus we conclude that $\Bb$ is indeed an object of $\IndCoh^{\BBB}(\Twistor_G)$. 
\end{proof}

\section{Construction of geometric Eisenstein series functors}
\label{Twistors and geometric Eisenstein series}

\subsection{de Rham and Betti functors}
\label{se: Eis functors}
\label{Eis dR and B}
Following \cite{arinkin&gaitsgory} we recall the construction of geometric Eisenstein series on the spectral side of de Rham geometric Langlands. We repeat the construction verbatim in the Betti theory and relate the two via the Riemann--Hilbert correspondence. 

A parabolic subgroup $P \subset G$ with Levi quotient $M \to P$ gives rise to a pair of induction morphisms  
\begin{equation}
\label{de Rham p and q}
\begin{tikzcd}
 & \LocSys_{P} \arrow[dl, "q^{\dR}"'] \arrow[dr, "p^{\dR}"] &  \\
\LocSys_{M} &  & \LocSys_{G}
\end{tikzcd},
\end{equation}
where $q^{\dR}$ is quasi-smooth and affine and $p^{\dR}$ is schematic and proper. On an $S$-point $X_{\dR} \times S \to BP$ in $\Map(S, \LocSys_P)$, the morphism $q^{\dR}$ records the map $X_{\dR} \times S \to BP \to BM$, and similarly the morphism $p^{\dR}$ records $X_{\dR} \times S \to BP \to BG$. 

A pull-push composition defines the \textit{de Rham spectral geometric Eisenstein series functors}  
\begin{equation}
\label{dR Eis}
\Eis^{\dR}_{P} := (p^{\dR})_{*} \circ (q^{\dR})^{!} : \IndCoh( \LocSys_{M} ) \to \IndCoh( \LocSys_{G} ) , 
\end{equation}
viewed as a geometric incarnation of the parabolic induction functors $\Rep(M) \to \Rep(G)$. 

The same construction taken over the Betti diagram 
\begin{equation}
\label{Betti p and q}
\begin{tikzcd}
 & \Rep_{P} \arrow[dl, "q^{\B}"'] \arrow[dr, "p^{\B}"] &  \\
\Rep_{M} &  & \Rep_{G}
\end{tikzcd}
\end{equation}
defines the \textit{Betti spectral geometric Eisenstein series functors}  
\begin{equation}
\label{Betti Eis spec}
\Eis^{\B}_{P} :=  (p^{\B})_{*} \circ (q^{\B})^{!} : \IndCoh( \Rep_{M}) \to \IndCoh( \Rep_{G} ). 
\end{equation}
One of the basic properties of $\Eis^{\dR}_{P}$ and $\Eis^{\B}_{P}$ is the preservation of nilpotent singular support, as in \cite[Propn. 13.2.6]{arinkin&gaitsgory}, so that both functors act on the spectral sides of de Rham and Betti geometric Langlands. 

Moreover, $\Eis^{\dR}_{P}$ and $\Eis^{\B}_{P}$ inherit the natural left adjoint functors 
\[
\CT^{\dR}_{P} := (q^{\dR})_{!} \circ (p^{\dR})^{*} : \IndCoh(\LocSys_{G}) \to \IndCoh(\LocSys_{M}) , 
\]
\[
\CT^{\B}_{P} := (q^{\B})_{!} \circ (p^{\B})^{*} : \IndCoh(\Rep_{G}) \to \IndCoh(\Rep_{M}) , 
\]
known as the \textit{geometric constant term functors}.  

\subsubsection{$\Eis^{\dR}_P$  and $\Eis^{\B}_P$ comparison} 
\label{se: Eis RH}

We use the Riemann--Hilbert correspondence to compare the de Rham and Betti Eisenstein functors, which later will be interpreted as 'Betti gluing data' for a twistor variant of Eisenstein functor.  

We pass to the analytic topology by considering analytifications of the morphisms in \eqref{de Rham p and q} and \eqref{Betti p and q} and thus define the functors 
\begin{equation*}
\Eis^{\dR, \an}_{P} := (p^{\dR, \an})_{*} \circ (q^{\dR, \an})^{!} : \IndCoh( \LocSys^{\an}_{M} ) \to \IndCoh( \LocSys^{\an}_{G} ) ,
\end{equation*}
\begin{equation*}
\Eis^{\B, \an}_{P} := (p^{\B, \an})_{*} \circ (q^{\B, \an})^{!} : \IndCoh( \Rep^{\an}_{M} ) \to \IndCoh( \Rep^{\an}_{G} ) . 
\end{equation*}

Analytification give rise to the following natural algebraic to analytic comparison results. We recall the subcategories $\IndCoh^{\GAGA}(\Yy^{\an}) \subset \IndCoh(\Yy^{\an})$ of \textit{GAGA sheaves} has been defined, in \ref{GAGA sheaves}, to be the essential image of the analytification functors $(\cdot)^{\an} : \IndCoh(\Yy) \to \IndCoh(\Yy^{\an}).$

\begin{proposition}
\label{Eisenstein GAGA property} 
\textit{"Eisenstein-GAGA property".}
\begin{enumerate}
    \item $\Eis^{\dR, \an}_{P}$ restricted to $\IndCoh^{\GAGA}( \LocSys^{\an}_{M} )$ can be presented as the analytification  
    \[
    \Eis^{\dR, \an}_{P} \simeq (\Eis^{\dR}_{P})^{\an} .
    \]
    \item $\Eis^{\B, \an}_{P}$ restricted to $\IndCoh^{\GAGA}( \Rep^{\an}_{M} )$ can be presented as the analytification  
    \[
    \Eis^{\B, \an}_{P} \simeq (\Eis^{\B}_{P})^{\an} .
    \]
    \item $\Eis^{\dR, \an}_{P}$ and $\Eis^{\B, \an}_{P}$ preserve categories of GAGA sheaves, thus defining functors  
    \[
    \Eis^{\dR, \an}_{P} : \IndCoh^{\GAGA}( \LocSys^{\an}_{M} ) \to \IndCoh^{\GAGA}( \LocSys^{\an}_{G} ) , 
    \]
    \[
    \Eis^{\B, \an}_{P} : \IndCoh^{\GAGA}( \Rep^{\an}_{M} ) \to \IndCoh^{\GAGA}( \Rep^{\an}_{G} )  . 
    \]
\end{enumerate}
\end{proposition}

\begin{proof}
We prove (1). (2) follows similarly and (3) is an immediate consequence of (2) and (3). 

Given $\Ff \in \IndCoh( \LocSys_M )$ with analytification $\Ff^{\an} \in \IndCoh^{\GAGA}( \LocSys^{\an}_{M} )$, it suffices to consider the composition of equivalences 
\begin{align*}
\Eis^{\dR, \an}_{P}(\Ff^{\an}) & = (p^{\dR, \an})_{*}(q^{\dR, \an})^{!}\Ff^{\an} , \\
& \cong (p^{\dR, \an})_{*} \big((q^{\dR})^{!}\Ff)^{\an} \big)  , \\
& \cong \big( (p^{\dR})_{*}(q^{\dR})^{!}\Ff\big)^{\an} , \\
& = (\Eis^{\dR}_{P}(\Ff))^{\an} ,
\end{align*}
where the first equivalence follows from the functoriality of analytification and the second equivalence follows from a GAGA-type theorem of Porta-Yu for proper pushforward \cite[Thm 1.2]{porta&yu}. 
\end{proof}
We now compare analytic de Rham and Betti Eisenstein functors via the analytic Riemann--Hilbert correspondence of Porta \cite[Thm. 2]{porta_RH} and Holstein--Porta \cite[Thm. 1.5]{holstein&porta}. We recall that their approach to Riemann--Hilbert is defined by the $\anMap( \bullet , BG )$-pullback of a remarkable morphism $\nu_{\RH} : X^{\an}_{\dR} \to X^{\an}_{\B}$, which induces an equivalence of derived analytic stacks
\[
\RH := \nu_{\RH}^{*} : \Rep^{\an}_G \xrightarrow{\cong} \LocSys^{\an}_G.
\]
\begin{proposition} 
\label{RH and Eis}
The Riemann--Hilbert correspondence intertwines the analytic de Rham and Betti Eisenstein functors via a commutative square  
\[
\begin{tikzcd}
\IndCoh(\Rep_M^{\an}) 
\arrow[d, "\RH_{*}"] \arrow[r, "\Eis^{\B, \an}_P"] 
& \IndCoh(\Rep_G^{\an}) 
\arrow[d, "\RH_{*}"]
\\
\IndCoh(\LocSys_M^{\an}) 
\arrow[r, "\Eis^{\dR, \an}_P"] 
& 
\IndCoh(\LocSys_G^{\an})
\end{tikzcd} .
\]
\end{proposition}

\begin{proof}
Riemann--Hilbert acts by pre-composition by $\nu_{\RH}$ and the induction morphisms act by post-composition by either $BP \to BG$ or $BP \to BM$. Thus, one has a commutative diagram 
\begin{equation}
\label{RH q and p}
\begin{tikzcd}
\Rep_M^{\an} \arrow[d, "\RH"]
& \Rep_P^{\an} \arrow[l, "q^{\B, \an}"'] \arrow[r, "p^{\B, \an}"] \arrow[d, "\RH"]
& \Rep_G^{\an} \arrow[d, "\RH"]
\\
\LocSys^{\an}_{M}
& \LocSys^{\an}_{P} \arrow[l, "q^{\dR, \an}"'] \arrow[r, "p^{\dR, \an}"]
& \LocSys^{\an}_{G}
\end{tikzcd} , 
\end{equation}
and the result follows from natural functorialities determined by this diagram. 
\end{proof}

\subsection{Dolbeault and Hodge functors} 
\label{spectral Eis Hodge}
In this section we naturally extend the functors $\Eis^{\dR}_P$ to account for Hodge filtrations and their associated graded or Dolbeault degenerations.  

\subsubsection{Construction}
\label{se: Eis Dol/Hod}
One can define the Dolbeault and Hodge Eisenstein functors identically to the de Ram and Betti cases described in \ref{Eis dR and B}, where computing the $G$ and $M$-inductions of $P$-Higgs bundles and $(P, \lambda)$-connections once more defines a pair of maps 
\begin{equation}
\label{Dol Hod p and q}
\begin{tikzcd}
 & \arrow[dl, "q^{\Dol}"'] \Higgs_P \arrow[dr, "p^{\Dol}"] & \\
\Higgs_M & & \Higgs_G
\end{tikzcd}, 
\quad 
\begin{tikzcd}
 & \arrow[dl, "q^{\Hod}"'] \Hodge_P \arrow[dr, "p^{\Hod}"] & \\
\Hodge_M & & \Hodge_G
\end{tikzcd} , 
\end{equation}
and subsequently the \textit{Dolbeault and Hodge spectral geometric Eisenstein series functors} 
\[
\Eis^{\Dol}_{P} :=  (p^{\Dol})_{*} \circ (q^{\Dol})^{!} : \IndCoh( \Higgs_M) \to \IndCoh( \Higgs_{G} ) , 
\]
\[
\Eis^{\Hod}_{P} :=  (p^{\Hod})_{*} \circ (q^{\Hod})^{!} : \IndCoh( \Hodge_M) \to \IndCoh( \Hodge_{G} ) . 
\]
However, the morphisms $\Higgs_P \to \Higgs_G$ and $\Hodge_P \to \Hodge_G$ are not proper, so it is not immediately clear how one deals with coherent pushforward in this construction. To address this, we follow an idea of Laumon \cite{laumon_automorphic}, who studied geometric Eisenstein series over $\Bun_G$ strata-by-strata along the Harder--Narasimhan stratification.  In our situation the construction is as follows. Consider the Harder–Narasimhan strata $\Higgs_P^{\sst, \nu}$ and $\Hodge_P^{\sst, \nu}$, defined in \ref{HN strata}, indexed by pairs $\nu = (\chi, \mu)$, where $\chi$ parametrises the connected components of the moduli and $\mu$ is the Harder--Narasimhan type. One has a restriction of the diagram \eqref{Dol Hod p and q} given by  
\[
\begin{tikzcd}
 & \arrow[dl, "q^{\Dol}"'] \Higgs_P^{\sst, \nu} \arrow[dr, "p^{\Dol}"] & \\
\Higgs_M^{\sst, \nu} & & \Higgs_G
\end{tikzcd}, 
\quad 
\begin{tikzcd}
 & \arrow[dl, "q^{\Hod}"'] \Hodge_P^{\sst, \nu} \arrow[dr, "p^{\Hod}"] & \\
\Hodge_M^{\sst, \nu} & & \Hodge_G
\end{tikzcd} ,  
\]
where $\Higgs_M^{\sst, \nu} \to \Higgs_G$ and $\Hodge_M^{\sst, \nu} \to \Hodge_G$ are both proper morphisms. We consider the corresponding Eisenstein functors defined by
\[
\Eis^{\Dol}_{P, \nu} :=  (p^{\Dol})_{*} \circ (q^{\Dol})^{!} : \IndCoh( \Higgs^{\sst, \nu}_{M}) \to \IndCoh( \Higgs_{G} ) , 
\]
\[
\Eis^{\Hod}_{P, \nu} :=  (p^{\Hod})_{*} \circ (q^{\Hod})^{!} : \IndCoh( \Hodge^{\sst, \nu}_{M}) \to \IndCoh( \Hodge_{G} ) . 
\]
One can then recover the usual Eisenstein functors as filtered colimits 
\begin{equation}
\label{eq: Eis colim}
\Eis^{\Dol}_P = \lim_{\longleftarrow \, \nu}( \Eis^{\Dol}_{P, \nu} ) , 
\quad \quad 
\Eis^{\Hod}_P = \lim_{\longleftarrow \, \nu}( \Eis^{\Hod}_{P, \nu} ) , 
\end{equation}
relative to the natural ordering on the Harder--Narasimhan type. 
\begin{corollary}
\label{co: Eis compact}
Both $\Eis^{\Dol}_P$ and $\Eis^{\Hod}_P$ preserve compact objects, thus restricting to functors 
\[
\Eis^{\Dol}_{P} : \Coh( \Higgs_{M}) \to \Coh( \Higgs_{G} ) , 
\]
\[
\Eis^{\Hod}_{P} : \Coh( \Hodge_{M}) \to \Coh( \Hodge_{G} ) . 
\]
\end{corollary}

\begin{proof}
Both $\Eis^{\Dol}_{P,\nu}$ and $\Eis^{\Hod}_{P, \nu}$ preserve compact objects by proper pushforward along the Harder--Narasimhan strata $\Higgs_P^{\sst, \nu} \to \Higgs_G$ and $\Hodge_P^{\sst, \nu} \to \Hodge_G$. Moreover so do $\Eis^{\Dol}_{P}$ and $\Eis^{\Hod}_{P}$, as the Harder--Narasimhan stratification is locally finitely presented, and so the colimits \eqref{eq: Eis colim} locally stabilise. 
\end{proof}

\subsubsection{Analytic Hodge functors} Analytification of the induction maps $q^{\Hod}$ and $p^{\Hod}$ allow us to define the analytic pair of Hodge Eisenstein functors 
\[
\Eis^{\Hod, \an}_{P} : \IndCoh( \Hodge^{\an}_{M}) \to \IndCoh( \Hodge^{\an}_{G} ) ,
\]
\[
\Eis^{\Hod, \an}_{P, \nu} : \IndCoh( \Hodge^{\sst, \nu, \an}_{M}) \to \IndCoh( \Hodge^{\an}_{G} ) . 
\]
The following is a Hodge variant of Proposition \ref{Eisenstein GAGA property}. 
\begin{proposition}
\label{Hodge Eisenstein GAGA property} 
\textit{`Eisenstein-GAGA property'.} The Eisenstein functors $\Eis^{\Hod, \an}_{P}$ and $\Eis^{\Hod, \an}_{P}$ restricted respectively to $\IndCoh^{\GAGA}( \Hodge^{\sst, \nu, \an}_{M} )$ and $\IndCoh^{\GAGA}( \Hodge^{\an}_{M} )$ can be presented as the analytifications  
    \[
    \Eis^{\Hod, \an}_{P, \nu} \simeq (\Eis^{\Hod}_{P, \nu})^{\an} , \quad 
    \Eis^{\Hod, \an}_{P} \simeq (\Eis^{\Hod}_{P})^{\an} , 
    \]
and thus define a pair of functors  
    \[
    \Eis^{\Hod, \an}_{P, \nu} : \IndCoh^{\GAGA}( \Hodge^{\sst, \nu, \an}_{M} ) \to \IndCoh^{\GAGA}( \Hodge^{\an}_{G} ) , 
    \]
    \[
    \Eis^{\Hod, \an}_{P} : \IndCoh^{\GAGA}( \Hodge^{\an}_{M} ) \to \IndCoh^{\GAGA}( \Hodge^{\an}_{G} ) . 
    \]
\end{proposition}

\begin{proof}
The statements involving $\Eis^{\Hod, \an}_{P, \nu}$ follow from an identical calculation to Proposition \ref{Eisenstein GAGA property}, after an application of GAGA theory to proper pushforward along $\Hodge_P^{\sst, \nu} \to \Hodge_G$. Subsequently the statements involving $\Eis^{\Hod, \an}_{P}$ follow from the fact that $(\cdot)^{\an}$ preserves filtered colimits, so we have 
\begin{align*}
(\Eis^{\Hod}_{P})^{\an} & = (\lim_{\longleftarrow \nu} ( \Eis^{\Hod}_{P, \nu} ))^{\an} , \\
& \simeq \lim_{\longleftarrow \nu} ((\Eis^{\Hod}_{P, \nu})^{\an} ) , \\
& \simeq \underset{\longleftarrow \, \nu}{\lim}( \Eis^{\Hod, \an}_{P, \nu} ) , \\
& \simeq \Eis^{\Hod, \an}_{P} . \qedhere 
\end{align*}
\end{proof}

\subsection{Twistor functors} 
\label{twistor spectral p and q}
We now Eisenstein functors over the twistor stack $\Twistor_G$. The construction follows from the interplay between Eisenstein functors, the Hodge filtration and the Riemann--Hilbert correspondence, as presented in \ref{se: Eis Dol/Hod} and \ref{se: Eis RH}.

Recall from \ref{se: twistor stack} the defining pushout square 
\begin{equation}
\label{pushout reminder}
\begin{tikzcd}
\Rep_G \times \GG_m 
\arrow[d, "\RH"] \arrow[r, "\ol{\RH}"] 
& \ol{\Hodge}_G \arrow[d] \\
\Hodge_{G} \arrow[r] & \Twistor_G
\end{tikzcd} .   
\end{equation}
This construction inherits a pair of pushout morphisms  
\begin{equation*}
\begin{tikzcd}
 & \Twistor_{P}  \arrow[dl, "q^{\Hod, \an} \sqcup_{\Betti} \ol{q}^{\Hod, \an} =: q^{\Tw}"'] \arrow[dr, "p^{\Tw} := p^{\Hod, \an} \sqcup_{\Betti} \ol{p}^{\Hod, \an}"] &  \\
\Twistor_{M} & & \Twistor_{G} 
\end{tikzcd}, 
\end{equation*}
that represent the $G$ and $M$-inductions of $P$-twistor structures.  
\begin{proposition}
\label{spec tw eis}
The twistor geometric Eisenstein series functor 
\[
\Eis^{\Tw}_{ P} := (p^{\Tw})_{*} \circ (q^{\Tw})^{!} \simeq \Eis^{\Hod, \an}_{ P} \, \times_{\Betti} \, \ol{\Eis}^{\Hod, \an}_{ P},
\]
acting as 
\[
\Eis^{\Tw}_{ P} : \IndCoh( \Twistor_{M} ) \to \IndCoh( \Twistor_{G} ),
\]
is well-defined.
\end{proposition}

\begin{proof} We confirm the universal property for the existence of $\Eis^{\Hod, \an}_{ P} \, \times_{\Betti} \, \ol{\Eis}^{\Hod, \an}_{ P}$ acting on  
\[
\IndCoh(\Twistor_M) = \IndCoh(\Hodge^{\an}_M) \times_{\Betti} \IndCoh(\ol{\Hodge}^{\an}_M) , 
\]
as determined by the pushout \eqref{pushout reminder}. A sheaf $\Ff \in \IndCoh(\Twistor_M)$ determines the data
\[
\Ff_{\Hod} \in \IndCoh(\Hodge^{\an}_M), \quad \ol{\Ff}_{\Hod} \in \IndCoh(\ol{\Hodge}^{\an}_M), 
\]
\[
f_{\Hod} : \RH_{*}(\Ff_{\Hod}|_{\GG_m}) \xrightarrow{\cong} \ol{\RH}_{*} (\ol{\Ff}_{\Hod}|_{\GG_m}), 
\]
where this formula involves the restriction functors $(\cdot)|_{\GG_m}$ to the substacks 
\[
\Hodge^{\an}_{M} \times_{\AA^1} \GG_m \subset  \Hodge^{\an}_{M}, \quad \ol{\Hodge}^{\an}_{M} \times_{\AA^1} \GG_m \subset \ol{\Hodge}^{\an}_{M} .
\]
To confirm the universal property, we must show that $f_{\Hod}$ induces an equivalence
\begin{equation}
\label{univ property}
\RH_{*} \circ \Eis^{\Hod}_{ P} (\Ff_{\Hod}|_{\GG_m}) \cong \ol{\RH}_{*} \circ \ol{\Eis}^{\Hod}_{P} (\ol{\Ff}_{\Hod}|_{\GG_m}). 
\end{equation}
By application of the $\GG_m$-equivariant equivalences $\LocSys^{\an}_G \times \GG_m \cong \Hodge^{\an}_G \times_{\AA^1} \GG_m$, as in \eqref{Hodge trivialisation}, \eqref{univ property}, this is equivalent to a de Rham valued condition on the restrictions 
\[
\Ff_{\dR} \in \IndCoh(\LocSys^{\an}_M), \quad \ol{\Ff}_{\dR} \in \IndCoh(\ol{\LocSys}^{\an}_M), 
\]
\[
f_{\dR} : \RH_{*}\Ff_{\dR} \xrightarrow{\cong} \ol{\RH}_{*} \ol{\Ff}_{\dR} .
\]
It then suffices to show that $f_{\dR}$ induces an equivalence 
\[
\RH_{*} \circ \Eis^{\dR}_{ P} (\Ff_{\dR}) \cong \ol{\RH}_{*} \circ \ol{\Eis}^{\dR}_{P} (\ol{\Ff}_{\dR}) , 
\]
as objects in $\IndCoh(\Rep_G \times \GG_m) = \IndCoh(\ol{\Rep}_G \times \GG_m)$. By Proposition \ref{RH and Eis}, we may commute Eisenstein functors with Riemann--Hilbert, so the claim is equivalent to establishing an equivalence  
\begin{equation}
\label{checking pullback}
\Eis^{\B}_{ P} \circ \RH_{*} (\Ff_{\dR}) \cong \ol{\Eis}^{\B}_{P} \circ \ol{\RH}_{*}  (\ol{\Ff}_{\dR}) .
\end{equation}
The proof concludes by noticing that the Betti invariance property $\Rep_G = \ol{\Rep}_G$ determines an identification $\Eis^{\B}_{ P} = \ol{\Eis}^{\B}_{ P}$ and so the equivalence \eqref{checking pullback} is naturally induced by $f_{\dR}$. 
\end{proof}

We now restrict this construction to the Harder--Narasimhan strata $\Twistor^{\sst, \nu}_{P}$ from \ref{HN strata}. The pullback of Proposition \ref{spec tw eis} to the diagram 
\begin{equation*}
\begin{tikzcd}
 & \Twistor^{\sst, \nu}_{P} \arrow[dl, "q^{\Tw}"'] \arrow[dr, "p^{\Tw}"] &  \\
\Twistor^{\sst, \nu}_{M} & & \Twistor_{G} 
\end{tikzcd}, 
\end{equation*}
immediately yields the following restriction of $\Eis^{\Tw}_{P}$. 

\begin{corollary}
\label{spec tw eis HN}
The twistor Eisenstein functor 
\[
\Eis^{\Tw}_{P, \nu} := (p^{\Tw})_{*} \circ (q^{\Tw})^{!} \simeq \Eis^{\Hod, \an}_{P, \nu} \, \times_{\Betti} \, \ol{\Eis}^{\Hod, \an}_{P, \nu} 
\]
acting as 
\[ 
\Eis^{\Tw}_{P, \nu} : \IndCoh( \Twistor_{M}^{\sst, \nu} ) \to \IndCoh( \Twistor_{G} ),
\]
is well-defined.
\end{corollary}

\subsubsection{Constant term functors} As in the de Rham theory, $\Eis_{P}^{\Tw}$ and $\Eis_{P, \nu}^{\Tw}$ admit natural left adjoints, the \emph{twistor constant term functors}, defined to be 
\[
\CT^{\Tw}_{P} := (q^{\Tw})_{!} \circ (p^{\Tw})^{*} : \IndCoh(\Twistor_{G}) \to \IndCoh(\Twistor_{M}) , 
\]
\[
\CT^{\Tw}_{P, \nu} := (q^{\Tw})_{!} \circ (p^{\Tw})^{*} : \IndCoh(\Twistor_{G}) \to \IndCoh(\Twistor^{\sst, \nu}_{M}) .
\]

\subsubsection{Action on GAGA sheaves} Recall from \ref{GAGA sheaves Tw_G} the pullback square 
\begin{equation}
\label{eq: GAGA reminder}
\begin{tikzcd}
\IndCoh^{\GAGA}(\Twistor_G) \arrow[d] \arrow[r] & \IndCoh^{\GAGA}(\ol{\Hodge}^{\an}_G) \arrow[d] \\
\IndCoh^{\GAGA}(\Hodge^{\an}_G) \arrow[r] & \IndCoh(\Rep^{\an}_G \times \GG_m)
\end{tikzcd} , 
\end{equation} 
which defines the category $\IndCoh^{\GAGA}(\Twistor_G)$ of GAGA sheaves on $\Twistor_G$. 

Proposition \ref{spec tw eis} provides the defining presentation 
\[
\Eis^{\Tw}_{P, \nu} = \Eis^{\Hod, \an}_{P, \nu} \times_{\Eis^{\B, \an}_{P, \nu}} \ol{\Eis}^{\Hod, \an}_{P, \nu} .
\]
By Proposition \ref{Hodge Eisenstein GAGA property}, the Hodge components $\Eis^{\Hod, \an}_{P, \nu}$ and $\ol{\Eis}^{\Hod, \an}_{P, \nu}$ both preserve GAGA sheaves, and thus, acting on the pullback \eqref{eq: GAGA reminder}, we are able to conclude the same is true for $\Eis^{\Tw}_{P, \nu}$. 
\begin{corollary}
\label{GAGA Eis^{Tw}}
$\Eis^{\Tw}_{P, \nu}$ sends $\IndCoh^{\GAGA}( \Twistor_{M}^{\sst, \nu} )$ to $\IndCoh^{\GAGA}( \Twistor_{G} )$, thus defining a functor 
    \[
    \Eis^{\Tw}_{P, \nu} : \IndCoh^{\GAGA}( \Twistor_{M}^{\sst, \nu} ) \to \IndCoh^{\GAGA}( \Twistor_{G} )  . 
    \]
\end{corollary}

\subsubsection{Action on nilpotent sheaves} We now show that $\Eis^{\Hod}_{P}$ and $\Eis^{\Tw}_{P, \nu}$ preserve the categories of sheaves with nilpotent singular support, understood with respect to the Hodge and twistor nilpotent cones defined in \ref{Hodge nilpotent cone} and \ref{twistor nilp} respectively. 

\begin{proposition}
\label{nilp to nilp}
"Nilpotent-Eisenstein properties". 
\begin{itemize}
\item $\Eis^{\Hod}_{P}$ preserves nilpotent singular support, thus defining a functor
\[
\Eis^{\Hod}_{P} : \IndCoh_{\Nn}(\Hodge_M) \to \IndCoh_{\Nn}(\Hodge_G).
\]
    
\item $\Eis^{\Tw}_{P}$ preserves GAGA sheaves with nilpotent singular support, thus defining a functor 
\[
\Eis^{\Tw}_{P, \nu} : \IndCoh^{\GAGA}_{\Nn}(\Twistor_M) \to \IndCoh^{\GAGA}_{\Nn}(\Twistor_G).
\]
\end{itemize}
\end{proposition}

The proof presented below is identical to the corresponding result for $\Eis_{P}^{\dR}$ provided by Arinkin and Gaitsgory in \cite[Propn. 13.2.6]{arinkin&gaitsgory}. We restate the main ideas here for convenience and later reference.  

It will be temporarily useful to introduce subscripts of the form $\Nn_G \subset \Sing(\Hodge_G)$ to help keep track of the structure group.

We require the following analysis of the \textit{singular codifferentials} (see \ref{singularities of stacks}) of the maps
\begin{equation*}
\begin{tikzcd}
 & \Hodge_{P}  \arrow[dl, "q^{\Hod}"'] \arrow[dr, "p^{\Hod}"] &  \\
\Hodge_{M} & & \Hodge_{G} 
\end{tikzcd}.
\end{equation*}

\begin{lemma} 
\label{N_M to N_P}
The singular codifferential 
\[
\Sing(q^{\Hod}) : \Sing(\Hodge_{M})_{\Hodge_{P}} \to \Sing(\Hodge_{P})
\]
acts on the global nilpotent cone $\Nn_M \subset \Sing(\Hodge_{M})$ via a canonical isomorphism  
\[
\Sing(q^{\Hod})\big( \Nn_M \times_{\Hodge_M} \Hodge_P \big) = \Nn_P. 
\]
\end{lemma}
\begin{proof}
$\Nn_M$ and $\Nn_P$ parametrise adjoint-valued sections valued in the respective nilpotent elements $\Nilp(\pP^{*}) \subset \pP^{*}$ and $\Nilp(\mM^{*}) \subset \mM^{*}$ (see \ref{Hodge nilpotent cone}) . The singular codifferential acts via extension of structure group along the projection $\pP^{*} \to \mM^{*}$ that quotients by the unipotent radical $\uU^{*} \subset \pP^{*}$. The statement of the lemma then follows from the fact that $\Nilp(\pP^{*}) = \Nilp(\mM^{*}) \ltimes \uU^{*}$ so the projection sends $\Nilp(\pP^{*})$ to $\Nilp(\mM^{*})$, and moreover, $\Nilp(\pP^{*})$ is the inverse image of $\Nilp(\mM^{*})$.   
\end{proof}

By similar reasoning we prove: 

\begin{lemma} 
\label{N_P to N_G}
The singular codifferential 
\[
\Sing(p^{\Hod}) : \Sing(\Hodge_{G})_{\Hodge_{P}} \to \Sing(\Hodge_{P}),
\]
gives rise to a containment  
\[
\Sing(p^{\Hod})^{-1}( \Nn_{P}) \subset \Nn_{G} \times_{\Hodge_G} \Hodge_P . 
\]
\end{lemma}

\begin{proof}
Follows verbatim from \cite[Propn. 13.2.6]{arinkin&gaitsgory}, which reduces to the following statement: if an element $a \in \gG^{*}$ is such that its projection to $\pP^{*}$ is nilpotent, then $a$ itself is nilpotent. 
\end{proof}

We are now able to prove that $\Eis_{P}^{\Hod}$ and $\Eis_{P}^{\Tw}$ preserve nilpotent singular support. 

\begin{proof} (\textit{Of Proposition \ref{nilp to nilp}}). By Lemma \ref{N_M to N_P} the map $q^{\Hod}$ satisfies the hypothesis of \cite[Lem. 8.4.2]{arinkin&gaitsgory} and so $(q^{\Hod})^{!}$ preserves nilpotent singular support. By Lemma \ref{N_P to N_G} the map $p^{\Hod}$ satisfies the hypothesis of \cite[Lem. 8.4.5]{arinkin&gaitsgory} and so $(p^{\Hod})_{*}$ also preserves nilpotent singular support. Combining these statements, we conclude that $\Eis^{\Hod}_{P}$ restricts to a functor 
\begin{equation}
\label{Eis^Hod nilp}
\Eis^{\Hod}_{P} : \IndCoh_{\Nn}(\Hodge_M) \to \IndCoh_{\Nn}(\Hodge_G) , 
\end{equation}
which proves the first statement. 

For the second statement, first recall from Proposition \ref{Hodge Eisenstein GAGA property} that $\Eis^{\Hod, \an}_{P}$ acting on GAGA sheaves satisfies the GAGA-type equivalence 
\begin{equation}
\label{}
(\Eis^{\Hod}_{P})^{\an} \simeq \Eis^{\Hod, \an}_{P}.
\end{equation}
Also by Corollary \ref{an and GAGA} we have the essentially surjective analytification functor 
\[
\IndCoh_{\Nn}(\Hodge_G) \to \IndCoh^{\GAGA}_{\Nn}(\Hodge^{\an}_G), 
\]
so the analytification of \eqref{Eis^Hod nilp} restricted to the strata labelled by $\nu$ yields a functor 
\[
\Eis^{\Hod, \an}_{P} : \IndCoh^{\GAGA}_{\Nn^{\an}}(\Hodge_M^{\an}) \to \IndCoh^{\GAGA}_{\Nn^{\an}}(\Hodge_G^{\an}). 
\]
We claim this description of $\Eis^{\Hod, \an}_{P}$ is sufficient to conclude that $\Eis^{\Tw}_{P}$ also preserves nilpotent singular support. 

Indeed, by Proposition \ref{spec tw eis}, $\Eis^{\Tw}_{P}$ is equivalent to the pullback $\Eis^{\Hod, \an}_{P} \times_{\Eis^{\B, \an}_{P}}\ol{\Eis}^{\Hod, \an}_{P}$ acting on 
\[
\IndCoh^{\GAGA}_{\Nn}(\Twistor_M) = \IndCoh^{\GAGA}_{\Nn}(\Hodge^{\an}_M) \times_{\Betti} \IndCoh^{\GAGA}_{\Nn}(\Hodge^{\an}_M) , 
\]
and so the fact that $\Eis^{\Hod, \an}_{P}$ and $\ol{\Eis}^{\Hod, \an}_{P}$ preserves nilpotent singular support implies that $\Eis^{\Tw}_{P}$ also preserves singular support. 
\end{proof}
 
\subsubsection{Action on $\GG_m$-equivariant sheaves} We remark that $\Eis^{\Tw}_P$ is compatible with the $\GG_m$-equivariant structure obtained from deformation to the normal cone. Recall from \ref{NAH Hodge} and \ref{GG_m spectral} that the structural map $\Hodge_G \to \AA^1$ is $\GG_m$-equivariant with respect to dilations and extends to $\Twistor_G$ in a manner such that $\Twistor_G \to \PP^1$ is also $\GG_m$-equivariant. Since $q^{\Hod}$, $p^{\Hod}$, $q^{\Tw}$, $p^{\Tw}$ are also $\GG_m$-equivariant maps, it follows that the action of an element $\mu \in \GG_m$ induces the equivalences  
\[
\mu^{*} \circ \Eis^{\Hod}_{P} \simeq \Eis^{\Hod}_{P} \circ \mu^{*} , 
\quad 
\mu^{*} \circ \Eis^{\Tw}_{P} \simeq \Eis^{\Tw}_{P} \circ \mu^{*} ,
\]
and so both functors preserve the categories $\IndCoh^{\GG_m}(\cdot)$ of $\GG_m$-equivariant sheaves. 
\begin{corollary}
\label{pr eis and equi}
The Hodge and twistor Eisenstein series functors induce well-defined functors
\[
\Eis^{\Hod}_{ P}: \IndCoh^{\GG_m}(\Hodge_{M}) \to \IndCoh^{\GG_m}(\Hodge_{G}), 
\]
\[
\Eis^{\Tw}_{ P}: \IndCoh^{\GG_m}(\Twistor_{M}) \to \IndCoh^{\GG_m}(\Twistor_{G}) .
\]
\end{corollary}

\subsubsection{Transitivity property} 
The following is a standard property of Eisenstein functors. Its proof is a short calculation that we include for completeness. To state the property neatly, it is helpful to apply a change in our notation: given a parabolic $P \subset G$ with Levi $M$, let us temporarily denote the twistor Eisenstein functor by $\Eis^{\Tw}_{M \rightarrow G} : \IndCoh(\Twistor_M) \to \IndCoh(\Twistor_G)$.  

\begin{proposition}
\label{pr: transitivity}
Let $P \subset G$ be a parabolic with Levi $M$, and let $P' \subset M$ be a parabolic with Levi $M'$. The corresponding twistor Eisenstein functors satisfy the relation 
\[
\Eis^{\Tw}_{M' \rightarrow G} \simeq \Eis^{\Tw}_{M \rightarrow G} \circ \Eis^{\Tw}_{M' \rightarrow M} , 
\]
where $\Eis^{\Tw}_{M' \rightarrow G}$ is defined by viewing $P'$ as a parabolic of $G$ via a choice of Levi splitting $M \to P$. 
\end{proposition}

The formula and proof is identical in the case of de Rham, Betti, Dolbeault or Hodge Eisenstein functors. 

\begin{proof}
The Levi splitting induces an equivalence 
\[
\Twistor_{P'} \cong \Twistor_{P'} \times_{\Twistor_M} \Twistor_P , 
\]
and so the formula follows by base change and functoriality with respect to the commutative diagram 
\[
\begin{tikzcd}[column sep = 0.5cm, row sep = 0.4cm]
 & & \Twistor_{P'} \arrow[d, "\cong"] \arrow[dddll, bend right] \arrow[dddrr, bend left] & & \\
 & & \Twistor_{P'} \times_{\Twistor_M} \Twistor_{P} \arrow[dr] \arrow[dl] & & \\
 & \Twistor_{P'} \arrow[dl] \arrow[dr] & & \Twistor_P \arrow[dl] \arrow[dr] & \\
\Twistor_{M'} & & \Twistor_M & & \Twistor_G 
\end{tikzcd} . \qedhere
\]
\end{proof}

\begin{remark} The same calculation in the de Rham theory is mentioned in \cite[\textsection 0.2.3]{braverman&gaitsgory}. The main result of loc. cit. is that the same formula holds for compactified Eisenstein functors over $\Bun_G$.  
\end{remark} 

\subsection{Hyperholomorphic action} 
\label{se: hyperholomorphic induction}

This section is dedicated to showing that $\Eis^{\Tw}_{ P}$ and $\CT^{\Tw}_{P}$ preserve the categories of ind-coherent BBB-branes introduced in Definition \ref{BBB-brane}. One could say that BBB-branes therefore pass the \textit{`Eisenstein test'} for a reasonable category to be considered as one side of a geometric Langlands correspondence. 

\subsubsection{Parabolic induction for eigensheaves} We first prove some preliminary results on horizontal twistor lines -- the objects of $\Horiz_G$ as per Definition \ref{horizontal}, and therefore on twistor Wilson eigensheaves -- the objects of $\Wils^{\Tw}_G$ as per Definition \ref{Wils}. The results are almost immediate consequences of the formalism of semi-simple twistor lines introduced in \ref{semi simple objects}. 

\begin{lemma} 
\label{induction eigensheaves}
Let $P \subset G$ be a parabolic with Levi $M$ and choose a Levi splitting $M \to P$. Take $\sigma_M : \PP^1 \to \Twistor_M$ to be a twistor line with $G$-induction $\sigma_G : \PP^1 \to \Twistor_G$. Then we have: 
\begin{enumerate}
    \item If $\sigma_M$ is semi-simple, then $\sigma_G$ is semi-simple. 

    \item If $\sigma_M$ is an object of $\Horiz_M$, then $\sigma_G$ is an object of $\Horiz_G$.

    \item If $\Ww_M := (\sigma_M)_{*}\Oo_{\PP^1}$ is an object of $\Wils^{\Tw}_M$, then $\Ww_G := (\sigma_G)_{*}\Oo_{\PP^1}$ is an object of $\Wils^{\Tw}_G$, and moreover there exists an equivalence
    \[
    \CT^{\Tw}_P(\Ww_G) \cong \Ww_M  .
    \]
\end{enumerate}
We also have the following partial converse. Take $\sigma_G : \PP^1 \to \Twistor_G$ to be a semi-simple twistor line with a fixed choice of Levi component $\sigma_M : \PP^1 \to \Twistor_M^{\irred}$. Then we have: 
\begin{enumerate}
    \item[(4)] If $\sigma$ lies in $\Horiz_G$ then $\sigma_M$ lies in $\Horiz_M$.
    
    \item[(5)] If $\sigma_{*}\Oo_{\PP^1}$ lies in $\Wils_G^{\Tw}$ then $(\sigma_M)_{*}\Oo_{\PP^1}$ lies in $\Wils_M^{\Tw}$. 
        \end{enumerate}
\end{lemma}
\begin{proof}
(1) follows from the transitivity of parabolic subgroups. To see this, let us take $\sigma_M$ to be semi-simple, so by Proposition \ref{associated iff equivalent}(3), there exists a parabolic $P' \subset M$, with Levi $M'$, such that $\sigma_M$ admits an irreducible Levi component $\sigma_{M'}$. This means, for some Levi splitting $M' \to P'$, $\sigma_{M}$ is the $M$-induction of $\sigma_{M'}$ along $M' \to P' \to M$ . It follows that $\sigma$ is the $G$-induction of $\sigma_{M'}$ along $M' \to P' \to M \to P \to G$, which is a Levi splitting for $P'$, when viewed as a parabolic subgroup of $G$. Thus $\sigma_{M'}$ defines a Levi component of $\sigma_G$, so once more by Proposition \ref{associated iff equivalent}(3), $\sigma_G$ is semi-simple. 

(2) follows from (1) and the fact that \textit{horizontal} is defined by conditions on an associated Levi component of a semi-simple twistor line, and as the proof of (1) shows, $\sigma_{G}$ and $\sigma_{M}$ share a Levi component. 

(3) follows from (2) and the following calculation. The Levi splitting $M \to P$ induces a section $r^{\Tw} : \Twistor_M \to \Twistor_P$ of the map $q^{\Tw} : \Twistor_P \to \Twistor_M$. This allows us to factor $\sigma_G$ as the composition $p^{\Tw} \circ r^{\Tw} \circ \sigma_M$. Then, the counit $(p^{\Tw})^{*}(p^{\Tw})_{!} \to \id$ induces an isomorphism
\begin{align*}
\CT^{\Tw}_{P}( \sigma_{*}\Oo_{\PP^1} ) & = (q^{\Tw})_{!}(p^{\Tw})^{*} (p^{\Tw})_{*} (r^{\Tw})_{*} (\sigma_{M})_{*} \Oo_{\PP^1}, \\
& \cong (q^{\Tw})_{!} (r^{\Tw})_{*} (\sigma_{M})_{*} \Oo_{\PP^1} , \\
& = (\sigma_{M})_{*} \Oo_{\PP^1} . 
\end{align*}
(4) follows from two observations: firstly that being an object of $\Horiz_M$ is, by definition, a condition on Levi components; secondly that $\sigma_M$ can be taken as its own Levi component. 

(5) Is an immediate consequence of (4). 
\end{proof}

\subsubsection{Parabolic induction for BBB-branes} We are now ready for the main result of this section. 
 
\begin{proposition}
\label{Eis BBB to BBB}
For every parabolic subgroup $P \subset G$ with Levi $M$, the Eisenstein functor $\Eis^{\Tw}_{ P}$ sends $\IndCoh^{\BBB}(\Twistor_M)$ to $\IndCoh^{\BBB}(\Twistor_G)$, thus defining a functor 
\[
\Eis^{\Tw}_{ P} : \IndCoh^{\BBB}(\Twistor_M) \to \IndCoh^{\BBB}(\Twistor_G) .
\]
\end{proposition}
By Proposition \ref{nilp to nilp}, $\Eis^{\Tw}_{P}$ also preserves nilpotent singular support, and thus preserves the categories 
\[
\IndCoh^{\BBB}_{\Nn}(\Twistor_G) := \IndCoh^{\BBB}(\Twistor_G) \cap \IndCoh_{\Nn}(\Twistor_G),
\]
of \textit{nilpotent BBB-branes.} We record this conclusion as follows. 

\begin{corollary}
\label{BBB nilp Eis}
$\Eis^{\Tw}_{ P}$ sends $\IndCoh^{\BBB}_{\Nn}(\Twistor_M)$ to $\IndCoh^{\BBB}_{\Nn}(\Twistor_G)$, thus defining a functor 
\[
\Eis^{\Tw}_{ P} : \IndCoh^{\BBB}_{\Nn}(\Twistor_M) \to \IndCoh^{\BBB}_{\Nn}(\Twistor_G).
\]
\end{corollary}

Proposition \ref{Eis BBB to BBB} is, by Lemma \ref{BBB and HN strata}, equivalent to verifying the BBB-brane condition strata-by-strata along the Harder--Narasimhan stratification. It is therefore equivalent to proving the following.   

\begin{proposition}
\label{Eis BBB to BBB HN}
For every pair $(P ,\nu)$, the Eisenstein functor $\Eis^{\Tw}_{P, \nu}$ sends $\IndCoh^{\BBB}(\Twistor^{\sst, \nu}_M)$ to $\IndCoh^{\BBB}(\Twistor_G)$, thus defining a functor 
\[
\Eis^{\Tw}_{P, \nu} : \IndCoh^{\BBB}(\Twistor^{\sst, \nu}_M) \to \IndCoh^{\BBB}(\Twistor_G).
\]
\end{proposition}

\begin{proof} Fix an object $\Bb \in \IndCoh^{\BBB}(\Twistor^{\sst, \nu}_M)$. We verify the BBB-brane condition on $\Eis^{\Tw}_{P, \nu}(\Bb)$ against a test object $\Ww = \sigma_{*}\Oo_{\PP^1} \in \Wils^{\Tw}_G$ supported on $\sigma \in \Horiz_G$. That is to say, we show there exists a graded vector space $V^{\bullet}_\Ww \in \QCoh(pt)$ and an equivalence
\begin{equation}
\label{eigensheaf test}
\Eis^{\Tw}_{P, \nu}(\Bb) \otimes \Ww \cong V^{\bullet}_{\Ww} \otimes_{\CC} \Ww , 
\end{equation}
in $\IndCoh(\Twistor_G)$. The idea is to naturally induce \eqref{eigensheaf test} via parabolic induction.

By hypothesis $\sigma$ is semi-simple. Let $\sigma_{M'}$ be a Levi component of $\sigma$, for some parabolic $P' \subset G$ equipped with a Levi splitting $M' \to P'$. By Proposition \ref{associated iff equivalent}(3), $\sigma$ is recovered from $\sigma_{M'}$ by $G$-inducing along $M' \to P' \to G$, so $\sigma$ can be written as a composition 
\[
\PP^1 \xrightarrow{\sigma_{M'}} \Twistor_{M'} \xrightarrow{r^{\Tw, '}} \Twistor_{P'} \xrightarrow{p^{\Tw}} \Twistor_G , 
\]
where $r^{\Tw, '} : \Twistor_{M'} \to \Twistor_{P'}$ is the section of $q^{\Tw}$ induced by the Levi splitting. 

By Remark \ref{HN type of Levi component},1
$\sigma_{M'}$ defines a map $\sigma_{M'} : \PP^1 \to \Twistor^{\sst, \nu^{triv}}_{M'}$, where $\nu^{triv}$ indexes trivial Harder--Narasimhan type. In other words, since $\sigma_{M'}$ is irreducible, it parametrises objects with trivial Harder--Narasimhan filtration. Thus $\Ww = \sigma_{*}\Oo_{\PP^1}$ is supported on the image of $\Twistor_{P'}^{\sst, \nu^{triv}} \to \Twistor_G$, which is denoted by $\Twistor_G^{P', \nu^{triv}}$. 

The support of the expression \eqref{eigensheaf test} is contained within the substack $\Twistor_G^{P, \nu} \cap \Twistor_G^{P', \nu^{triv}}$, which is empty for $(P' , \nu^{triv}) \neq (P, \nu)$, and so \eqref{eigensheaf test} holds vacuously in this case with $V^{\bullet}_{\Ww} = 0$. 

We proceed in the remaining case where $(P' , \nu^{triv}) = (P, \nu)$. By Lemma \ref{induction eigensheaves}(3) we have
\[
\Ww_{M} := (\sigma_{M})_{*}\Oo_{\PP^1} \cong \CT_{P, \nu}^{\Tw}(\Ww) , 
\]
and moreover by Lemma \ref{induction eigensheaves}(5) we have that $\Ww_{M}$ lies in $\Wils^{\Tw}_{M}$. By hypothesis, $\Bb$ is an object of $\IndCoh^{\BBB}(\Twistor^{\sst, \nu}_M)$, so there exists a graded vector space $V^{\bullet}_{\Ww_{M}}$ and an equivalence 
\begin{equation}
\label{eq: BBB condition hypothesis}
\Bb \otimes \Ww_{M} \cong V^{\bullet}_{\Ww_{M}} \otimes_{\CC} \Ww_{M} , 
\end{equation}
in $\IndCoh(\Twistor^{\sst, \nu}_M)$. Then, $\Eis^{\Tw}_{P, \nu}$ applied to the left hand side of \eqref{eq: BBB condition hypothesis} has a natural surjection  
\begin{align*}
\Eis^{\Tw}_{P, \nu}(\Bb \otimes \Ww_{M} ) 
& \cong (p^{\Tw})_{*} ((q^{\Tw})^{!}\Bb \otimes (q^{\Tw})^{!}\Ww_M ) , \\
& \cong (p^{\Tw})_{*} ((q^{\Tw})^{!}\Bb \otimes (q^{\Tw})^{!} (q^{\Tw})_{*} (p^{\Tw})^{!} \Ww ) , \\
& \xrightarrow{(q^{\Tw})^{!}(q^{\Tw})_{*} \to \id} (p^{\Tw})_{*} ((q^{\Tw})^{!}\Bb \otimes (p^{\Tw})^{!} \Ww ) , \\
& \cong (p^{\Tw})_{*} ((q^{\Tw})^{!}\Bb \otimes (p^{\Tw})^{!} \Ww ) , \\
& \cong (p^{\Tw})_{*} (q^{\Tw})^{!}\Bb \otimes \Ww ) , \\
& = \Eis^{\Tw}_{P, \nu}( \Bb ) \otimes \Ww , 
\end{align*}
induced by the counit $(q^{\Tw})^{!}(q^{\Tw})_{*} \to \id$ along the cohomologically affine morphism $q^{\Tw}$. The 
$-1$-shifted cone of this morphism is given by 
\[
\Cone\Big( \Eis^{\Tw}_{P, \nu}(\Bb \otimes \Ww_{M} ) \to \Eis^{\Tw}_{P, \nu}( \Bb ) \otimes \Ww \Big)\big[-1\big] 
\cong \Eis^{\Tw}_{P, \nu}( \Bb ) \otimes \Ww^{\perp} , 
\]
where $\Ww^{\bot} := \Cone(\Eis^{\Tw}_{P, \nu}(\Ww_M) \to \Ww) [-1]$ of the natural surjection $\Eis^{\Tw}_{P, \nu}(\Ww_M) \to \Ww$. 

By a similar calculation, $\Eis^{\Tw}_{P, \nu}$ applied to the right hand side of \eqref{eq: BBB condition hypothesis} has a natural surjection 
\begin{align*}
\Eis^{\Tw}_{P, \nu}(V^{\bullet}_{\Ww_M} \otimes_{\CC} \Ww_M) 
& \cong V^{\bullet}_{\Ww_M} \otimes_{\CC} \Eis^{\Tw}_{P, \nu}(\Ww_M) \\
& \xrightarrow{ (q^{\Tw})^{!}(q^{\Tw})_{*} \to \id } V^{\bullet}_{\Ww_M} \otimes_{\CC} \Ww ,
\end{align*}
with $-1$-shifted cone $V^{\bullet}_{\Ww_M} \otimes \Ww^{\perp}$. We arrive at a pair of exact triangles
\[
\begin{tikzcd}
\Eis^{\Tw}_{P, \nu}(\Bb) \otimes \Ww^{\bot} \arrow[r] 
& 
\Eis^{\Tw}_{P, \nu}(\Bb \otimes \Ww_M) \arrow[r, "(q^{\Tw})^{!}(q^{\Tw})_{!} \to \id"] 
\hspace{5mm}
& \hspace{5mm}
\Eis^{\Tw}_{P, \nu}(\Bb) \otimes \Ww 
\\
V^{\bullet}_{\sigma_M} \otimes_{\CC} \Ww^{\bot} \arrow[r] \arrow[u, "\cong"]
& V^{\bullet}_{\sigma_M} \otimes_{\CC} \Eis^{\Tw}_{P, \nu}(\Ww_M) \arrow[u, "\cong"] \arrow[r, "(q^{\Tw})^{!}(q^{\Tw})_{!} \to \id"]
\hspace{5mm}
& \hspace{5mm} V^{\bullet}_{\sigma_M} \otimes_{\CC} \Ww \arrow[u, "\cong"]
\end{tikzcd} .
\]
The equivalence of the middle terms is induced functorially by $\Eis^{\Tw}_{P, \nu}$ acting on \eqref{eq: BBB condition hypothesis}. This in turn induces the displayed equivalence of the other two terms. For the equivalence on the right hand side, this follows via the functoriality of $(\cdot) \otimes \Ww$ and the presentations 
\[
\Eis_{P, \nu}^{\Tw}(\Bb) \otimes \Ww \cong \Eis^{\Tw}_{P, \nu}(\Bb \otimes \Ww_M) \otimes \Ww , 
\quad 
V^{\bullet}_{\sigma_M} \otimes_{\CC} \Ww \cong V^{\bullet}_{\sigma_M} \otimes_{\CC} \Eis^{\Tw}_{P, \nu}(\Ww_M) \otimes \Ww , 
\]
both of which follow from the relation $(q^{\Tw})^{!}\Ww_M \otimes (r^{\Tw})_{*} \Ww \cong (r^{\Tw})_{*}\Ww$. In other words, both of the displayed surjections induced by the counit can be presented by restriction to the subspace $\Supp(\Ww) \subset \Supp(\Eis^{\Tw}_{P, \nu}(\Ww_M))$, which is precisely $\Image(\sigma) \subset p(q^{-1}( \Image(\sigma_M))$. The left hand side equivalence then follows by functoriality of the cone construction. The right hand side equivalence $\Eis^{\Tw}_{P, \nu}(\Bb) \otimes \Ww \cong V^{\bullet}_{\Ww_M} \otimes_{\CC} \Ww$ allows us to set $V^{\bullet}_{\Ww} := V^{\bullet}_{\Ww_M} \in \QCoh(pt)$ and conclude that \eqref{eigensheaf test} holds for this choice of $V^{\bullet}_{\Ww}$. This concludes the proof. 
\end{proof}

\subsubsection{Parabolic reduction for BBB-branes} We now provide the adjoint version of Proposition \ref{Eis BBB to BBB}. 

\begin{proposition}
\label{CT BBB to BBB}
$\CT^{\Tw}_{ P}$ sends $\IndCoh^{\BBB}(\Twistor_G)$ to $\IndCoh^{\BBB}(\Twistor_M)$, thus defining a functor 
\[
\CT^{\Tw}_{ P} : \IndCoh^{\BBB}(\Twistor_G) \to \IndCoh^{\BBB}(\Twistor_M) .
\]
\end{proposition}
Once more, by Lemma \ref{BBB and HN strata}, it is equivalent to prove the following. 
\begin{proposition}
\label{CT BBB to BBB strata}
$\CT^{\Tw}_{P, \nu}$ sends $\IndCoh^{\BBB}(\Twistor_G)$ to $\IndCoh^{\BBB}(\Twistor^{\sst, \nu}_M)$, thus defining a functor 
\[
\CT^{\Tw}_{P, \nu} : \IndCoh^{\BBB}(\Twistor_G) \to \IndCoh^{\BBB}(\Twistor^{\sst, \nu}_M) .
\]
\end{proposition}

\begin{proof}

Fix $\Bb \in \IndCoh^{\BBB}(\Twistor_G)$. We verify the BBB-brane condition on $\CT_{P, \nu}^{\Tw}(\Bb) \in \IndCoh(\Twistor_M^{\sst, \nu})$ against a test object $\Ww = \sigma_{*}\Oo_{\PP^1} \in \Wils_M^{\Tw}$ supported on $\sigma \in \Horiz_M$. That is to say, we show there exists a graded vector space $V^{\bullet}_\Ww \in \QCoh(pt)$ and an equivalence
\begin{equation}
\label{eigensheaf test 2}
\CT^{\Tw}_{P, \nu}(\Bb) \otimes \Ww \cong V^{\bullet}_{\Ww} \otimes_{\CC} \Ww , 
\end{equation}
in $\IndCoh(\Twistor^{\sst, \nu}_M)$. The idea is to induce \eqref{eigensheaf test 2} via parabolic reduction. 

To each Levi splitting $M \xrightarrow{s} P$, denote by $\sigma^{s}_G$ the associated $G$-induction of $\sigma$, taken along the composition $M \xrightarrow{s} P \to G$. Thus $\sigma^{s}_G : \PP^1 \to \Twistor_G$ is defined to be the composition
\[
\PP^1 
\xrightarrow{\, \, \sigma \, \,} \Twistor_M 
\xrightarrow{r^{s}} \Twistor_P 
\xrightarrow{p^{\Tw}} \Twistor_G , 
\]
where $\Twistor_M \xrightarrow{r^{s}} \Twistor_P$ is the section of $\Twistor_P \xrightarrow{q^{\Tw}} \Twistor_M$ induced by $s$. By Lemma \ref{induction eigensheaves}, $\sigma_G^{s}$ is an object of $\Horiz_G$ for every $s$, and the correponding skyscraper sheaf 
\[
\Ww^{s}_G 
:= (\sigma^{s}_G)_{*}\Oo_{\PP^1} 
= (p^{\Tw})_{*}(r^{s})_{*} \sigma_{*}\Oo_{\PP^1}
= (p^{\Tw})_{*}(r^{s})_{*} \Ww , 
\]
is an object of $\Wils_G^{\Tw}$. Taking the union over all Levi splittings $s : M \to P$ yields an identification
\[
\Supp\big(\Eis^{\Tw}_{P, \nu}(\Ww)\big) = p^{\Tw}\big((q^{\Tw})^{-1}(\Image(\sigma))\big) = \cup_{s} \Image(\sigma_G^{s}) = \cup_{s} \Supp(\Ww_G^{s}) , 
\]
between the supports of length one skyscraper sheaves. Thus one has an equivalence 
\[
\Eis^{\Tw}_{P, \nu}(\Ww) \cong \colim_{s} (\Ww_G^{s}) . 
\]
Twisting this equivalence by $\Bb$ yields 
\begin{equation}
\label{eq: CT this}
\Bb \otimes \Eis^{\Tw}_{P, \nu}(\Ww) 
\cong \Bb \otimes \colim_{s}(\Ww_G^{s}) 
\cong \colim_{s} ( \Bb \otimes \Ww_G^{s} )  
\cong \colim_s ( V^{\bullet}_{\Ww_G^{s}} \otimes_{\CC} \Ww_G^{s} ), 
\end{equation}
where the final equivalence is the application of the BBB-brane hypothesis on $\Bb$ for each $\Ww_G^{s} \in \Wils^{\Tw}_G$, which yields a family $V^{\bullet}_{\Ww_G^{s}} \in \QCoh(pt)$. Then $\CT^{\Tw}_{P, \nu}$ applied to the left hand side of \eqref{eq: CT this} yields 
\[
\CT^{\Tw}_{P, \nu}\big(\Bb \otimes \Eis^{\Tw}_{P, \nu}(\Ww)
\big) \cong \CT^{\Tw}_{P, \nu}(\Bb) \otimes \Ww , 
\]
and $\CT^{\Tw}_{P, \nu}$ applied to the right hand side of \eqref{eq: CT this} yields
\[
\CT^{\Tw}_{P, \nu}\big(\oplus_s ( V^{\bullet}_{\Ww_G^{s}} \otimes_{\CC} \Ww_G^{s} )
\big) 
\cong \colim_s (\CT^{\Tw}_{P, \nu}(V^{\bullet}_{\Ww_G^{s}})) \otimes \Ww .  
\]
Thus $\CT^{\Tw}_{P, \nu}$ applied to \eqref{eq: CT this} induces the sought-after equivalence \eqref{eigensheaf test 2} for the choice of vector space 
\[
V^{\bullet}_{\Ww} 
:= \colim_s (\CT^{\Tw}_{P, \nu}(V^{\bullet}_{\Ww_G^{s}})) 
 \in \QCoh(pt) , 
\]
and this concludes the proof.  
\end{proof}

\section{Applications of parabolic induction I: cuspidal--Eisenstein decompositions}
\label{se: Hyperholomorphic applications}

\subsection{Decompositions in non-abelian Hodge theory}
\label{se structure}

We now present our main results. We describe how coherent sheaves with nilpotent singular support on moduli from non-abelian Hodge theory can be decomposed into Eisenstein and cuspidal components. We refer to the introduction for discussion on the profound influence of the seminal cuspidal-Eisenstein decompositions of Arinkin--Gaitsgory \cite{arinkin&gaitsgory}, whose work we extend and reference throughout the course of the proofs.  

\subsubsection{Statement of results} Let us recall that we have constructed the Eisenstein functors 
\begin{align*}
\Eis^{\Tw}_{P} : \IndCoh(\Twistor_M) & \to \IndCoh(\Twistor_G) , \\
\Eis^{\Hod}_{P} :  \IndCoh(\Hodge_M) & \to \IndCoh(\Hodge_G) , \\
\Eis^{\Dol}_{P} :  \IndCoh(\Higgs_M) & \to \IndCoh(\Higgs_G) , 
\end{align*}
such that all three preserve nilpotent singular support by Proposition \ref{nilp to nilp}, preserve compact objects by Corollary \ref{co: Eis compact}, and moreover $\Eis^{\Tw}_{P}$ preserves BBB-branes by Proposition \ref{Eis BBB to BBB HN}. Collectively these properties yield the Eisenstein functors used in the following structure theorems for coherent sheaves with nilpotent singular support in non-abelian Hodge theory. 

\begin{theorem}
\label{decompose BBB}
The category $\Coh^{\BBB}_{\Nn}(\Twistor_G)$ is generated by the essential image of 
\[
\Eis^{\Tw}_{P} : \Perf^{\BBB}( \Twistor_{M} ) \to \Coh^{\BBB}_{\Nn}( \Twistor_{G} ),
\]
for all parabolic subgroups $P \subset G$. 
\end{theorem}

\begin{theorem}
\label{th eis and generate} 
The category $\Coh^{\GAGA}_{\Nn}(\Twistor_{G})$ is generated by the essential image of 
\[
\Eis^{\Tw}_{P} : \Perf^{\GAGA}( \Twistor_{M} ) \to \Coh^{\GAGA}_{\Nn}( \Twistor_{G} ),
\]
for all parabolic subgroups $P \subset G$. 
\end{theorem}

\begin{theorem}
\label{th eis hodge} 
The category $\Coh_{\Nn}(\Hodge_{G})$ is generated by the essential image of 
\[
\Eis^{\Hod}_{P} : \Perf( \Hodge_{M} ) \to \Coh_{\Nn}( \Hodge_{G} ),
\]
for all parabolic subgroups $P \subset G$. 
\end{theorem}

\begin{theorem}
\label{th eis Higgs} 
The category $\Coh_{\Nn}(\Higgs_{G})$ is generated by the essential image of 
\[
\Eis^{\Dol}_{P} : \Perf( \Higgs_{M} ) \to \Coh_{\Nn}( \Higgs_{G} ),
\]
for all parabolic subgroups $P \subset G$. 
\end{theorem}

By the following logical dependencies we reduce the contents of this section to proving Theorem \ref{th eis hodge}. 

\begin{lemma}
\label{concluding decompositions} 
\, 
\begin{enumerate}
    \item Theorem \ref{th eis Higgs} follows from Theorem \ref{th eis hodge}, 
    \item Theorem \ref{th eis and generate} follows from Theorem \ref{th eis hodge}, 
    \item Theorem \ref{decompose BBB} follows from Theorem \ref{th eis and generate}. 
\end{enumerate}
\end{lemma}

The pivotal Theorem \ref{th eis hodge} will be proven during a sequence of results established in \ref{image of pullback} - \ref{se: cusp-eis}. For now let us prove the comparative Lemma.  

\begin{proof} \textit{Of Lemma \ref{concluding decompositions} (1)}. Follows from the commutativity of the diagram 
\[
\begin{tikzcd}
\Perf(\Hodge_M) \arrow[r, "\Eis^{\Hod}_{P}"] \arrow[d] & \Coh_{\Nn}(\Hodge_G) \arrow[d] \\
\Perf(\Higgs_M) \arrow[r, "\Eis^{\Dol}_{P}"] & \Coh_{\Nn}(\Higgs_G) 
\end{tikzcd} ,
\]
where the vertical arrows are essentially surjective forgetful functors defined by restriction along 
\[
\Higgs_M = \Hodge_M \times_{\AA^1} \{0\} \to \Hodge_M ,
\]
\[
\Higgs_G = \Hodge_G \times_{\AA^1} \{0\} \to \Hodge_G . \qedhere
\]
\end{proof}

\begin{proof}\textit{Of Lemma \ref{concluding decompositions} (2)}. We give a Hodge to twistor gluing argument. By considering the twistor decompositions
\[
\Coh^{\GAGA}_{\Nn}(\Twistor_G) = \Coh^{\GAGA}_{\Nn}(\Hodge_G) \times_{\Betti} \Coh^{\GAGA}_{\Nn}(\Hodge_G), 
\]
\[
\Eis^{\Tw}_{P}|_{\Perf^{\GAGA}(\Twistor_M)} = (\Eis^{\Hod}_{P})^{\an} \times_{\Betti} (\ol{\Eis}^{\Hod}_{P})^{\an} , 
\]
described in \ref{Nilp GAGA} and Proposition \ref{spec tw eis} respectively, the result follows from checking that $(\Eis^{\Hod}_{P})^{\an}$, and therefore symmetrically $(\ol{\Eis}^{\Hod}_{P})^{\an}$, generates the target category. We therefore prove the following statement: by hypothesis one assumes that the essential image of
\[
\Eis^{\Hod}_{P} : \Perf(\Hodge_M) \to \Coh_{\Nn}(\Hodge_G) , 
\]
for all $P$ generates the target category, then we check that the same is true for the analytifications 
\[
(\Eis^{\Hod}_{P})^{\an} : \Perf^{\GAGA}(\Hodge_M^{\an}) \to \Coh^{\GAGA}_{\Nn}(\Hodge^{\an}_G) .
\]
This statement is a consequence of the commutativity of the diagram
\[
\begin{tikzcd}
\Perf(\Hodge_M) \arrow[r, "\Eis^{\Hod}_{P}"] \arrow[d, "(\cdot)^{\an}"] & \Coh_{\Nn}(\Hodge_G) \arrow[d, "(\cdot)^{\an}"] \\
\Perf^{\GAGA}(\Hodge^{\an}_M) \arrow[r, "\, \, (\Eis^{\Hod}_{P})^{\an}"] & \Coh^{\GAGA}_{\Nn}(\Hodge^{\an}_G)
\end{tikzcd} ,
\]
provided by Proposition \ref{Hodge Eisenstein GAGA property}, in which the vertical arrows are the analytification functors landing on their essential image.
\end{proof}

For the proof of the final \textit{Part (3)} of the Lemma we recall the modified twistor Eisenstein functors 
\[
\Eis^{\Tw}_{P, \nu} : \IndCoh(\Twistor_M^{\sst, \nu}) \to \IndCoh(\Twistor_G) , 
\]
defined strata-by-strata over the Harder--Narasimhan stratification, as per Corollary \ref{spec tw eis HN}.  

\begin{proof} \textit{Of Lemma \ref{concluding decompositions} (3)}. The statement can be rephrased in the following notation. Let $\langle \Eis^{\Tw}_{P}(\Cc_{M}) \rangle_{P}$  and $\langle \Eis^{\Tw}_{P, \nu}(\Cc_{M, \nu}) \rangle_{P, \nu}$ denote the essential image generated $\Eis^{\Tw}_{P}$ by $\Eis^{\Tw}_{P, \nu}$ acting on a family of subcategories $\Cc_{M} \subset \IndCoh(\Twistor_M)$ and $\Cc_{M, \nu} \subset \IndCoh(\Twistor^{\sst, \nu}_M)$, such that $\Cc_{M, \nu}$ is the pullback of $\Cc_M$. Because $\Eis^{\Tw}_{P, \nu}$ is a left adjoint to $\CT^{\Tw}_{P , \nu}$, their essential images preserve colimits, and so the colimit \eqref{eq: Eis colim} over Harder--Narasimhan types $\nu$ induces the identification 
\[
\langle \Eis^{\Tw}_{P}(\Cc_{M}) \rangle_{P} = \langle \Eis^{\Tw}_{P, \nu}(\Cc_{M, \nu}) \rangle_{P, \nu} . 
\]
Interpreting Theorem \ref{th eis and generate} and Theorem \ref{decompose BBB} as the identifications 
\[
\big\langle \Eis^{\Tw}_{P}\big(\Perf(\Twistor_M)\big) \big\rangle_{P} = \Coh_{\Nn}(\Twistor_G) ,
\]
\[
\big\langle \Eis^{\Tw}_{P}\big(\Perf^{\BBB}(\Twistor_M)\big) \big\rangle_{P} = \Coh^{\BBB}_{\Nn}(\Twistor_G) ,
\]
the statement of Lemma \ref{concluding decompositions}(3) is that the former identification implies the existence of the latter.  

By Propositions \ref{Eis BBB to BBB} and \ref{Eis BBB to BBB}, the functors $\Eis^{\Tw}_{P}$ and $\Eis^{\Tw}_{P, \nu}$ preserve the BBB-brane condition, so there exists a pair of natural fully faithful inclusion functors 
\begin{equation}
\begin{tikzcd}
\label{fully faithful guy}
\big\langle \Eis^{\Tw}_{P}(\Coh_{\Nn}^{\BBB}(\Twistor_M) \big\rangle_{P} 
\arrow[r, "\Theta"] 
\arrow[d, equal]
& \Coh^{\BBB}_{\Nn}(\Twistor_G) \\
\big\langle \Eis^{\Tw}_{P, \nu}(\Coh_{\Nn}^{\BBB}(\Twistor^{\sst, \nu}_M) \big\rangle_{P, \nu} \arrow[ur, "\Theta_{HN}"']
\end{tikzcd}
\end{equation}
Fix $\Bb \in \Coh^{\BBB}_{\Nn}(\Twistor_G)$ and consider the restrictions $\Bb^{P, \nu}$ to the images of the morphisms $p^{\Tw} : \Twistor_P^{\sst, \nu} \to \Twistor_G$. The counit $(p^{\Tw})_{*}(p^{\Tw})^{!} \to \id$ induces an isomorphism 
\[
\Bb^{P, \nu} \cong (p^{\Tw})_{*}(p^{\Tw})^{!}\Bb .
\]
Recall that, by Proposition \ref{CT BBB to BBB}, the functors $\CT_{P, \nu}^{\Tw}$ preserve the BBB-brane condition, so the objects $\CT_{P, \nu}^{\Tw}(\Bb)$ lie in $\Coh^{\BBB}_{\Nn}(\Twistor^{\sst, \nu}_M)$, for every $(P, \nu)$. 

Moreover, the counits $(q^{\Tw})^{!}(q^{\Tw})_{!} \to \id$ and $(p^{\Tw})_{*}(p^{\Tw})^{*} \to \id$ induce a natural surjection 
\[
\Eis^{\Tw}_{P, \nu}(\CT^{\Tw}_{P, \nu}(\Bb^{P, \nu})) \to \Bb^{P , \nu} . 
\]
By varying $(P, \nu)$ this generates a surjection 
\[
\big\langle \Eis^{\Tw}_{P,\nu}(\CT^{\Tw}_{P, \nu}(\Bb^{P, \nu})) \big\rangle_{P, \nu} \to \Bb , 
\]
thus defining an essentially surjective functor 
\[
\Phi_{HN} : \big\langle \Eis^{\Tw}_{P, \nu} (\CT^{\Tw}_{P, \nu} (\Coh^{\BBB}_{\Nn}(\Twistor_G^{P, \nu}))) \big\rangle_{P ,\nu} \to \Coh^{\BBB}_{\Nn}(\Twistor_G) .
\]
Moreover $\Phi_{HN}$ and the inclusion $\Theta_{HN}$ from \eqref{fully faithful guy} form a commutative diagram  
\[
\begin{tikzcd}
\big\langle \Eis^{\Tw}_{P, \nu} ( \CT^{\Tw}_{P, \nu} (\Coh^{\BBB}_{\Nn}(\Twistor_G^{P, \nu})) \big\rangle_{P , \nu} \arrow[r, phantom, sloped, "\subset"] 
\arrow[d, two heads, "\Phi_{HN}"'] 
& \big\langle \Eis^{\Tw}_P(\Coh^{\BBB}_{\Nn}(\Twistor_M) \big\rangle_{P} \arrow[dl, hook', "\Theta_{HN}"] \\
\Coh^{\BBB}_{\Nn}(\Twistor_G) & 
\end{tikzcd} .
\]
It follows that $\Theta_{HN}$, and therefore $\Theta$, is essentially surjective, thus providing the identifications 
\[
\begin{tikzcd}
\big\langle \Eis^{\Tw}_P(\Coh^{\BBB}_{\Nn}(\Twistor_M) \big\rangle_{P} \arrow[d, equal] \arrow[r, equal] & \Coh^{\BBB}_{\Nn}(\Twistor_G) \\
\big\langle \Eis^{\Tw}_{P, \nu}(\Coh^{\BBB}_{\Nn}(\Twistor^{\sst, \nu}_M) \big\rangle_{P , \nu} \arrow[ru, equal] & 
\end{tikzcd}
\]
Induction on the semi-simple rank reduces the horizontal arrow to an identification 
\[
\big\langle \Eis^{\Tw}_P(\Perf^{\BBB}(\Twistor_M) \big\rangle_{P} = \Coh^{\BBB}_{\Nn}(\Twistor_G) ,
\]
and this concludes the proof. 
\end{proof}

\subsubsection{Harder--Narasimhan variant} We begin the proof of Theorem \ref{th eis hodge} by first reducing the statement to one involving strata-by-strata Eisenstein functors acting on coherent rather than perfect complexes. 

\begin{lemma} 
\label{le: TFAE}
The following are equivalent: 
\begin{enumerate}
    \item (\textit{The statement of Theorem \ref{th eis hodge}}). $\Coh_{\Nn}(\Hodge_{G})$ is generated by the essential image of 
\[
\Eis^{\Hod}_{P} : \Perf( \Hodge_{M} ) \to \Coh_{\Nn}( \Hodge_{G} ),
\]
for all parabolic subgroups $P \subset G$. 

\item $\Coh_{\Nn}(\Hodge_{G})$ is generated by the essential image of 
\[
\Eis^{\Hod}_{P} : \Coh_{\Nn
}( \Hodge_{M} ) \to \Coh_{\Nn}( \Hodge_{G} ),
\]
for all parabolic subgroups $P \subset G$. 

\item $\Coh_{\Nn}(\Hodge_{G})$ is generated by the essential image of 
\[
\Eis^{\Hod}_{P, \nu} : \Coh_{\Nn}( \Hodge^{\sst, \nu}_{M} ) \to \Coh_{\Nn}( \Hodge_{G} ),
\]
for all pairs $(P,\nu)$ of parabolic subgroups $P\subset G$ and Harder--Narasimhan types $\nu$.  
\end{enumerate}
\end{lemma}

\begin{proof}
(1) is equivalent to (2) via induction on the semi-simple rank, the transitivity of parabolic subgroups and the evident transitivity relations of $\Eis_P^{\Hod}$, as per Proposition \ref{pr: transitivity}. The de Rham variant of (1) $\iff$ (2) appears as \cite[Cory 13.3.9]{arinkin&gaitsgory} where the proof is identical. (2) is equivalent to (3) since $\Eis^{\Hod}_{P, \nu}$ is a left adjoint and so its essential image is preserved under colimits. 
\end{proof}

To prove Theorem \ref{th eis hodge} it therefore suffices to prove the statement in Lemma \ref{le: TFAE}(3) and this is the approach we shall follow. 

From here onwards let us introduce subscripts such as $\Nn_G \subset \Sing(\Hodge_G)$ on the Hodge theoretic global nilpotent cones to keep track of structure group. 

\subsubsection{Image of $(q^{\Hod})^{!}$}
\label{image of pullback}
Our analysis of the Eisenstein functors  
\[
\Eis^{\Hod}_{P, \nu} = (p^{\Hod})_{*} \circ (q^{\Hod})^{!} : \Coh_{\Nn_{M}}( \Hodge^{\sst, \nu}_{M} ) \to \Coh_{\Nn_{G}}( \Hodge_G ) , 
\]
begins with the pullback functor. 
\begin{lemma}
\label{pull image}
For each pair $(P, \nu)$ the essential image of the functor
\[
(q^{\Hod})^{!} : \Coh_{\Nn_{M}}( \Hodge^{\sst, \nu}_{M} ) \to \Coh_{\Nn_{P}}( \Hodge^{\sst, \nu}_{P} ) ,
\]
generates the target category.
\end{lemma}

\begin{proof} 
Follows as per \cite[Propn. 13.4.4]{arinkin&gaitsgory}. Indeed, Lemma \ref{N_M to N_P} shows that $q^{\Hod}$ satisfies the hypothesis of \cite[Propn. 8.4.14]{arinkin&gaitsgory} and thus $(q^{\Hod})^{!}$ induces an equivalence 
    \begin{align*}
    \Coh_{\Nn_{M}}( \Hodge^{\sst, \nu}_{M} ) \otimes_{ \Coh( \Hodge^{\sst, \nu}_{M} ) }  \Coh( \Hodge^{\sst, \nu}_{P} ) \xrightarrow{(q^{\Hod})^{!}} \Coh_{\Nn_{P}}( \Hodge^{\sst, \nu}_{P} ).
    \end{align*}
Therefore it suffices to show that 
    \[
    (q^{\Hod})^{!} : \Coh( \Hodge^{\sst, \nu}_{M} ) \to \Coh( \Hodge^{\sst, \nu}_{P} ) ,
    \]
generates the target, or equivalently that the adjoint $(q^{\Hod})_{*}$ is conservative. This follows from the fact that induction along the Levi quotient map $P \to M$ can be presented as a quotient of a schematic affine morphism with respect to the action of the unipotent group, and so the map $q^{\Hod} : \Hodge^{\sst, \nu}_{P} \to \Hodge^{\sst, \nu}_{M}$ is cohomologically affine. 
\end{proof}

\subsubsection{Image of $(p^{\Hod})_{*}$} The singular codifferential $\Sing(p_{\Hod})$ of the $G$-induction morphism $p_{\Hod} : \Hodge_P^{\sst, \nu} \to \Hodge_G$ induces a diagram 
\[
\begin{tikzcd}[column sep = 1.8cm]
\Nn_P \arrow[d]
& \Sing(p_{\Hod})^{-1}(\Nn_P) 
\arrow[r] \arrow[d] \arrow[l]
& \Nn_G^{P, \nu} \arrow[d] 
\\
\Sing(\Hodge_P^{\sst, \nu}) \, \, 
\arrow[dr]
& ^{cl}( \Sing(\Hodge_G) \times_{\Hodge_G} \Hodge_P^{\sst, \nu} ) 
\arrow[l, "\, \, \Sing(p_{\Hod}) \, \, \, \, "'] \arrow[d] \arrow[r]
& \Sing(\Hodge_G) \arrow[d]
\\
& \Hodge_{P}^{\sst, \nu} \arrow[r, "p_{\Hod}"] & \Hodge_G   
\end{tikzcd} , 
\]
where the upper squares are Cartesian by construction. Indeed, $\Nn_G^{P, \nu} \to \Sing(\Hodge_G)$ is defined to be the closed immersion obtained from the image of $\Sing(p_{\Hod})^{-1}(\Nn_P)$ along the canonical map 
\[
^{cl}( \Sing(\Hodge_G) \times_{\Hodge_G} \Hodge_P^{\sst, \nu} ) \to \Sing(\Hodge_G) . 
\]
By Lemma \ref{N_P to N_G} we have $\Nn_G^{P, \nu} \subset \Nn_G$ and so one obtains a full subcategory  
\[
\Coh_{\Nn_G^{P, \nu}}( \Hodge_{G} ) \subset \Coh_{\Nn_G}( \Hodge_{G} ), 
\]
of sheaves which have, by construction, their usual support contained within the image of $p_{\Hod} : \Hodge_P^{\sst, \nu} \to \Hodge_G$, which is denoted by $\Hodge^{P\, \nu}_G$. 

By functoriality of pushforward for singular support \cite[Proposition 8.4.19]{arinkin&gaitsgory} we obtain the following. 

\begin{lemma}
\label{push image}
Given a parabolic $P \subset G$ and Harder--Narasimhan type $\nu$ the essential image of the functor 
\[
(p^{\Hod})_{*} : \Coh_{\Nn_P}( \Hodge^{\sst, \nu}_{P} )  \to \Coh_{\Nn_G^{P, \nu}}( \Hodge_{G} ),
\]
generates the target category. 
\end{lemma}

\subsubsection{Image for fixed strata}

By combining the statements of Lemma \ref{pull image} and Lemma \ref{push image} we obtain: 
\begin{lemma} 
\label{Eis image fixed P}
For fixed $(P, \nu)$ the essential image of the Eisenstein functor 
\[\
\Eis_{P, \nu}^{\Hod} : \Coh_{\Nn_M}(\Hodge^{\sst, \nu}_{M}) \to \Coh_{\Nn_G}( \Hodge_{G} ) ,
\]
is given by the subcategory $\Coh_{\Nn_G^{P, \nu}}( \Hodge_{G} )$.
\end{lemma}

\subsubsection{Image for all strata}
\label{all strata}

We now consider the full Harder--Narasimhan stratification 
\[
\Hodge_G = \bigsqcup_{(P, \nu)} \Hodge_G^{P, \nu} , \quad \Nn^{\HN}_G := \bigsqcup_{(P, \nu)} \Nn_G^{P, \nu} . 
\]
By the properties of localisation described in \cite[Cory. 3.3.9]{arinkin&gaitsgory}, alongside Lemma \ref{Eis image fixed P} applied to all pairs $(P, \nu)$, we obtain the following computation of all Eisenstein-generated sheaves.

\begin{lemma}
\label{Eis gen coarse}
$\Coh_{\Nn_G^{\HN}}( \Hodge_{G} )$ is generated by the essential image of 
\begin{equation*}
\Eis^{\Hod}_{P, \nu} : \Coh_{\Nn_M}( \Hodge^{\sst, 
\nu}_{M} ) \to \Coh_{\Nn_G}( \Hodge_{G} ) ,
\end{equation*}
applied to all pairs $(P, \nu)$. 
\end{lemma}
Lemma \ref{Eis gen coarse} generates our category of interest, namely $\Coh_{\Nn_G}( \Hodge_{G} )$, by the following comparison between $\Nn_G^{\HN}$ and $\Nn_G$.  
\begin{lemma}
\label{gen by Ii}
One has an identification $\Coh_{\Nn_G^{\HN}}( \Hodge_{G} ) = \Coh_{\Nn_G}( \Hodge_{G} )$. Therefore, the category $\Coh_{\Nn_G}( \Hodge_{G} )$ is generated by the essential image of 
\begin{equation}
\label{image gen}
\Eis^{\Hod}_{P, \nu} : \Coh_{\Nn_M}( \Hodge^{\sst, 
\nu}_{M} ) \to \Coh_{\Nn_G}( \Hodge_{G} ) ,
\end{equation}
applied to all pairs $(P, \nu)$. 
\end{lemma}

\begin{proof}
To show the natural fully faithful inclusion  
\[
\Coh_{\Nn_G^{\HN}}( \Hodge_{G} ) \to \Coh_{\Nn_G}( \Hodge_{G} ) , 
\]
is moreover an equivalence, it suffices to show that the canonical immersion  
\[
\Nn_{G}^{\HN} := \bigsqcup_{(P, \nu)} \Nn_G^{P, \nu} \to \Nn_G , 
\]
is surjective at $k$-points, or in other words, that $\Nn_G^{P, \nu}$ surjects onto $\Nn_G$ when restricted to the image of $p^{\Hod} : \Hodge_P^{\sst, \nu} \to \Hodge_G$. Recall from \ref{Hodge nilpotent cone} that $k$-points of $\Nn_{G}$ are parametrised by the data 
\[
((E, \nabla^{\lambda}), A^{\lambda}) , 
\]
where $\lambda \in \AA^1$, $(E, \nabla^{\lambda})$ is a $\lambda$-connection on $X$ and $A^{\lambda} \in H^{0}(X, \ad(E, \nabla^{\lambda})^{\vee})$ is a \textit{nilpotent covector} ($A^{\lambda}$ is nilpotent when, for any local trivialisation, $A^{\lambda}$ takes values in $\Nilp(\gG^{*}) \subset \gG^{*}$). 

Thus, for any given $(E, \nabla^{\lambda}, A^{\lambda}) \in \Nn_G$ with $P$-reducible $(E, \nabla^{\lambda})$, it suffices to show there exists a $P$-reduction $(E_P, \nabla^{\lambda}_{P})$ such that $A^{\lambda}$ is sent to a nilpotent covector $A^{\lambda}_P \in H^{0}\Gamma(X, \ad(E_P, \nabla^{\lambda}_P)^{\vee
})$. 

We prove this statement via a case by case analysis on the deformation coordinate $\lambda$. 

For $\lambda = 1$, the statement is proven for the local system $(E, \nabla) = (E, \nabla^{1}) \in \LocSys_G$ by Arinkin and Gaitsgory in \cite[\textsection 13.4.7]{arinkin&gaitsgory}, using a fixed point theorem on the proper space of all horizontal $P$-reductions of $(E, \nabla)$.  

For $\lambda \neq 0$, the statement follows from the $\lambda = 1$ case by application of the $\GG_m$-equivariant isomorphism between the non-zero fibers of $\Hodge_G \to \AA^1$. 

For $\lambda = 0$, the Arinkin--Gaitsgory argument is no longer applicable, as the $P$-reduction spaces are no longer proper. Instead, we give a constructive argument. We show the canonical $P$-reduction $(E_P, \phi_P) \in \Higgs_P^{\sst, \nu}$ of $(E, \phi) = (E, \nabla^{0}) \in \Higgs_G$ provided by the Harder--Narasimhan filtration evidently receives a nilpotent covector $A_P$. 

Compatibility between the Harder--Narasimhan filtration and the adjoint representation goes back to Atiyah--Bott \cite{atiyah&bott}. In work of Biswas--Holla \cite{biswas&holla} it  is expressed as the following statement: the canonical reduction $(E_P, \phi_P)$ in the adjoint representation admits a short exact sequence 
\[
0 \to \ad(E_P, \phi_P) \to \ad(E, \phi) \to \ad(E_P, \phi_P)(\gG/\pP) \to 0 , 
\]
where $\ad(E_P, \phi_P)(\gG/\pP)$ is the $\gG/\pP$-induction of the vector bundle $\ad(E_P, \phi_P)$. Taking duals and global sections induces an exact sequence in cohomology that includes a map 
\[
H^{0}(X, \ad(E,\phi)^{\vee}) \to H^{0}(X, \ad(E_P,\phi_P)^{\vee}). 
\]
Define $A_P \in H^{0}(X, \ad(E_P,\phi_P)^{\vee})$ to be the image of $A$ under this map. Since $A$ takes values in $\Nilp(g^{*})$, it follows that $A_P$ takes values in $\Nilp(p^{*})$, because the map $g^{*} \to p^{*}$ preserves nilpotent elements. Therefore $(E_P, \phi_P, A_P)$ defines a geometric point of $\Nn_P$, supported over $\Higgs_P^{\sst, \nu}$, and this is the desired $P$-reduction of $(E, \phi, A)$.   
\end{proof}

By Lemma \ref{le: TFAE}, the proof of Lemma \ref{gen by Ii} concludes our proof of Theorem \ref{th eis hodge}. 

\subsubsection{Cuspidal-Eisenstein components} 
\label{se: cusp-eis}
Using our main results, we specify how to induce cuspidal--Eisenstein decompositions on coherent nilpotent sheaves in non-abelian Hodge theory. The constructions are standard and are repeated verbatim as per the de Rham theory \cite[\textsection 13.3.1]{arinkin&gaitsgory}. Let us introduce the notation 
\begin{align}
\begin{split}
\label{Hodge Eis component}
\Coh_{\Nn}( \Higgs_{G} )_{\Eis} & \subset \Coh_{\Nn}( \Higgs_{G} ) , \\ 
\Coh_{\Nn}( \Hodge_{G} )_{\Eis} & \subset \Coh_{\Nn}( \Hodge_{G} ) , \\
\Coh^{\GAGA}_{\Nn}( \Twistor_{G} )_{\Eis} & \subset \Coh^{\GAGA}_{\Nn}( \Twistor_{G} ) , \\
\Coh^{\BBB}_{\Nn}( \Twistor_{G} )_{\Eis} & \subset \Coh^{\BBB}_{\Nn}( \Twistor_{G} ) ,
\end{split}
\end{align}
for the full subcategories generated by the essential image of the functors $\Eis_P^{\Dol}$, $\Eis_P^{\Hod}$ and $\Eis_P^{\Tw}$, acting as specified in Theorems \ref{decompose BBB}, \ref{th eis and generate}, \ref{th eis hodge}, \ref{th eis Higgs}, but only for all proper parabolics $P \neq G$. These subcategories are known as \textit{Eisenstein components.} 

By Lemma \ref{concluding decompositions} one obtains the following relations between the Eisenstein components:  
\[
\Coh_{\Nn}( \Higgs_{G} )_{\Eis} =  \Coh_{\Nn}( \Hodge_{G} )_{\Eis} \times_{\AA^1} \{ 0 \} , 
\]
\[
\Coh^{\GAGA}_{\Nn}( \Twistor_{G} )_{\Eis} = (\Coh_{\Nn}( \Hodge_{G} )_{\Eis})^{\an} \times_{\Betti} (\Coh_{\Nn}( \ol{\Hodge}_{G} )_{\Eis})^{\an} , 
\]
\[
\Coh^{\BBB}_{\Nn}( \Twistor_{G} )_{\Eis} = \Coh^{\BBB}_{\Nn}( \Twistor_{G} ) \cap \Coh^{\GAGA}_{\Nn}( \Twistor_{G} )_{\Eis} . 
\]
In the Hodge case, the Eisenstein component can be described explicitly as follows. The other cases are identical. Let $\Hodge_G^{\red}$ denote the closed substack of $\Hodge_G$ given by the union of the images of $\Hodge_P^{\sst, \nu} \to \Hodge_G$, for all pairs $(P , \nu)$ such that $P \subset G$ is a proper parabolic. Define the pullback 
\[
\Nn_G^{\red} 
= \Nn_G \times_{\Hodge_G} \Hodge_G^{\red} 
= \bigsqcup_{(P, \nu) , P \neq G} \Nn_G^{P, \nu} .   
\]
By Lemma \ref{gen by Ii} applied to all proper parabolics one has the identification 
\[
\Coh_{\Nn_G}( \Hodge_{G} )_{\Eis} = \Coh_{\Nn^{\red}_G}( \Hodge_{G} ) ,
\]
of subcategories in $\Coh_{\Nn_G}( \Hodge_{G} )$. Now we describe the orthogonal complement. The complementary open substack to $\Hodge_G^{\red} \subset \Hodge_G$ is the moduli $\Hodge_G^{\irred}$ of \textit{irreducible} $(G, \lambda)$-connections: \textit{i.e.} those with no parabolic reduction. A nilpotent covector supported on an irreducible local systems must vanish \cite[Lem. 13.3.4]{arinkin&gaitsgory} and so the pullback of $\Nn_{G}$ to $\Hodge_G^{\irred}$ coincides with the zero section $\{ 0 \} \subset T^{*}\Hodge_G$. By general properties of singular support described in \cite[Cory. 8.2.10(b)]{arinkin&gaitsgory} we obtain the exact sequence of dg categories 
\begin{equation}
\label{Eis and irred}
\begin{tikzcd}
\Coh_{\Nn_G^{\red}} (\Hodge_G) 
\arrow[d, equal]
\arrow[r,shift left=2pt] 
& \Coh_{\Nn_G}( \Hodge_{G} ) 
\arrow[l,shift left=2pt]
\arrow[r,shift left=2pt] 
\arrow[d, equal]
& \Coh_{\{0\}}( \Hodge_{G}^{\irred} )
\arrow[l,shift left=2pt] 
\arrow[d, equal]
\\
\Coh_{\Nn_G}( \Hodge_{G} )_{\Eis}
\arrow[r,shift left=2pt] 
& \Coh_{\Nn_G}( \Hodge_{G} )
\arrow[l,shift left=2pt]
\arrow[r,shift left=2pt] 
& \Perf( \Hodge_{G}^{\irred} )
\arrow[l,shift left=2pt]
\end{tikzcd}. 
\end{equation}
Running the same argument in the Dolbeault and twistor cases yields the following. 

\begin{corollary}
There exists a non-abelian Hodge family of cuspidal-Eisenstein decompositions for coherent nilpotent sheaves, determined by the exact sequences of dg categories 
\[
\begin{tikzcd}
\Coh_{\Nn_G}( \Higgs_{G} )_{\Eis}
\arrow[r,shift left=2pt] 
& \Coh_{\Nn_G}( \Higgs_{G} )
\arrow[l,shift left=2pt]
\arrow[r,shift left=2pt] 
& \Perf( \Higgs_{G}^{\irred} )
\arrow[l,shift left=2pt] , 
\end{tikzcd} , 
\]
\[
\begin{tikzcd}
\Coh_{\Nn_G}( \Hodge_{G} )_{\Eis}
\arrow[r,shift left=2pt] 
& \Coh_{\Nn_G}( \Hodge_{G} )
\arrow[l,shift left=2pt]
\arrow[r,shift left=2pt] 
& \Perf( \Hodge_{G}^{\irred} )
\arrow[l,shift left=2pt]  
\end{tikzcd} , 
\]
\[
\begin{tikzcd}
\Coh^{\GAGA}_{\Nn_G}( \Twistor_{G} )_{\Eis}
\arrow[r,shift left=2pt] 
& \Coh^{\GAGA}_{\Nn_G}( \Twistor_{G} )
\arrow[l,shift left=2pt]
\arrow[r,shift left=2pt] 
& \Perf^{\GAGA}( \Twistor_{G}^{\irred} )
\arrow[l,shift left=2pt]  
\end{tikzcd} , 
\]
\[
\begin{tikzcd}
\Coh^{\BBB}_{\Nn_G}( \Twistor_{G} )_{\Eis}
\arrow[r,shift left=2pt] 
& \Coh^{\BBB}_{\Nn_G}( \Twistor_{G} )
\arrow[l,shift left=2pt]
\arrow[r,shift left=2pt] 
& \Perf^{\BBB}( \Twistor_{G}^{\irred} )
\arrow[l,shift left=2pt]  
\end{tikzcd} . 
\]
\end{corollary}

\subsection{Decomposition of Dirac--Higgs complexes}
\label{Dirac--Higgs decompositions}
This section describes the cusipidal-Eisenstein decompositions of Theorem \ref{decompose BBB} on a particular family of objects: the \textit{Dirac--Higgs BBB-branes} $\DD_{\rho} \in \Perf^{\BBB}(\Twistor_G)$ parametrised by representations $\rho : G \to GL_n$. We show that the decompositions of $\DD_{\rho}$ are controlled by the adjoint pair of parabolic induction and reduction functors 
\begin{equation}
\begin{tikzcd}
\Rep(M) 
\arrow[r,shift left=2pt, "\Ind^{G}_M"]
& \Rep(G)
\arrow[l , shift left=2pt, "\Res^{G}_M"] 
\end{tikzcd} . 
\end{equation} 
\subsubsection{Construction} The BBB-branes $\DD_{\rho}$ originate from gauge theory constructions of Hitchin \cite{hitchin_dirac}, who defined a hyperholomorphic bundle, known as the \textit{Dirac-Higgs bundle}, as the null space of a Dirac operator on the infinite dimensional affine space of Higgs fields. In his thesis, Hausel \cite{hausel_thesis} provided an algebraic construction, defined \'etale-locally by the pushforward of the local universal Higgs bundle to the Dolbeault moduli spaces. Franco and the author \cite{FH1} studied a natural extension of Hausel's construction, defined by the following universal objects. They construct a \textit{Dirac--Higgs functor} 
\begin{equation}
\label{Dirac-Higgs functor}
\DD : \Rep(G) \to \Perf^{\BBB}(\Twistor_G) ,  
\end{equation}
defined as follows. Let $\Uu^{\an}_{\Hod, G} \to X_{\Hod}^{\an} \times \Hodge_G^{\an}$ denote the universal family for $\Hodge_G^{\an}$ and let $p_2 : X^{\an}_{\Hod} \times \Hodge_G^{\an} \to \Hodge_G^{\an}$ denote the canonical projection. Given a representation $\rho : G \to GL_n$, the associated \textit{Dirac--Higgs complex} $\DD_{\rho} \in \Coh^{\BBB}(\Twistor_G)$ is constructed by first considering the underlying Hodge component 
\[
\DD^{\Hod}_{\rho} = p_{2, *} \rho(\Uu^{\an}_{\Hod, G}) \in \Perf(\Hodge^{\an}_G) . 
\]
The conjugate Hodge component $\ol{\DD}^{\Hod}_{\rho} \in \Perf(\ol{\Hodge}_G)$ is defined symmetrically. The complex $\DD_{\rho} \in \Perf^{\BBB}(\Twistor_G)$ is then defined to be the gluing of $\DD^{\Hod}_{\rho}$ and $\ol{\DD}^{\Hod}_{\rho}$.

\subsubsection{Decomposition} Consider a collection $\{ \rho_{M} \in \Rep(M) \}_{P}$ parametrised by parabolics $P \subset G$. Let $\rho \in \Rep(G)$ be the representation generated by $\Ind^{G}_M : \Rep(M) \to \Rep(G)$ acting on the collection. The Dirac--Higgs construction yeilds the complexes
\[
\DD_{\rho}^{\cusp} := \DD_{\rho} |_{\Twistor_G^{\irred}} , 
\]
\[
\DD_{\rho}^{\Eis} := 
\big\langle 
\Eis^{\Tw}_{P}
\big( \DD_{\rho_M}|_{\Twistor_M} \big) 
\big\rangle_{P} = \big\langle 
\Eis^{\Tw}_{P, \nu}
\big( \DD_{\rho_M}|_{\Twistor^{\sst, \nu}_M} \big) 
\big\rangle_{(P, \nu)}, 
\]
where the equation for $\DD_{\rho}^{\Eis}$ describes the complex generated by $\Eis^{\Tw}_{P}$, or equivalently $\Eis^{\Tw}_{P, \nu}$, evaluated on the displayed complexes.   

\begin{proposition}
\label{Dirac-Higgs decompositions}
In the above notation, the Dirac--Higgs complex $\DD_{\rho} \in \Perf^{\BBB}(\Twistor_G)$ is generated by the cuspidal component $\DD_{\rho}^{\cusp}$ and the Eisenstein component $\DD_{\rho}^{\Eis}$. 
\end{proposition}

\begin{remark} 
In the adjoint representation, the Dirac--Higgs construction yields the tangent complexes on $\Twistor_G$, and the consequential fact that $\TT_{\Twistor_G}$ is an object of $\Perf^{\BBB}(\Twistor_G)$ is analogous to the fact that the tangent bundle of a hyperk\"ahler manifold is naturally hyperholomorphic. Proposition \ref{Dirac-Higgs decompositions intro} in this case expresses $\TT_{\Twistor_G}$ as generated by the tangent complexes $\TT_{\Twistor^{\sst, \nu}_M}$, which reduces to the fact that $\Res_M^{G}$ applied to the adjoint representation of $G$ coincides with the adjoint representation of $M$.
\end{remark}

Proposition \ref{Dirac-Higgs decompositions} follows directly from the following calculation.  

\begin{lemma}
Given a pair $(P, \nu)$ where $P$ has Levi quotient $P \to M$, the following diagram is commutative: 
\[
\begin{tikzcd}
\Rep(M) \arrow[r, "\Ind_{M}^{G}"] \arrow[d, "\DD"] 
& \Rep(G) \arrow[d, "\DD"] 
\\
\Perf^{\BBB}(\Twistor_M) \arrow[dd, "(\cdot)|_{\Twistor^{\sst, \nu}_M}"] 
& \Perf^{\BBB}(\Twistor_G) \arrow[d, "(\cdot)|_{\Twistor_G^{P, \nu}}"] 
\\
& \Perf^{\BBB}(\Twistor_G^{P, \nu}) \arrow[d, hook]
\\
\Perf^{\BBB}(\Twistor^{\sst, \nu}_M) \arrow[r, "\Eis^{\Tw}_{P, \nu}"] & \Coh_{\Nn}^{\BBB}(\Twistor_G)
\end{tikzcd} .
\]
\end{lemma}

\begin{proof} 
It suffices to establish commutivity over the substack 
\[
\Hodge^{\an}_G  = \Twistor_G \times_{\PP^1} \AA^1 ,
\]
for it then also holds symmetrically over $\ol{\Hodge}^{\an}_G \subset \Twistor_G$ and extends to $\Twistor_G$ via a gluing argument similar to Lemma \ref{concluding decompositions}(2). 

The calculation is identical in the analytic or algebraic topology. We proceed with the algebraic stack $\Hodge_G$. The analytic formula can then be obtained by analytification. 

Fix $\rho_M \in \Rep(G)$. On $\Hodge_M$ the Dirac--Higgs complex is defined by the universal expression $\DD^{\Hod}_{\rho} = p_{2, *}\rho_M(\Uu_{\Hod, M})$. Our calculations rely on the universal equivalence
\begin{equation}
\label{Hodge universal equivalence}
(\id \times q^{\Hod})^{!}\rho_M(\Uu_{\Hod, M}) \cong (\id \times p^{\Hod})^{!} \Ind^{G}_{M}\rho_M(\Uu_{\Hod, G}) ,
\end{equation}
which exists due to the following observation: the left hand side is universal for objects obtained by inducing along 
\[
P \to M \xrightarrow{\rho_M} GL_n , 
\]
and the right hand side is universal for inductions along 
\[
P \to G \xrightarrow{\Ind^{G}_{M}\rho_M} GL_n.
\]
The universal isomorphism \eqref{Hodge universal equivalence} is therefore a consequence of the universal property alongside the fact that these two compositions agree.  

Recall that $\Hodge_G^{P, \nu}$ is the image of $p^{\Hod} : \Hodge_P^{\sst, \nu} \to \Hodge_G^{P, \nu}$. By adjunction we have  
\begin{equation}
\label{univ. equiv}
(p^{\Hod})_{*}(p^{\Hod})^{!}\DD^{\Hod}_{\rho}|_{\Hodge_G^{P, \nu}} \cong \DD^{\Hod}_{\rho}|_{\Hodge_G^{P, \nu}} . 
\end{equation}
By applying base change and functoriality around the Cartesian squares
\[
\begin{tikzcd}
X_{\Hod} \times \Hodge_M^{\sst, \nu} \arrow[d, "p_2"] 
& X_{\Hod} \times \Hodge_P^{\sst, \nu}
\arrow[d, "p_2"] \arrow[l, "\id \times q^{\Hod}"'] \arrow[r, "\id \times p^{\Hod}"]
& X_{\Hod} \times \Hodge_G^{P, \nu}
\arrow[d, "p_2"]
\\
\Hodge_M^{\sst, \nu}
& \Hodge_P^{\sst, \nu} \arrow[l, "q^{\Hod}"'] \arrow[r, "p^{\Hod}"]
& \Hodge_G^{P, \nu}
\end{tikzcd} ,
\]
alongside \eqref{Hodge universal equivalence} and \eqref{univ. equiv}, we are able to conclude the proof with the calculations 
\begin{align*}
\Eis_{P, \nu}^{\Hod}\big(\DD^{\Hod}_{\rho_M}|_{\Hodge_M^{\sst, \nu}}\big) 
& = (p^{\Hod})_{*} (q^{\Hod})^{!}p_{2, *} \rho_M
\big(\Uu_{\Hod, M}|_{X_{\Hod} \times \Hodge_M^{\sst, \nu}} \big) , \\
& \cong (p^{\Hod})_{*} p_{2, *} (\id \times q^{\Hod})^{!} \rho_M(\Uu_{\Hod, M}|_{X_{\Hod} \times \Hodge_M^{\sst, \nu}}) , \\
& \cong p_{2, *} (\id \times p^{\Hod})_{*} (\id \times p^{\Hod})^{*} \Ind^{G}_{M}\rho_M(\Uu_{\Hod, G}^{\an}|_{X_{\Hod} \times \Hodge_G^{P, \nu}}) , \\
& \cong p_{2, *} \big( \Ind^{G}_{M}\rho_M(\Uu_{\Hod, G}^{\an}) |_{X_{\Hod} \times \Hodge_G^{P, \nu}} \big) , \\ 
& \cong \DD^{\Hod}_{\Ind^{G}_{M}\rho_M} |_{\Hodge_G^{P, \nu}} . \qedhere
\end{align*}
\end{proof}

\section{Applications of parabolic induction II: compatibility with Wilson operators}
\label{se: Wilson}

We conclude this article by computing a relation between the twistor Eisenstein functors $\Eis^{\Tw}_P$ and certain \textit{twistor Wilson operators}, thus comparing two fundamental sources of symmetry experienced by the spectral categories $\IndCoh_{\Nn}^{\BBB}(\Twistor_G)$. Over fixed points of the base curve $X$, we consider a type of Wilson operator defined by twistoral gluing construction, where one glues complex conjugate copies of the classical limit Wilson operators constructed by Donagi and Pantev \cite{donagi&pantev}. 

\subsection{Construction of Wilson operators} 
\label{se: Wilson operators}

\subsubsection{Hodge Wilson operators} First we specify some universal objects. Once more let $\Uu^{\Hod}_G$ denote the universal family on $X_{\Hod} \times \Hodge_G$. Then $\Uu^{\Hod}_G$ is a $G$-bundle on $X_{\Hod} \times \Hodge_G$, or equivalently, a $\lambda$-connection on $X \times \Hodge_G$. Let $\Pp^{\Hod}_G$ denote the underlying $G$-bundle of $\Uu^{\Hod}_G$ on $X \times \Hodge_G$, which can be constructed as the pullback of the universal $G$-bundle on $X \times \Bun_G$ along the forgetful map $X \times \Hodge_G \to X \times \Bun_G$.   

Fix a representation $\rho : G \to GL_n$ and a geometric point $x \in X$. One has the associated restriction $\Pp^{\Hod}_{G, x} := \Uu_G^{\Hod}|_{\{x\} \times \Hodge_G}$ to a $G$-bundle on $\Hodge_G$ and the associated vector bundle $\rho(\Pp^{\Hod}_{G, x})$. We then define the \textit{Hodge Wilson operators} by the tensoral action 
\begin{equation}
\label{Wilson}
\WW^{\Hod}_{\rho, x} : \IndCoh(\Hodge_G) \to \IndCoh(\Hodge_G) ,
\end{equation}
\[
\Ff \longmapsto \Ff \otimes \rho(\Pp^{\Hod}_{G, x}) . 
\]
In \cite{donagi&pantev}, the functor $\WW^{\Hod}_{\rho, x}$ is interpreted as a Hodge deformation family, interpolating between the de Rham Wilson functor $\WW^{\dR}_{\rho, x} = \WW^{\Hod}_{\rho, x} \otimes_{\AA^1} \{1\}$ and the Dolbeault Wilson functor $\WW^{\Dol}_{\rho, x} = \WW^{\Hod}_{\rho, x} \otimes_{\AA^1} \{0\}$. In other words, $\WW^{\Hod}_{\rho, x}$ is used to express the restriction $\WW^{\Dol}_{\rho, x}$ as the \textit{classical limit} of $\WW^{\dR}_{\rho, x}$.  

\subsubsection{Twistor Wilson operators} To extend the construction of $\WW^{\Hod}_{\rho, x}$ to the full twistor $\PP^1$ we perform a twistoral gluing. Consider 
\[
\WW^{\Hod, \an}_{\rho, x} := (\cdot) \otimes \rho(\Pp^{\Hod, \an}_{G, x})  : \IndCoh(\Hodge^{\an}_G) \to \IndCoh(\Hodge^{\an}_G) ,
\]
\[
\ol{\WW}^{\Hod, \an}_{\rho, \ol{x}} := (\cdot) \otimes \rho(\Pp^{\Hod, \an}_{G, x}) : \IndCoh(\ol{\Hodge}^{\an}_G) \to \IndCoh(\ol{\Hodge}^{\an}_G) , 
\]
where $\WW^{\Hod, \an}_{\rho, x}$ denotes the same construction as in \eqref{Wilson}, performed in the analytic topology with the analytic $G$-bundle $\Pp^{\Hod, \an}_{G, x} \to \Hodge_G^{\an}$. Moreover, $\ol{\WW}^{\Hod, \an}_{\rho, \ol{x}}$ denotes the same construction over $\ol{X}^{\an}$, defined at a geometric point $\ol{x} \in \ol{X}^{\an}$. 

\begin{proposition} 
\label{twistor wilson exists}
Fix points $x \in X^{\an}$ and $\ol{x} \in \ol{X}^{\an}$ that represent the same point on the underlying topological space $X^{top} = \ol{X}^{top}$\footnote{In other words, $x$ and $\ol{x}$ represent the same geometric point on the Betti shape $(X^{\an})_{\B} = (\ol{X}^{\an})_{\B}$.}. Then, the associated twistor Wilson operator
\[
\WW^{\Tw}_{\rho, x, \ol{x}} := \WW^{\Hod, \an}_{\rho, x} \times_{\Betti} \ol{\WW}^{\Hod, \an}_{\rho, \ol{x}} : \IndCoh(\Twistor_G) \to \IndCoh(\Twistor_G) , 
\]
is well-defined. Moreover, $\rho(\Pp^{\Hod}_{G, x})$ and $\rho(\ol{\Pp}^{\Hod}_{G, \ol{x}})$ glue to define a locally free sheaf $\rho(\Pp^{\Tw}_{G, x , \ol{x}})$ on $\Twistor_G$ such that the twistor Wilson operator can be represented as a tensor action 
\begin{equation}
\label{represent W^Tw}
\WW^{\Tw}_{\rho, x, \ol{x}} \simeq (\cdot) \otimes \rho(\Pp^{\Tw}_{G, x, \ol{x}}) . 
\end{equation}
\end{proposition}
\begin{proof}
We recall the defining pushout $\Twistor_G = \Hodge_G^{\an} \sqcup_{\Betti} \ol{\Hodge}_G^{\an}$ induces the pullback 
\begin{equation*}
\label{pullback nilps}
\begin{tikzcd}
\IndCoh(\Twistor_G) \arrow[d] \arrow[r] & \IndCoh(\ol{\Hodge}^{\an}_G) \arrow[d, "\ol{\RH}_{*}"] \\
\IndCoh(\Hodge^{\an}_G) \arrow[r, "\RH_{*}"] & \IndCoh(\Rep^{\an}_G \times \GG_m)
\end{tikzcd} , 
\end{equation*}
where $\RH$ denotes the Riemannn--Hilbert correspondence acting as 
\[
\Hodge^{\an}_G \times_{\AA^1} \GG_m 
\xrightarrow{\cong} \LocSys^{\an}_G \times \GG_m 
\xrightarrow{\cong} \Rep^{\an}_G \times \GG_m . 
\]
To show that $\WW^{\Hod, \an}_{\rho, x} \times_{\Betti} \ol{\WW}^{\Hod, \an}_{\rho, \ol{x}}$ is well-defined is to checking the following: a pair $(\Ff_{\Hod}, \ol{\Ff}_{\Hod})$ of sheaves from $\IndCoh(\Hodge^{\an}_G)$ and $\IndCoh(\ol{\Hodge}^{\an}_G)$ equipped with an isomorphism 
\begin{equation}
\label{glue data}
\RH_{*} \big( \Ff_{\Hod}|_{\Hodge^{\an}_G \times_{\AA^1} \GG_m} \big) \cong 
\ol{\RH}_{*} \big(\ol{\Ff}_{\Hod}|_{\ol{\Hodge}^{\an}_G \times_{\AA^1} \GG_m} \big) ,
\end{equation}
induces an isomorphism after evaluation of the Wilson operators: 
\begin{equation}
\label{pullout condition}
\RH_{*} \big( \WW^{\Hod}_{\rho, x}(\Ff)|_{\Hodge^{\an}_G \times_{\AA^1} \GG_m} \big)
\cong 
\ol{\RH}_{*} \big( \ol{\WW}^{\Hod}_{\rho, x}(\ol{\Ff})|_{\ol{\Hodge}^{\an}_G \times_{\AA^1} \GG_m} \big) . 
\end{equation}
To see this, first note the Riemann--Hilbert correspondence $\RH : \LocSys_G^{\an} \xrightarrow{\cong} \Rep_G^{\an}$ induces a universal equivalence on the universal $G$-bundles:  
\begin{equation}
\label{eq: univ Pp equiv}
\RH_{*}\rho(\Pp^{\Hod, \an}_{G, x} |_{\Hodge^{\an}_G \times_{\AA^1} \GG_m}) 
\cong \ol{\RH}_{*}\rho(\ol{\Pp}^{\Hod, \an}_{G, \ol{x}} |_{\ol{\Hodge}^{\an}_G \times_{\AA^1} \GG_m}) . 
\end{equation}
Then \eqref{pullout condition} can be written as a twist of \eqref{glue data} by the universal equivalence \eqref{eq: univ Pp equiv}. This proves the first part regarding existence of $\WW^{\Tw}_{\rho, x \ol{x}}$. For the moreover part regarding the tensor presentation, it suffices to check that $\rho(\Pp^{\Hod}_{x})$ and $\rho(\ol{\Pp}^{\Hod}_{\ol{x}})$ glue to define a locally free sheaf on $\Twistor_G$. The equivalence \eqref{represent W^Tw} follows immediately from the gluing. The required gluing condition is precisely \eqref{eq: univ Pp equiv}. 
\end{proof}

\subsubsection{Preservation properties} We check that $\WW^{\Tw}_{\rho, x, \ol{x}}$ preserves various sheaf theories on $\Twistor_G$.  

\begin{proposition}
\label{Wilson acts nicely}
The functor $\WW^{\Tw}_{\rho, x, \ol{x}}$ preserves 
\begin{enumerate}
    \item the categories $\IndCoh^{\GAGA}(\Twistor_G)$ of GAGA sheaves defined in \ref{GAGA sheaves Tw_G},
    \item the categories $\IndCoh_{\Nn}(\Twistor_G)$ of sheaves with nilpotent singular support defined in \ref{nilpotent sheaves}, 
    \item the categories $\IndCoh^{\BBB}(\Twistor_G)$ of BBB-branes defined in \ref{BBB-branes}, 
\end{enumerate}
and so (1), (2) and (3) collectively show that $\WW^{\Tw}_{\rho, x, \ol{x}}$ defines a functor 
\[
\WW^{\Tw}_{\rho, x, \ol{x}} : \IndCoh^{\BBB}_{\Nn}(\Twistor_G) \to \IndCoh^{\BBB}_{\Nn}(\Twistor_G) . 
\] 
\end{proposition}

\begin{proof}
For (1), we recall that $\IndCoh^{\GAGA}(\Twistor_G)$ is defined by a pullback 
\[
\IndCoh^{\GAGA}(\Twistor_G) = \IndCoh^{\GAGA}(\Hodge^{\an}_G) \times_{\Betti} \IndCoh^{\GAGA}(\ol{\Hodge}^{\an}_G).
\]
It suffices to show that $\WW^{\Hod, \an}_{\rho, x}$ preserves $\IndCoh^{\GAGA}(\Hodge^{\an}_G)$, for then symmetrically so does $\ol{\WW}^{\Hod}_{\rho, \ol{x}}$. This is equivalent to checking that, when acting on $\IndCoh^{\GAGA}(\Hodge^{\an}_G)$, we have an equivalence of functors  
\[
\WW^{\Hod, \an}_{\rho, x} \simeq (\WW^{\Hod}_{\rho, x})^{\an} . 
\]
By the universal GAGA property $\Maps(X_{\Hod}, BG) \cong \anMaps(X_{\Hod}^{\an}, BG^{\an})$, the analytic universal $G$-bundle $\Pp^{\Hod, \an}_{G} \to X^{\an} \times \Hodge_G^{\an}$ is precisely the analytification of the algebraic universal $G$-bundle $\Pp^{\Hod}_G \to X \times \Hodge_G$. This allows us to calculate the action of the Wilson operators as follows. For an object $\Ff^{\an} \in \IndCoh^{\GAGA}(\Hodge_G^{\an})$, consider the equivalence 
\[
\WW_{\rho, x}^{\Hod, \an}(\Ff^{\an}) 
= \Ff^{\an} \otimes \rho(\Pp^{\Hod, \an}_{G, x})
\cong (\Ff \otimes \rho(\Pp^{\Hod}_{G, x}))^{\an}  = (\WW_{\rho, x}^{\Hod}(\Ff))^{\an}  , 
\]
and so indeed $\WW_{\rho, x}^{\Hod}(\Ff^{\an})$ lands in $\IndCoh^{\GAGA}(\Twistor_G)$. This proves (1). 

For (2), observe that, since $\rho(\Pp^{\Tw}_{G, x, \ol{x}})$ is locally free, its tensor action preserves $\IndCoh_{\Nn}(\Twistor_G)$. 

For (3), consider an object $\Bb \in \IndCoh^{\BBB}(\Twistor_G)$ and fix a test object $\Ww = \sigma_{*} \Oo_{\PP^1} \in \Wils^{\Tw}_G$ supported on $(\sigma : \PP^1 \to \Twistor_G) \in \Horiz_G$. By the BBB-brane hypothesis on $\Bb$, there exists a graded vector space $V^{\bullet}_{\Ww} \in \QCoh(pt)$ and an equivalence 
\[
\Bb \otimes \Ww \cong V^{\bullet}_{\Ww} \otimes_{\CC} \Ww.
\]
To show that $\WW^{\Tw}_{\rho, x, \ol{x}}(\Bb)$ is an object of $\IndCoh^{\BBB}(\Twistor_G)$ is to provide a new vector space $\widetilde{V}^{\bullet}_{\Ww} \in \QCoh(pt)$ and an equivalence 
\begin{equation}
\label{test W is BBB}
\WW^{\Tw}_{\rho, x}(\Bb) \otimes \Ww \cong \widetilde{V}^{\bullet}_{\Ww} \otimes_{\CC} \Ww.
\end{equation}
Let us expand the terms on the left hand side by considering the expression 
\begin{equation}
\label{W^Tw on B}
\WW^{\Tw}_{\rho, x}(\Bb) \otimes \Ww  = \Bb \otimes \rho(\Pp_{G, x, \ol{x}}^{\Tw}) \otimes \Ww .
\end{equation}
To $\sigma$ and $\Ww$ we consider the Hodge restrictions $\sigma_{\Hod} : \AA^1 \to \Hodge^{\an}_G$ and $\Ww_{\Hod} = (\sigma_{\Hod})_{*}\Oo_{\AA^1}$. The map $\sigma_{\Hod}$ represents an $\AA^1$-family of $\lambda$-connections $(E, \nabla^{\lambda}) \to X^{\an}$. By the universal property of $\Pp^{\Hod, \an}_G$, the restriction to $\id \times \sigma_{\Hod} : X^{\an} \times \AA^1 \to X^{\an} \times \Hodge_G$ represents the underlying $G$-bundle $E$, and so $(\sigma_{\Hod})^{*}\rho(\Pp^{\Hod, \an}_G)$ represents the vector bundle $\rho(E) \to X^{\an}$. Moreover $(\sigma_{\Hod})^{*}\rho(\Pp^{\Hod, \an}_{G, x})$ represents the vector space $\rho(E)|_{x}$. Therefore we obtain an equivalence 
\begin{equation}
\label{universal restriction thing}
\rho(\Pp^{\Hod, \an}_{G, x}) \otimes \Ww_{\Hod} \cong \rho(E|_{x}) \otimes_{\CC} \Ww_{\Hod} , 
\end{equation}
in $\IndCoh(\Hodge_G^{\an})$.  

The complex conjugate case is identical: one has a section $\ol{\sigma}_{\Hod} : \AA^1 \to \ol{\Hodge}_G^{\an}$ that represents an $\AA^1$-family $(\ol{E}, \ol{\nabla}^{\lambda})$ of $\lambda$-connections on $\ol{X}^{\an}$. By the same reasoning, the sheaf $\ol{\Ww}_{\Hod} = (\ol{\sigma}_{\Hod})_{*}\Oo_{\AA^1}$ satisfies the relation 
\begin{equation}
\label{universal restriction thing 2}
\rho(\ol{\Pp}^{\Hod, \an}_{G, x}) \otimes \ol{\Ww}_{\Hod} \cong \rho(\ol{E}|_{\ol{x}}) \otimes_{\CC} \ol{\Ww}_{\Hod} , 
\end{equation}
in $\IndCoh(\ol{\Hodge}_G^{\an})$, where $\rho(E|_{x}) \cong \rho(\ol{E}|_{\ol{x}})$. Denote this vector space by $V$. Then \eqref{universal restriction thing} and \eqref{universal restriction thing 2} glue to define an equivalence 
\[
\rho(\Pp^{\Tw}_{G, x, \ol{x}}) \otimes \Ww \cong V \otimes_{\CC} \Ww, 
\]
in $\IndCoh(\Twistor_G)$. Substituting this into \eqref{W^Tw on B} provides 
\[
\WW^{\Tw}_{\rho, x}(\Bb) \otimes \Ww  = \Bb \otimes \rho(P)|_x \otimes_{\CC} \Ww \cong V_{\Ww}^{\bullet} \otimes_{\CC} V \otimes_{\CC} \Ww . 
\]
The choice $\widetilde{V}^{\bullet}_{\Ww} := V_{\Ww}^{\bullet} \otimes_{\CC} V \in \QCoh(pt)$ allows us to conclude that \eqref{test W is BBB} holds and so indeed $\WW^{\Tw}_{\rho, x, \ol{x}}(\Bb)$ is an object of $\IndCoh^{\BBB}(\Twistor_G)$. This proves (3) and concludes the proof. 
\end{proof}

\subsection{Intertwining properties} 
\label{se: intertwining}
We now present our final result. We show that twistor Wilson operators commute with twistor Eisenstein functors. 

\begin{theorem}
\label{Wilson and Eis}
Let $P \subset G$ be a parabolic with Levi quotient $M$. Fix a representation $\rho_M : M \to GL_n$ with parabolic induction $\rho = \Ind^{G}_{M}(\rho_M) : G \to GL_n$. Then there exists a commutative square 
\[
\begin{tikzcd}
\IndCoh(\Twistor_M) \arrow[d, "\WW^{\Tw}_{\rho_M, x, \ol{x}}"] \arrow[r, "\Eis_{P}^{\Tw}"] & \IndCoh(\Twistor_G) \arrow[d, "\WW^{\Tw}_{\rho, x, \ol{x}}"] \\
\IndCoh(\Twistor_M) \arrow[r, "\Eis_{P}^{\Tw}"] & \IndCoh(\Twistor_G)
\end{tikzcd}, 
\]
which restricts to the commutative square on the nilpotent BBB-brane categories 
\[
\begin{tikzcd}
\IndCoh^{\BBB}_{\Nn}(\Twistor_M) \arrow[d, "\WW^{\Tw}_{\rho_M, x, \ol{x}}"] \arrow[r, "\Eis_{P}^{\Tw}"] & \IndCoh^{\BBB}_{\Nn}(\Twistor_G) \arrow[d, "\WW^{\Tw}_{\rho, x, \ol{x}}"] \\
\IndCoh^{\BBB}_{\Nn}(\Twistor_M) \arrow[r, "\Eis_{P}^{\Tw}"] & \IndCoh^{\BBB}_{\Nn}(\Twistor_G)
\end{tikzcd} .
\]
\end{theorem}

\begin{remark}
In the de Rham theory, commutation between Wilson operators and spectral Eisenstein functors is self-evident and computed along similar means. On the Langlands dual side, commutation between Hecke operators and (compactified) automorphic Eisenstein functors over $\Bun_{G}$ is highly non-trivial and is one of the main results of Braverman--Gaitsgory \cite{braverman&gaitsgory}. 
\end{remark}

\begin{proof} All terms in the diagram can be presented as fiber products of Hodge components. In particular $\Eis^{\Tw}_{P} = \Eis^{\Hod}_{P} \times_{\Betti} \ol{\Eis}^{\Hod}_{P}$ by Proposition \ref{spec tw eis} and $\WW^{\Tw}_{\rho, x , \ol{x}} = \WW^{\Hod}_{\rho, x} \times_{\Betti} \ol{\WW}^{\Hod}_{\rho, \ol{x}} $ by Proposition \ref{twistor wilson exists}. Thus it suffices to check the commutativity of the diagram 
\[
\begin{tikzcd}
\IndCoh(\Hodge_M) \arrow[d, "\WW^{\Hod}_{\rho_M, x}"] \arrow[r, "\Eis_{P}^{\Hod}"] 
& \IndCoh(\Hodge_G) \arrow[d, "\WW^{\Hod}_{\rho, x}"] 
\\
\IndCoh(\Hodge_M) \arrow[r, "\Eis_{P}^{\Hod}"] 
& \IndCoh(\Hodge_G)
\end{tikzcd}, 
\]
Recall the notation $\Hodge_M \xleftarrow{q^{\Hod}} \Hodge_P \xrightarrow{p^{\Hod}} \Hodge_G$ for the extension of structure group morphisms in the Hodge context. The essential content of the proof is the existence of the isomorphism 
\[
(\id \times q^{\Hod})^{!}\rho_M(\Uu^{\Hod}_{M}) \cong (\id \times p^{\Hod})^{!} \rho (\Uu^{\Hod}_{G}), 
\]
from \eqref{Hodge universal equivalence}. Restricting to the underlying $G$-bundle and a point $x \in X$ yields the natural relation  
\[
(p^{\Hod})^{!}\rho_M(\Pp^{\Hod}_{M, x}) \cong (q^{\Hod})^{!} \rho (\Pp^{\Hod}_{G, x}) , 
\]
which alongside the projection formula allows us to conclude with the isomorphisms 
\begin{align*}
\Eis_{P}^{\Hod} \circ \WW^{\Hod}_{\rho_M, x}(\Ff) 
& = (p^{\Hod})_{*} \big( (q^{\Hod})^{!} \Ff \otimes (q^{\Hod})^{!} \rho_M( \Uu^{\Hod}_{M, x} ) \big) , \\
& \cong (p^{\Hod})_{*} \big( (q^{\Hod})^{!} \Ff \otimes (p^{\Hod})^{!} \rho( \Uu^{\Hod}_{G, x} ) \big) , \\
& \cong (p^{\Hod})_{*} (q^{\Hod})^{!} \Ff \otimes \rho( \Uu^{\Hod}_{G, x} ) = \WW^{\Hod}_{\rho, x} \circ \Eis_{P}^{\Hod} (\Ff) . 
\end{align*} 
The existence of the restricted commutative square is a consequence of the fact that $\Eis^{\Tw}_P$ and $\WW^{\Tw}_{\rho}$ both preserve BBB-brane and nilpotent singular support conditions, as proven in Corollary \ref{BBB nilp Eis} and Proposition \ref{Wilson acts nicely} respectively. 
\end{proof} 


\renewbibmacro{in:}{}

\printbibliography

\end{document}